\documentclass{amsart}

\usepackage[left=1.25in, right=1.25in, top=1.5in, bottom=1.5in]{geometry}
\usepackage{amsmath, amssymb, amsthm, tikz-cd}
\usepackage{mathrsfs}
\usetikzlibrary{fit}
\usepackage{mathtools}

\usepackage[colorlinks=true, linkcolor=black, citecolor=black, urlcolor=black]{hyperref}
\usepackage{cite}

\newcommand{\Dom}{\operatorname{Dom}}
\newcommand{\lgth}{\operatorname{lgth}}
\newcommand{\sigmaArt}{\sigma_{\mathrm{Art}}}
\newcommand{\Domgpd}{\mathfrak{Dom}}

\newtheorem{theorem}{Theorem}[section]
\newtheorem{proposition}[theorem]{Proposition}
\newtheorem{lemma}[theorem]{Lemma}
\newtheorem{corollary}[theorem]{Corollary}
\newtheorem{construction}[theorem]{Construction}

\newtheorem{thmletter}{Theorem}

\makeatletter
\let\c@equation\c@theorem

\makeatother

\theoremstyle{definition}
\newtheorem{definition}[theorem]{Definition}
\newtheorem{example}[theorem]{Example}

\theoremstyle{remark}
\newtheorem{remark}[theorem]{Remark}
\newtheorem*{remark*}{Remark}

\begin{document}

\title{Higher dimensional dominoes in de Rham--Witt cohomology}
\author{Yuanning Zhang}
\address{Department of Mathematics, Northwestern University}
\email{yuanningzhang2026@u.northwestern.edu}
\date{July 28, 2026}

\begin{abstract}
The de Rham--Witt cohomology of a smooth proper variety in characteristic \(p\)
contains a canonical piece called the \emph{domino}, which is not finitely generated
over the Witt vectors and carries the nonzero differentials of the slope spectral
sequence. Beyond dimension one, dominoes were unclassified. We
classify the two-dimensional ones and put a domino of any dimension into a normal
form. To each domino we attach two unipotent groups,
one formal and one perfect, and prove that their \emph{isogeny partitions} agree.
In degree two we recover the domino
of a Mazur--Ogus variety from its crystalline cohomology, compute it for two
families of supersingular abelian varieties, and bound the exponent of the
\(p\)-primary Brauer group in terms of the \(a\)-number, for every prime \(p\). This
answers a question of Grammatica--Skorobogatov--Yang.
\end{abstract}

\maketitle

\tableofcontents

\section{Introduction}

Over an algebraically closed field $k$ of characteristic $p>0$, a supersingular
K3 surface or a supersingular abelian surface has Brauer group $k$ and formal
Brauer group $\widehat{\mathbb G}_a$. Both invariants are induced by a single
\emph{domino} $U$. This is the part of the de Rham--Witt cohomology that carries
the nonvanishing differential $E_1^{0,2}\to E_1^{1,2}$ of the slope spectral
sequence. The formal Brauer group is induced by its degree-$0$ part $U^0$, and
the $p$-primary Brauer group by its degree-$1$ part $U^1$.

A domino has two terms
$U^0\xrightarrow{d}U^1$ and is killed by a power of $p$, but it is not finitely
generated over $W=W(k)$. Nevertheless, it does have a notion of dimension
$\dim(U)$ (Definition~\ref{def:domino-numerics}), and the one-dimensional
dominoes, also known as the elementary dominoes, are completely classified and
denoted by $U_j$ for some integer $j\in\mathbb Z$. In the two examples above the
domino is an elementary domino $U_j$, where $j$ is known as the \emph{Artin
invariant}. Here $j\in\{1,\ldots,10\}$ for a supersingular K3 surface and
$j\in\{1,2\}$ for a supersingular abelian surface. For a supersingular abelian
$g$-fold, the domino carrying $E_1^{0,2}\to E_1^{1,2}$ has dimension
$\binom{g}{2}$, so $g>2$ already forces higher-dimensional dominoes. Beyond the
elementary dominoes, no classification is known. We classify the two-dimensional
ones, put a domino of any dimension into a normal form, and use this to compute
both Brauer groups of natural supersingular families in every dimension.

Following Illusie--Raynaud~\cite{IR83} and Ekedahl~\cite{ekedahl1,ekedahl2,ekedahl3}, the Hodge--Witt groups $H^j(X,W\Omega_X^i)$ are studied as coherent modules over the Raynaud ring $R$, which encodes $F$, $V$, and the de Rham--Witt differential. Ekedahl's d\'evissage gives each a canonical filtration whose graded pieces are Dieudonn\'e modules and dominoes (Definition~\ref{def:coherent}). We recall this formalism in \S\ref{sec:recollections}.

The first part of the paper (\S\ref{sec:domino-structure}) develops the structure theory: a domino $U$ is built from the elementary $U_j$ by iterated extensions and is controlled by its type sequence $J(U)$ and the extension data. The two-dimensional case illustrates this extension phenomenon.

\begin{thmletter}[Two-dimensional dominoes; Theorem~\ref{thm:2d-classification}]
Every $2$-dimensional indecomposable domino $U$ fits into a short exact sequence
\[
  0 \longrightarrow U_{j_2} \longrightarrow U \longrightarrow U_{j_1} \longrightarrow 0,
\]
where $j_2 - j_1 \ge 2$. The isomorphism classes of indecomposable dominoes with type sequence $(j_1, j_2)$ are the orbits
\[
  [f(V)] \;\in\; \bigl(k_\sigma[V]_{\le j_2-j_1-2}\setminus\{0\}\bigr)\big/(k^\times\times k^\times),
\]
where the quotient is by the Frobenius-skew frame-change relation
\(f(V)\mapsto\sigma(b_0)f(V)c_0^{-1}\). In general this is not an ordinary projective space.
\end{thmletter}

In arbitrary dimension we prove a presentation theorem: every domino of fixed
type is presented by a strictly upper triangular matrix of
\(\operatorname{Ext}^1\)-classes over the skew polynomial ring \(k_\sigma[V]\), uniquely
up to explicit upper triangular changes of generators
(Theorem~\ref{thm:fixed-type-groupoid}). In the simplest cases this
equivalence becomes a \(\sigma\)-twisted conjugacy problem, whose
untwisted model (\(\sigma=\mathrm{id}\)) is the \emph{wild} Kronecker problem. A full
list of indecomposables is not to be expected. We nevertheless settle
three cases completely: dimension two (Theorem~A), the two-block types
(\S\ref{sec:explicit-dominoes}), and the type sequences
\((j,j+2,\ldots,j+2n-2)\), where maximal \(p\)-exponent singles out a
\emph{distinguished} domino.

\begin{thmletter}[Distinguished dominoes; Theorem~\ref{thm:distinguished}]
For each $j\in\mathbb{Z}$ and $n\ge 1$, define
\[
  U_{j,\,j+2,\,\ldots,\,j+2n-2}
  \;:=\;
  \widehat{R}/\widehat{R}(F^n,\,dV^{j-1}).
\]
This is the unique domino, up to isomorphism, with type sequence
$(j,j+2,\ldots,j+2n-2)$ and maximal \(p\)-exponent \(n\). In particular, it is
indecomposable, and its dimension and \(p\)-exponent are both equal to \(n\).
\end{thmletter}

Here \(dV^{m}\) means \(F^{-m}d\) for \(m<0\), so that \(dV^{j-1}\) makes sense for every
\(j\in\mathbb Z\).

For the remainder of this introduction we take $k$ to be algebraically closed.
The bridge to the geometric applications comes from the following two realizations
of a domino, one for each of its degrees. By Cartier theory, its degree-zero term
\(U^0\) is the covariant
Cartier module of a formal unipotent group \(\widehat G(U)\). Its degree-one
term \(U^1\) gives a weight-one syntomic realization
\(G^{\mathrm{perf}}(U)\), represented by a perfect unipotent group. Both are
constructed in Appendix~\ref{app:unipotent}, the first from Cartier theory
\cite{Zin84} and the second from the \(F\)-fixed-point construction of
Illusie--Raynaud \cite[Lemma~IV.3.8(b)]{IR83}.
Over \(k=\bar k\), each of these groups is isogenous to a product of truncated
Witt groups,
\[
  \widehat G(U)\;\sim\;\prod_i\widehat W_{\lambda_i},
  \qquad
  G^{\mathrm{perf}}(U)\;\sim\;\prod_j W_{\mu_j}^{\mathrm{perf}},
\]
where \(\widehat W_m\) and \(W_m^{\mathrm{perf}}\) are the formal and perfect
\(m\)-truncated Witt group schemes. We call
\(\lambda=(\lambda_1\ge\lambda_2\ge\cdots)\) and
\(\mu=(\mu_1\ge\mu_2\ge\cdots)\) the \emph{isogeny partitions} of the two
groups.

\begin{thmletter}[Isogeny partition; Theorem~\ref{thm:isogeny-partition}]
Let \(U\) be a domino over \(k=\bar k\). The formal and syntomic realizations
\(\widehat G(U)\) and \(G^{\mathrm{perf}}(U)\) of
Appendix~\ref{app:unipotent} have the same isogeny partition,
\(\lambda=\mu\). Writing \(\lambda(U)\) for the common partition and
\(d_r(U)=\dim(p^rU)\), its conjugate is given by
\(\lambda(U)^\vee_r=d_{r-1}(U)-d_r(U)\). In particular,
\(\lambda_1(U)=p\text{-}\!\exp(U)\).
\end{thmletter}

Every differential \(E_1^{i,j}\to E_1^{i+1,j}\) of the slope spectral sequence
carries a domino \(U_X^{i,j}\). The one in bidegree \((0,2)\) recovers two
unipotent Brauer-type invariants of \(X\). On one side is the formal Brauer group \(\widehat{\mathrm{Br}}_X\), whose unipotent part is the formal realization \(\widehat G(U_X^{0,2})\). On the other is the \emph{perfect Brauer group}\footnote{This terminology is not standard. We use it to emphasize the parallel with the formal Brauer group \(\widehat{\mathrm{Br}}_X\).} \(\mathrm{Br}^{\mathrm{perf}}_X\), the perfect quasi-algebraic group whose \(k\)-points give the nondivisible part of the \(p\)-primary Brauer group, with connected unipotent part the syntomic realization \(G^{\mathrm{perf}}(U_X^{0,2})\). Theorem~C says these two realizations share one
isogeny partition. They are two functors of the single object \(U_X^{0,2}\), which
records extension data that the shared partition
forgets, complementing the isogeny-level relation established by
Grammatica--Skorobogatov--Yang through an auxiliary Bragg--Olsson group
\cite{GSY25,BO21}. For a supersingular abelian variety \(A\) the picture simplifies:
both \(\widehat{\mathrm{Br}}_A\) and \(\mathrm{Br}^{\mathrm{perf}}_A\) are already
unipotent, equal to the two realizations of \(U_A^{0,2}\).

The second part of the paper (\S\ref{sec:applications}) applies the theory to geometry and
explicit computation by extending a reconstruction technique of Nygaard
\cite[Theorem~4.3]{Nygaard81}. Recovering the $E_1$-page of the slope spectral
sequence from the $F$-crystal means recovering its nonvanishing differentials,
and these are carried by the dominoes.
The object we reconstruct is slightly larger than the domino: the degree-two
\emph{diagonal slice} $H^2(X,W\Omega_X^\bullet)^{[0,0]}$, a two-term $R$-module
with domino part $U_X^{0,2}$, defined by a diagonal truncation
(Definition~\ref{def:diagonal-t-structure}). We introduce an exact category of \emph{Ekedahl modules} containing the slice.
An Ekedahl module with no nonzero finite-torsion subobject (nft) is determined
by the data of its \emph{Nygaard pair} $(K_0,K_{-1},\iota,\nu)$: the kernels
$K_0=\ker d$ and
$K_{-1}=\ker(dV)$, the inclusion $\iota$ of $K_{-1}$ in $K_0$, and the
restriction $\nu$ of $V$ to $K_{-1}$. In fact passing from an nft Ekedahl module to its
Nygaard pair is an equivalence of exact categories
(Theorem~\ref{thm:two-kernel-presentation}). The Ekedahl modules are moreover exactly the
intersection of Ekedahl's three hearts on $D^b_c(R)$ (Postnikov, diagonal, and
$F$-gauge), a characterization we establish in
Appendix~\ref{app:ekedahl-three-hearts}. A smooth proper variety is
\emph{Mazur--Ogus} if its crystalline cohomology is torsion-free and its
Hodge--de Rham spectral sequence degenerates at $E_1$. For such a variety the
diagonal slice is an nft Ekedahl module. We show that the $F$-crystal
$H^2_{\mathrm{crys}}(X/W)$ determines its Nygaard pair, and with it the whole
integral slice.

\begin{thmletter}[Reconstruction of $H^2(X,W\Omega_X^\bullet)^{[0,0]}$; Theorem~\ref{thm:geometric-ekedahl-pair-partial-v}]
Let \(X/k\) be a smooth proper Mazur--Ogus variety of dimension \(d\ge2\). Then
the \(F\)-crystal \(H^2_{\mathrm{crys}}(X/W)\) functorially reconstructs the
Nygaard pair \((K_0,K_{-1},\iota,\nu)\) attached to
\(H^2(X,W\Omega_X^\bullet)^{[0,0]}\). Consequently it reconstructs the
coherent $R$-module \(H^2(X,W\Omega_X^\bullet)^{[0,0]}\).
\end{thmletter}

For a supersingular abelian variety \(A\) whose Dieudonn\'e module is
\emph{cyclic} (Definition~\ref{def:cyclic-crystal}), the reconstruction reduces to an explicit combinatorial procedure:
extracting the Nygaard pair is a finite semilinear matrix calculation, and from the
resulting sequence of valuations a combinatorial algorithm computes the entire integral
domino \(U_A^{0,2}\), and with it its type sequence \(J(U_A^{0,2})\), its isogeny partition
\(\lambda(U_A^{0,2})\), and both the formal Brauer group \(\widehat{\mathrm{Br}}_A\) and the
perfect Brauer group \(\mathrm{Br}^{\mathrm{perf}}_A\) (\S\ref{sec:cyclic-matrices}). We apply this
procedure to the superspecial and cyclic supergeneral families.

\begin{thmletter}[Supersingular abelian varieties; Theorems~\ref{thm:supergeneral} and~\ref{thm:superspecial}]
Let \(A/k\) be a supersingular abelian \(g\)-fold.
\begin{enumerate}
\item If \(A\) belongs to the cyclic supergeneral family whose \(p\)-divisible group has
Dieudonn\'e module \(R^0/R^0(F^g-V^g)\), then
\[
  U_A^{0,2}
  \;\cong\;
  U_{2,4,\ldots,2g-2}
  \;\oplus\;
  U_{3,5,\ldots,2g-3}
  \;\oplus\;
  \cdots
  \;\oplus\;
  U_g.
\]
Consequently \(\widehat{\mathrm{Br}}_A\cong
\widehat{W}_{g-1}\oplus\widehat{W}_{g-2}\oplus\cdots\oplus\widehat{\mathbb{G}}_a\)
and \(\mathrm{Br}^{\mathrm{perf}}_A\cong
W_{g-1}^{\mathrm{perf}}\oplus W_{g-2}^{\mathrm{perf}}\oplus\cdots\oplus
\mathbb G_a^{\mathrm{perf}}\).

\item If \(A\) is superspecial, then \(U_A^{0,2}\cong U_1^{\oplus\binom{g}{2}}\). Consequently
\(\widehat{\mathrm{Br}}_A\cong\widehat{\mathbb{G}}_a^{\oplus\binom{g}{2}}\) and
\(\mathrm{Br}^{\mathrm{perf}}_A\cong
\bigl(\mathbb G_a^{\mathrm{perf}}\bigr)^{\oplus\binom{g}{2}}\).
\end{enumerate}
\end{thmletter}

Beyond these families, the reconstruction bounds the invariants of \(U_A^{0,2}\) for
\emph{every} supersingular abelian variety \(A\). Two are of interest: its \(p\)-exponent, the
smallest power of \(p\) annihilating the \(p\)-primary Brauer group, and its degree
\(\deg(U_A^{0,2})\), which we call the \emph{degree-two Artin invariant} \(\sigmaArt(A)\). For a
supersingular surface the latter recovers the classical Artin invariant. We bound both in terms
of the \(a\)-number.

\begin{thmletter}[Bounds on the \(p\)-exponent and \(\sigmaArt\); Lemma~\ref{lem:rank-estimate-lower-bound} and Theorems~\ref{thm:extremal-exponent}, \ref{thm:sigma-a-number-lower-bound}, and~\ref{thm:sigma-discriminant-bound}]
Let \(A/k\) be a supersingular abelian variety of dimension \(g\geq3\), and
set \(a=a(A)\) and \(e=p\text{-}\exp\!\left(U_A^{0,2}\right)\). Then
\[
   \left\lceil\frac{g-1}{a}\right\rceil
   \le
   e
   \le
   g-a+1.
\]
Moreover, \(a=1\) if and only if \(e=g-1\), and \(a=2\) implies \(e\leq g-2\).
For the degree-two Artin invariant \(\sigmaArt(A)=\deg(U_A^{0,2})\), with a principal
polarization assumed for the upper bound,
\[
  g(g-1)-\binom{a}{2}
  \;\le\;
  \sigmaArt(A)
  \;\le\;
  \left\lfloor\frac{e\,g(2g-1)}{2}\right\rfloor,
\]
so the upper bound for \(\sigmaArt\) inherits the \(a\)-dependence of \(e\):
\(\sigmaArt(A)\le\lfloor(g-a+1)\,g(2g-1)/2\rfloor\), with \(g-a+1\) improved to
\(g-2\) when \(a=2\), and \(\sigmaArt(A)\le(g-1)\,g(2g-1)/2\) unconditionally.
The results hold for every prime \(p\), including \(p=2\).
\end{thmletter}

The Nygaard-pair reconstruction turns the $p$-exponent $p\text{-}\exp(U_A^{0,2})$ into the
lattice exponent of $H^2_{\mathrm{crys}}(A/W)$ modulo its Hodge--Witt filtration $L:=\mathrm{Fil}^1_{\mathrm{HW}}$, which
one then bounds by an elementary-divisor calculation for $F^r$ together with Li's superspecial
envelope (Lemma~\ref{lem:li-superspecial-envelope}). For $\sigmaArt$, a principal polarization
places $L$ in a self-dual chain whose discriminant group $L^\vee/L$ has
length $2\sigmaArt(A)$, so crystalline Poincar\'e duality bounds $\sigmaArt$ by half the
maximal discriminant length. We do not know how to derive these bounds without the Hodge--Witt filtration and the domino
formalism. In particular, these geometric results seem difficult to obtain from the Nygaard
filtration alone.

\medskip
\noindent\textbf{Relation to prior work.}
Since Illusie--Raynaud \cite{IR83} and Ekedahl
\cite{ekedahl1,ekedahl2,ekedahl3} introduced the coherent $R$-module formalism and its
d\'evissage, the domino has largely been approached through its dimension $T^{i,j}$, known as
the \emph{domino numbers}. Antieau--Bragg \cite[Theorem~5.29]{AB22} prove that the $T^{i,j}$ are derived
invariants of smooth proper varieties of dimension $\le 3$.

Using the classifying-stack approximation of Antieau--Bhatt--Mathew
\cite[Theorem~1.2]{ABM21}, Li--Yang
\cite[Theorem~1.1 and Corollary~5.12]{LY26} compute the Hodge--Witt cohomology of $B\alpha_p$ to produce an explicit counterexample to Hodge--Witt symmetry: a smooth proper fourfold with asymmetric Hodge--Witt numbers in total degree three, sharp in both dimension and total degree. Their work involves only the elementary dominoes $U_j$, the $1$-dimensional dominoes, whereas here we study the higher-dimensional ones.

Grammatica--Skorobogatov--Yang study the same $p$-primary Brauer groups through
the Nygaard filtration and pose the exponent problem in
\cite[Question~5.3]{GSY25}. Theorem~F answers it, proves that $a=1$ if and only if the $p$-exponent equals $g-1$, and sharpens their bound on both sides in terms of the $a$-number, for every prime $p$. Where they compute the isogeny type of the connected unipotent group, we determine the integral $R$-module $U_A^{0,2}$ that refines it.

\medskip
\noindent\textbf{Broader perspective.}
In the language of the syntomic stack \(k^{\mathrm{Syn}}\) \cite{BhattFgauges}, Ekedahl's derived equivalence identifies \(D^b_c(R)\) with the coherent derived category \(D^b_c(k^{\mathrm{Syn}})\) of \(F\)-gauges \cite[Theorem~II.5.3]{ekedahl3}. Through the stacky approach to prismatic cohomology \cite{BhattLurie22,Drinfeld20,BhattFgauges}, dominoes and diagonal dominoes correspond to torsion objects in \(D^b_c(k^{\mathrm{Syn}})\). Every Ekedahl module gives a coherent \(F\)-gauge with Hodge--Tate weights in \([0,2]\). Among nft Ekedahl modules, the torsion objects are exactly the positive dominoes (Appendix~\ref{app:ekedahl-three-hearts}). Thus positive dominoes form a concrete torsion family in Hodge--Tate weights \([0,2]\), beyond the weights \([0,1]\) of classical Dieudonn\'e theory. Ekedahl's diagonal dominoes give the analogues for general intervals \([m,n]\). In a sequel we prove that they admit a \(\mathbb{Q}\)-indexed Harder--Narasimhan filtration refining his half-integer-indexed type filtration.

\medskip
\noindent\textbf{Notation and conventions.}
We use \(X\) for a general smooth proper variety and \(A\) for an abelian variety.
Throughout, \(k\) is a perfect field of characteristic \(p>0\) with Witt vectors
\(W=W(k)\), fraction field \(K\), and Frobenius \(\sigma\). In
\S\ref{sec:supersingular} and Appendix~\ref{app:unipotent} we take \(k=\bar k\), as
stated there. We use triangulated categories throughout, although we believe
everything can be upgraded to \(\infty\)-categorical statements as in \cite{LY26}.
The remaining notation is collected here.

\smallskip
\begin{tabular}{@{}p{0.30\textwidth}@{\hspace{1em}}p{0.60\textwidth}@{}}
\(R=R^0\oplus R^1\), \(\widehat R\), \(\widehat R^0\)
  & the Raynaud ring, with \(R^0\) the Dieudonn\'e ring, and their completions
    (Definition~\ref{def:completed-raynaud})\\[2pt]
\(\mathrm{Mod}_c(R)\), \(\Delta\), \(\mathcal G\)
  & the hearts of Ekedahl's Postnikov, diagonal, and \(F\)-gauge \(t\)-structures
    on \(D^b_c(R)\)\\[2pt]
\(\mathcal E=\mathrm{Mod}_c(R)\cap\Delta\cap\mathcal G\)
  & the Ekedahl modules (Appendix~\ref{app:ekedahl-three-hearts})\\[2pt]
\(U_j:=\widehat R/\widehat R(F,dV^{j-1})\)
  & the elementary dominoes\\[2pt]
\(C(M)\), \(\Dom(M)\)
  & the core and the domino part of a coherent \(R\)-module \(M\)
    (\S\ref{sec:hearts-dominoes})\\[2pt]
\(J(U)\), \(\lambda(U)\), \(p\text{-}\!\exp(U)\)
  & the type sequence, isogeny partition, and \(p\)-exponent of a domino \(U\)\\[2pt]
\(\widehat G(U)\), \(G^{\mathrm{perf}}(U)\)
  & the formal and syntomic unipotent realizations of a domino \(U\)
    (Appendix~\ref{app:unipotent})\\[2pt]
\(U_X^{i,j}\)
  & the domino attached to
    \(d\colon H^j(X,W\Omega_X^i)\to H^j(X,W\Omega_X^{i+1})\)
    (\S\ref{sec:hearts-dominoes})\\[2pt]
\(a(A)\), \(\sigmaArt(A):=\deg(U_A^{0,2})\)
  & the \(a\)-number and the degree-two Artin invariant of a supersingular abelian
    variety (Definition~\ref{def:degree-two-artin-invariant})
\end{tabular}
\smallskip

We call \(C(M)\) the \emph{core} rather than the \emph{heart}, translating
Illusie--Raynaud's \emph{c\oe ur}, to avoid a clash with the hearts of
\(t\)-structures; likewise we speak of \emph{survival of the core} \cite{IR83}.

\medskip
\noindent\textbf{Acknowledgements.} The author thanks his advisor, Ben Antieau, for his guidance and for many helpful discussions throughout the preparation of this work. He is grateful to Sasha Petrov for suggesting this project, for pointing out several errors in an earlier draft, and for the appendix to \cite{SP25}. He thanks Alexei Skorobogatov for helpful correspondence about the results of \cite{GSY25}, and Shizhang Li and Yuan Yang for patiently answering his questions about \cite{LY26}. The author is supported by the Simons Collaboration on Perfection in Algebra, Geometry, and Topology.

\section{Recollections on de Rham--Witt cohomology}\label{sec:recollections}

The Illusie--Raynaud--Ekedahl formalism recalled below is organized by the fact
that the de Rham--Witt cohomology
\(R\Gamma(X,W\Omega_X^\bullet)\), with its operators \(F\), \(V\), and \(d\), is
naturally a single object of the derived category \(D^b_c(R)\) over the Raynaud ring
\(R\). Many properties of de Rham--Witt cohomology are then consequences of the rich
homological algebra of \(D^b_c(R)\).
We begin in \S\ref{sec:raynaud-ring} with the Raynaud ring, its completion, and the
finiteness condition defining coherent \(R\)-modules, of which the elementary
dominoes \(U_j\) are the basic examples that are not finitely generated over
\(W\).
It then presents Ekedahl's d\'evissage, which extracts from a
coherent module its Dieudonn\'e \emph{core} and its \emph{domino}, and records the
survival-of-the-core theorem that makes dominoes measure the failure of
\(E_1\)-degeneration.
\S\ref{sec:devissage} collects the type filtration and numerical invariants of a domino, the Poincar\'e-duality constraint locating dominoes in the slope spectral sequence, and the Nygaard modifications under which all these invariants are stable. Their finer structure theory is then developed in
\S\ref{sec:domino-structure}.

\subsection{Raynaud ring and coherent modules}\label{sec:raynaud-ring}\label{sec:hearts-dominoes}
For a smooth \(k\)-scheme \(X\), Illusie's de Rham--Witt complex
\(W\Omega_X^\bullet\) computes crystalline cohomology
\cite[Theorem~II.1.4]{Illusie79} and gives
the slope spectral sequence
\[
  E_1^{i,j}=H^j(X,W\Omega_X^i) \;\Rightarrow\;
  H^{i+j}_{\mathrm{crys}}(X/W).
\]
We call the terms \(H^j(X,W\Omega_X^i)\) the \emph{Hodge--Witt cohomology} of \(X\).
Assembled with \(F\), \(V\), and \(d\), they form the \emph{de Rham--Witt
cohomology} \(R\Gamma(X,W\Omega_X^\bullet)\). Neither is finitely generated over
\(W\) in general, whereas the abutment \(H^{i+j}_{\mathrm{crys}}(X/W)\) always
is.

After inverting \(p\), the spectral sequence degenerates
\cite[Theorem~II.3.2]{Illusie79} and identifies \(H^j(X,W\Omega_X^i)[1/p]\) with the slope-\([i,i+1)\) part \(H^{i+j}_{\mathrm{crys}}(X/W)[1/p]_{[i,i+1)}\) of crystalline cohomology. Integrally, its failure to degenerate is
measured by the domino part of the coherent \(R\)-module cohomology. By the
survival-of-the-core theorem~\cite{IR83} (Proposition~\ref{prop:survival-core}
below), the nonvanishing differentials detect precisely this contribution.

\begin{definition}
The \emph{Raynaud ring} is the graded associative \(\mathbb{Z}_p\)-algebra
\(R := W_\sigma\{F,V,d\}/{\sim}\), where \(F,V\) have degree \(0\), \(d\) has degree
\(1\), the subscript \(\sigma\) indicates Frobenius semilinearity \(\sigma(a)F=Fa\),
\(aV=V\sigma(a)\) for \(a\in W\), and \(\sim\) is the two-sided ideal generated by the
relations \(FV=VF=p\), \(d^2=0\), and \(FdV=d\).
\end{definition}

\begin{remark}
The ring \(R=R^0\oplus R^1\) is concentrated in degrees \(0\) and \(1\), the degree-\(0\)
piece \(R^0=W_\sigma[F,V]/(FV-p)\) being the classical Dieudonn\'e ring. From \(FdV=d\)
one deduces \(dF=pFd\) and \(Vd=pdV\): multiplying on the right by \(F\) gives
\(dF=FdVF=pFd\) since \(VF=p\), and on the left by \(V\) gives \(Vd=VFdV=pdV\).
\end{remark}

To relate the full de Rham--Witt complex to its truncated versions, one equips $R$ with a natural filtration.

\begin{definition}[Filtration, truncations, and completion]\label{def:completed-raynaud}
The right ideals \(V^nR+dV^nR\) form a decreasing filtration
\(R\supset VR+dVR\supset V^2R+dV^2R\supset\cdots\). For \(n\ge1\) the \(n\)-truncated
quotient \(R_n:=R/(V^nR+dV^nR)\) carries a natural \((W_n[d],R)\)-bimodule structure,
where \(W_n[d]\) is the graded \(W_n\)-algebra on a degree-\(1\) generator \(d\) with
\(d^2=0\). Finally, the \emph{completed Raynaud ring} is
\(\widehat{R}:=\varprojlim_n R_n\), with transition maps the quotient maps. Although
\(V^nR+dV^nR\) is not a two-sided ideal of \(R\), the multiplication extends by
continuity to \(\widehat{R}\), making it a graded topological ring whose degree-\(0\)
part \(\widehat{R}^0=W_\sigma[[V]][F]/(FV-p)\) is the \(V\)-completed Dieudonn\'e ring.
\end{definition}

\begin{remark}
Every element of $R$ can be written uniquely in the form of a finite sum
\[
  \sum_{n>0} a_{-n}V^n + \sum_{n\ge 0} a_nF^n + \sum_{n>0} b_{-n}\,dV^n + \sum_{n\ge 0} b_nF^n d,
\]
with coefficients $a_n,b_n\in W$. For an element of $\widehat{R}$, the sums involving the $V^n$ and $dV^n$ terms may be infinite.

Similarly, every element of the truncated quotient $R_n=R/(V^nR+dV^nR)$ can be written uniquely as a finite sum
\[
  \sum_{n>m>0} a_{-m}V^m + \sum_{m\ge 0} a_mF^m + \sum_{n>m>0} b_{-m}\,dV^m + \sum_{m\ge 0} b_mF^m d,
\]
where $a_{-m},b_{-m}\in W/p^{n-m}$ for $n>m>0$, and $a_m,b_m\in W/p^n$ for $m\ge 0$.
\end{remark}

\begin{definition}\label{def:raynaud}
Let $D_{\mathrm{gr}}(\mathbb{Z}_p)$ denote the derived category of graded $\mathbb{Z}_p$-modules. The Raynaud ring $R$ is an associative algebra over its center $\mathbb{Z}_p$. Note that $W$ is not central, since $Fa=\sigma(a)F$. We define
\[
  D(R) \;:=\; \mathrm{LMod}_R\!\bigl(D_{\mathrm{gr}}(\mathbb{Z}_p)\bigr),
\]
the category of left $R$-module objects in $D_{\mathrm{gr}}(\mathbb{Z}_p)$. Such a module has an underlying graded $W$-module through $W\subset R^0$. It carries the
Postnikov \(t\)-structure induced from the Postnikov \(t\)-structure on
\(D_{\mathrm{gr}}(\mathbb{Z}_p)\). Its heart is
\[
  \mathrm{Mod}(R) \;:=\; \mathrm{LMod}_R\!\bigl(\mathrm{Mod}_{\mathrm{gr}}(\mathbb{Z}_p)\bigr),
\]
the abelian category of graded left $R$-modules. An object of $\mathrm{Mod}(R)$ is called an \emph{\(R\)-module} (also known as a \emph{Raynaud module}). We write $D^b(R)$ for the full subcategory of bounded objects in $D(R)$.
\end{definition}

For a smooth variety $X$ over $k$, the Frobenius $F$, Verschiebung $V$, and differential $d$ on the de Rham--Witt complex $W\Omega_X^\bullet$ satisfy the defining relations of $R$, making it a sheaf of graded $R$-modules. Its derived global sections form an object
\(R\Gamma(X,W\Omega_X^\bullet)\in D(R)\). Its cohomology objects package the Hodge--Witt groups and the de Rham differential
\[
  H^j(X,W\Omega_X^0) \xrightarrow{d} H^j(X,W\Omega_X^1) \xrightarrow{d} H^j(X,W\Omega_X^2) \to \cdots
\]
together with their $F$- and $V$-structures. This is the $j$-th row of the $E_1$-page of the slope spectral sequence.

\begin{remark}[Totalization]\label{rem:simple-functor}
An object of \(D(R)\) may be represented by a double complex
\((M^{\bullet,\bullet},d,\partial)\), where \(d\) is the internal Raynaud
differential and \(\partial\) is the cohomological one. The derived category
localizes only in the \(\partial\)-direction. We write \(M(n)\) for the shift of the \emph{internal grading}, \((M(n))^i=M^{n+i}\), and \(M[n]\) for the shift of the \emph{cohomological grading}. Let \(\mathbf s:D(R)\to D(W)\) denote
Ekedahl's simple functor, which totalizes this double complex
\cite[Section~O.3]{ekedahl3}.
For a smooth variety $X/k$, one has a canonical identification
\[
  \mathbf s\,R\Gamma(X,W\Omega_X^\bullet)\simeq R\Gamma_{\mathrm{crys}}(X/W).
\]
\end{remark}

\begin{remark}[Truncation]\label{rem:truncation}
The $(W_n[d],R)$-bimodule structure on $R_n$ defines a truncation functor
\(R_n\otimes_R^{\mathbf L}(-):D(R)\to D(W_n[d])\).
For the de Rham--Witt complex this recovers the usual truncated complex:
\[
  R_n\otimes_R^{\mathbf L} R\Gamma(X,W\Omega_X^\bullet)
  \simeq
  R\Gamma(X,W_n\Omega_X^\bullet);
\]
in particular \(R_1\otimes_R^{\mathbf L} R\Gamma(X,W\Omega_X^\bullet)\simeq
R\Gamma(X,\Omega_{X/k}^\bullet)\), the algebraic de Rham complex with its Hodge
filtration.
\end{remark}

The coherence of $R$-modules, due to Illusie--Raynaud and Ekedahl, is not the ring-theoretic coherence of $R$. In fact $R$ itself is not coherent in this sense. The correct viewpoint, recalled in Appendix~\ref{app:ekedahl-three-hearts}, is that it is forced by coherence on the $F$-gauge side of Ekedahl's equivalence. Concretely, a coherent module is assembled from two elementary types, Dieudonn\'e modules and elementary dominoes. For simplicity we keep the Dieudonn\'e part as one type, whereas Ekedahl's finer d\'evissage \cite[Theorem~IV.3.3]{ekedahl1} splits it into four, the slope-$0$, positive-slope, semisimple-torsion, and nilpotent-torsion pieces.

\begin{definition}[Coherent $R$-modules]\label{def:coherent}
An \(R\)-module is \emph{coherent} if it admits a finite filtration whose successive quotients are shifts \(N(n)\), \(n\in\mathbb Z\), of modules of the following two types.

\smallskip
\noindent\emph{(a) Dieudonn\'e modules.} An \(R\)-module concentrated in degree \(0\) (an \(R^0\)-module), finitely generated over \(W\), on which \(V\) is topologically nilpotent (equivalently \(V^nM\to0\) in the \(p\)-adic topology). These are the classical Dieudonn\'e modules, not necessarily finite free.

\smallskip
\noindent\emph{(b) Elementary dominoes.} For \(j\in\mathbb Z\), the module \(U_j:=\widehat R/\widehat R(F,dV^{j-1})\) concentrated in degrees \(0\) and \(1\), with
\[
  (U_j)^0 = \prod^\infty_{n\ge 0} k\,V^n,\qquad (U_j)^1 = \prod^\infty_{n\ge j} k\,dV^n,
\]
where \(F,V,d\) act by left multiplication (with the convention \(dV^n = F^{-n}d\) for \(n<0\)). Here \(V\) is injective on \((U_j)^0\) and zero on \((U_j)^1\), while \(F\) is zero on \((U_j)^0\) and surjective on \((U_j)^1\).

\smallskip
In particular a coherent \(R\)-module has bounded degrees.
\end{definition}

\begin{remark}
The structure of $U_j$ depends on the sign of $j$ (cf.\ \cite[I, (2.14)]{IR83}).

\smallskip
\noindent\textit{Case $j \geq 0$.} The differential $d$ is zero on the first $j$ generators of $(U_j)^0$ and an isomorphism onto $(U_j)^1$ from the $j$-th generator onward:
\[
\begin{array}{ccccccccccc}
(U_j)^0\colon & k & \xrightarrow{\;V\;} & \cdots & \xrightarrow{\;V\;} & kV^{j} & \xrightarrow{\;V\;} & kV^{j+1} & \xrightarrow{\;V\;}\cdots \\[6pt]
 & & & & & \big\downarrow\scriptstyle{d} & & \big\downarrow\scriptstyle{d} & \\[4pt]
(U_j)^1\colon & & & & & k\,dV^j & \xleftarrow{\;F\;} & k\,dV^{j+1} & \xleftarrow{\;F\;}\cdots
\end{array}
\]
Hence $\ker(d)\cong k^j$ (spanned by $1, V, \ldots, V^{j-1}$) and $\operatorname{coker}(d)=0$.

\smallskip
\noindent\textit{Case $j \leq 0$.} The differential $d$ acts on every element of $(U_j)^0$, while $(U_j)^1$ has a left tail $k\,dV^j \xleftarrow{F} \cdots \xleftarrow{F} k\,dV^{-1}$ (where $dV^n = F^{-n}d$ for $n < 0$) with no incoming $d$ arrows:
\[
\begin{array}{ccccccccccc}
(U_j)^0\colon & & & k & \xrightarrow{\;V\;} & kV & \xrightarrow{\;V\;} & kV^2 & \xrightarrow{\;V\;}\cdots \\[6pt]
 & & & \big\downarrow\scriptstyle{d} & & \big\downarrow\scriptstyle{d} & & \big\downarrow\scriptstyle{d} & \\[4pt]
(U_j)^1\colon & k\,dV^{j} & \xleftarrow{\;F\;}\cdots\xleftarrow{\;F\;} & k\,d & \xleftarrow{\;F\;} & k\,dV & \xleftarrow{\;F\;} & k\,dV^2 & \xleftarrow{\;F\;}\cdots
\end{array}
\]
Hence $\ker(d)=0$ and $\operatorname{coker}(d)\cong k^{-j}$ (spanned by $dV^j,\ldots,dV^{-1}$).

\smallskip
We define the Euler characteristic of $U_j$ to be \(\dim_k\ker(d) - \dim_k\operatorname{coker}(d)\), and it is obviously equal to $j$.
\end{remark}

\begin{definition}[Domino]\label{def:domino-intrinsic}
A \emph{domino} is a two-term $R$-module $U=(d:U^0\to U^1)$, concentrated in degrees $0$ and $1$, that is a finite iterated extension of elementary dominoes $U_j$.
\end{definition}

We use the following notation for coherent objects:
\(\mathrm{Mod}_c(R)\subset \mathrm{Mod}(R)\)
denotes the full subcategory of coherent $R$-modules, $D_c(R)\subset D(R)$ the full subcategory of objects whose cohomology lies in $\mathrm{Mod}_c(R)$, and
\(D^b_c(R):=D^b(R)\cap D_c(R)\)
the bounded objects with coherent cohomology.

\begin{definition}[Completion]
For an $R$-module $M$, define its completion by \(\widehat{M}:=\varprojlim_n M/(V^nM+dV^nM)\). We say that \(M\) is \emph{complete} if the natural map \(M\to\widehat{M}\) is an isomorphism.
For a bounded object of $D(R)$, we use the same term when its cohomology $R$-modules are complete.
\end{definition}

\begin{proposition}[Ekedahl's criterion for coherence {\cite[Proposition~III.1.1]{ekedahl2}}]
\label{prop:coherence-criterion}
Let $M\in D^b(R)$. The following conditions are equivalent.
\begin{enumerate}
  \item $M\in D^b_c(R)$, i.e. $M$ has coherent cohomology.
  \item $M$ is complete and \(R_n \otimes_R^{\mathbf L} M \in D^b_c(W_n[d])\) for all $n\ge 1$, where $D^b_c(W_n[d])$ denotes the full subcategory of $D^b(W_n[d])$ consisting of complexes whose cohomology groups are finitely generated $W_n$-modules.
  \item $M$ is complete and \(R_1 \otimes_R^{\mathbf L} M \in D^b_c(k[d])\).
\end{enumerate}
\end{proposition}

It follows \cite[Chapter~0, p.~191]{ekedahl1} that \(D^b_c(R)\subset D(R)\) is
triangulated. In particular, coherent \(R\)-modules form an abelian subcategory
of \(\mathrm{Mod}(R)\) closed under extensions.

To isolate the domino contribution inside a coherent $R$-module $M$, one identifies the submodule on which $d$ is \emph{stably} killed by $V$ (i.e.\ $dV^n = 0$ for all $n \gg 0$) and the submodule generated by all $F$-iterated images of $d$.

Let $M$ be an $R$-module. For each $i\in\mathbb Z$ set $Z^i(M):=\ker\bigl(d:M^i\to M^{i+1}\bigr)$ and $B^i(M):=\mathrm{im}\bigl(d:M^{i-1}\to M^i\bigr)$.
The submodule $Z^i(M)$ is stable under $F$ (but not in general under $V$), whereas $B^i(M)$ is stable under $V$ (but not in general under $F$). We therefore pass to the following $R^0$-submodules of $M^i$:
\begin{equation*}
  V^{-\infty}Z^i(M):=\bigcap_{n\ge 0}\ker\bigl(dV^n:M^i\to M^{i+1}\bigr),\qquad
  F^{\infty}B^i(M):=\bigcup_{n\ge 0}\mathrm{im}\bigl(F^n d:M^{i-1}\to M^i\bigr).
\end{equation*}
One has inclusions $B^i(M)\subset F^{\infty}B^i(M)\subset V^{-\infty}Z^i(M)\subset Z^i(M)$.

\begin{definition}[Core]
The $R^0$-module
\[
  C^i(M):=V^{-\infty}Z^i(M)\big/ F^{\infty}B^i(M)
\]
is called the \emph{core} of \(M\) in degree \(i\).
\end{definition}

The differential $d:M^i\to M^{i+1}$ admits a canonical factorization
\[
  M^i \longrightarrow M^i/V^{-\infty}Z^i(M)
  \longrightarrow F^{\infty}B^{i+1}(M)
  \longrightarrow M^{i+1}.
\]

When $M$ is coherent, the factorization above separates the Dieudonn\'e part from the domino part, as in Definition~\ref{def:coherent}.

\begin{proposition}[D\'evissage of core and domino {\cite[Lemma~0.8]{ekedahl1}}]\label{prop:core-domino-devissage}
If $M$ is coherent, then in each degree $i$ the core $C^i(M)$ is a Dieudonn\'e module, finitely generated over $W$ with $V$ topologically nilpotent, while the complementary map $M^i/V^{-\infty}Z^i(M)\longrightarrow F^{\infty}B^{i+1}(M)$ is a shift $U(-i)$ of a domino $U$.
\end{proposition}

For a coherent \(R\)-module \(M\), we denote the domino in degree \(i\) by
\[
  \Dom^i(M)
  :=
  \left[
    M^i/V^{-\infty}Z^i(M)
    \longrightarrow
    F^{\infty}B^{i+1}(M)
  \right],
\]
viewed after shifting internal degree by \(i\) as a domino concentrated in degrees
\(0\) and \(1\). For a smooth proper variety \(X/k\), if \(M_j:=H^j\!\left(R\Gamma(X,W\Omega_X^\bullet)\right)\),
we also write \(U_X^{i,j}:=\Dom^i(M_j)\)
for the domino attached to $d:H^j(X,W\Omega_X^i)\longrightarrow H^j(X,W\Omega_X^{i+1})$.

\begin{remark}\label{rem:domino-intrinsic-characterization}
By \cite[Definition~I.2.16 and Proposition~I.2.17]{IR83}, a two-term $R$-module $U$ is a domino if and only if it is profinite\footnote{This condition prevents $\widehat{R}$ from being a domino.}, i.e.\ $U/(V^nU+dV^nU)$ is of finite length over $W$ for all $n\ge 1$, and has trivial core, $V^{-\infty}Z^0(U)=0$ and $F^{\infty}B^1(U)=U^1$, so that $U$ coincides with its own domino part $\Dom^0(U)$.
\end{remark}

The core and the domino behave differently in the slope spectral sequence: the core survives to the $E_\infty$-page, while the domino, non-finitely generated on $E_1$, becomes finite-length torsion there. The next result makes the survival of the core precise.

\begin{proposition}[Survival of the core {\cite[Proposition~III.1.1(ii)]{ekedahl2}; \cite[Proposition~I.2.3(i)]{ekedahl3}}]\label{prop:survival-core}
Let $M\in D^b_c(R)$. Consider the slope spectral sequence \(E_1^{i,j}(M)=H^j(M^i) \Rightarrow H^{i+j}(\mathbf s(M))\).
Then for all $i,j$ one has
\[
  B^{i,j}_\infty \subset F^{\infty}B^i\bigl(H^j(M)\bigr) \subset V^{-\infty}Z^i\bigl(H^j(M)\bigr) \subset Z^{i,j}_\infty,
\]
and hence \(C^i\bigl(H^j(M)\bigr)\) survives as a subquotient of \(E^{i,j}_\infty\).
\end{proposition}

\subsection{Invariants of dominoes}\label{sec:devissage}\label{sec:poincare-domino}

Dominoes appeared in \S\ref{sec:hearts-dominoes} as the non-finitely-generated pieces complementary to the core. We record the invariants of a domino: Ekedahl's canonical type filtration and the numerical invariants it determines, the Poincar\'e-duality constraint locating dominoes in the slope spectral sequence, and the Nygaard modifications under which these invariants are stable.

Ekedahl shows that every domino carries a canonical, functorial filtration, so that its type is an isomorphism invariant, independent of any chosen presentation.

\begin{proposition}[Type filtration {\cite[Proposition~III.3.2]{ekedahl3}}]
\label{prop:type-filtration}
Every domino $U$ admits a unique finite decreasing filtration
\[
  0 \subset \cdots \subset \mathrm{Fil}_{\mathrm{type}}^j(U)
  \subset \mathrm{Fil}_{\mathrm{type}}^{j-1}(U) \subset \cdots \subset U,
\]
whose $j$-th associated graded piece \(\operatorname{gr}_{\mathrm{type}}^j(U)\) is a finite iterated extension of copies of \(U_j\). We write \(\operatorname{gr}_{\mathrm{type}}(U)=\bigoplus_j\operatorname{gr}_{\mathrm{type}}^j(U)\) for the total associated graded.
Here \(\mathrm{Fil}_{\mathrm{type}}^j(U)\) is the maximal sub-\(R\)-module on which \(\operatorname{coker}(dV^j)=0\).
\end{proposition}

\begin{definition}[Numerical invariants of a domino]\label{def:domino-numerics}\label{def:type-sequence}
Let \(U\) be a domino. We define its \emph{dimension} to be \(\dim(U):=\dim_k(U^0/VU^0)<\infty\),
equivalently, $\dim(U)=\dim_k(U^1[F])$
\cite[Proposition~I.2.18, (2.18.1)]{IR83}. The \emph{multiplicity} of $U_j$ in $U$ is $m_j(U):=\dim(\operatorname{gr}_{\mathrm{type}}^j(U))$, and the \emph{type sequence} $J(U)=(j_1,\ldots,j_n)$ is the nondecreasing sequence in which each $j\in\mathbb Z$ appears $m_j(U)$ times, so $n=\sum_{j\in\mathbb Z}m_j(U)=\dim(U)$. The last equality holds because \(\dim\) is additive in short exact sequences of dominoes: by \(\dim(U)=\dim_k(U^1[F])\) and the surjectivity of \(F\) on the degree-one term of a domino, taking \(F\)-kernels of \(0\to U'^1\to U^1\to U''^1\to0\) stays exact on the right. We define
\[
  \deg(U):=\sum_{i=1}^n j_i,\qquad
  \mu(U):=\frac{\deg(U)}{\dim(U)}.
\]
The degree also has an Euler-characteristic interpretation,
\[
  \deg(U)=\chi(U):=\lgth_W(H^0(U))-\lgth_W(H^1(U)),
\]
where $\lgth_W(-)$ denotes length over $W$ and $H^0(U)$, $H^1(U)$ are the cohomology of the two-term complex $U^0\xrightarrow{d}U^1$; this is proved in Remark~\ref{rem:kernel-lengths} below.
Finally, define the \emph{\(p\)-exponent} of \(U\) to be \(p\text{-}\!\exp(U):=\min\{n\in\mathbb{N}\mid p^nU^0=0\}\).
Equivalently, $p\text{-}\!\exp(U)$ is the smallest integer $n\ge 1$ such that $p^nU^1=0$
\cite[Proposition~I.2.18, (2.18.3)]{IR83}.
\end{definition}

For an elementary domino \(U_j\), one has \(J(U_j)=(j)\), hence
\(\dim(U_j)=1\), \(\deg(U_j)=\mu(U_j)=j\), and
\(p\text{-}\!\exp(U_j)=1\).

\begin{remark}[Kernel lengths]\label{rem:kernel-lengths}
For a domino \(U\) with multiplicities \(m_j=m_j(U)\) and any \(r\in\mathbb Z\),
\begin{equation}\label{eq:kernel-lengths}
  \lgth_W\ker(dV^r)=\sum_j m_j\max(0,j-r),
  \qquad
  \lgth_W\operatorname{coker}(dV^r)=\sum_j m_j\max(0,r-j),
\end{equation}
which for \(r=0\) is the equality \(\deg(U)=\chi(U)\) used above. On \(U_j\) the
two diagrams of \S\ref{sec:hearts-dominoes} give the two lengths as
\(\max(0,j-r)\) and \(\max(0,r-j)\), so \(dV^r\) is surjective there for
\(r\le j\) and injective for \(r\ge j\), and both properties pass to
extensions. Hence in the snake sequence of \(dV^r\) for
\(0\to\mathrm{Fil}_{\mathrm{type}}^{j+1}(U)\to\mathrm{Fil}_{\mathrm{type}}^{j}(U)
\to\operatorname{gr}_{\mathrm{type}}^{j}(U)\to0\) the connecting map vanishes,
the first cokernel being zero for \(r\le j+1\) and the last kernel for
\(r\ge j+2\). Both sides of \eqref{eq:kernel-lengths} are therefore additive
along the canonical type filtration, which reduces the claim to a single
\(U_j\).
\end{remark}

Finally, we turn to how dominoes appear in the cohomology of a smooth proper variety. Their dimensions are the domino numbers \(T^{i,j}\), and Ekedahl's Poincar\'e duality \cite[Corollary~IV.3.5.1]{ekedahl1} restricts the bidegrees at which they can occur. We use only this numerical constraint.

\begin{definition}\label{def:domino-numbers}
Let \(X/k\) be a smooth proper variety. The \emph{domino number} of \(X\) at position \((i,j)\) is \(T^{i,j}(X):=\dim U_X^{i,j}\), the dimension of the domino of \S\ref{sec:hearts-dominoes} attached to \(d:H^j(X,W\Omega_X^i)\to H^j(X,W\Omega_X^{i+1})\). We set \(T^{i,j}(X):=0\) for \(i<0\) or \(j<0\).
\end{definition}

The \emph{Nygaard modifications} of Definition~\ref{def:autoequivalence} serve two purposes. In \S\ref{sec:domino-structure} they shift the numerical types of the elementary dominoes, normalizing extension calculations to the single module \(U_0\). They also modify the slope spectral sequence, replacing the differential \(d\) by \(F^rd\) or \(dV^r\). The constraint applies to every such modification, and the reconstruction of \S\ref{Mazur-Ogus} relies on this generality.

\begin{definition}[Nygaard modifications]\label{def:autoequivalence}
Let \(\tau\colon\mathbb Z\to\mathbb Z\) have finite support, and set \(s_i(\tau):=\sum_{m<i}\tau(m)\) for each \(i\in\mathbb Z\). For an \(R\)-module \(M\), the \emph{Nygaard modification} \(M(\tau)\) is the graded module with \(M(\tau)^i:=\sigma_*^{s_i(\tau)}M^i\), where \(\sigma_*\) is the Frobenius twist of the \(W\)-action. Its \(F\)- and \(V\)-actions are induced from those on \(M\), and its differential \(d_\tau\colon M(\tau)^i\to M(\tau)^{i+1}\) is \(F^{\tau(i)}d\), with the convention \(F^{-r}d:=dV^r\) for \(r>0\). This map is \(\sigma^{\tau(i)}\)-semilinear, and since \(s_{i+1}(\tau)=s_i(\tau)+\tau(i)\) it is \(W\)-linear between the twisted pieces, so \(M(\tau)\) is again an \(R\)-module.

Each \(M\mapsto M(\tau)\) is an exact endofunctor of \(\mathrm{Mod}(R)\), and \((M(\tau_1))(\tau_2)=M(\tau_1+\tau_2)\) with \(M(0)=M\). Thus the group of finitely supported functions \(\mathbb Z\to\mathbb Z\) acts on \(\mathrm{Mod}(R)\) by exact autoequivalences, and degreewise on \(D(R)\). Writing \(\delta_0\) for the indicator of \(0\), we have \(U_j(r\delta_0)=U_{j+r}\), and more generally \(U_j(\tau)\simeq U_{j+\tau(0)}\), since \(\sigma_*U_j\simeq U_j\): the Frobenius twist affects only the coefficients of the topological basis of Definition~\ref{def:coherent}. As Dieudonn\'e modules twist to Dieudonn\'e modules, the modifications preserve \(\mathrm{Mod}_c(R)\) and \(D^b_c(R)\).
\end{definition}

By exactness and the uniqueness of the type filtration, modifications preserve the dimension of every domino and translate its type sequence by \(\tau(0)\): \(\dim(U(\tau))=\dim(U)\) and \(J(U(\tau))=J(U)+\tau(0)\). Ekedahl established the following for the slope spectral sequence itself. We need the version incorporating its Nygaard modifications.

\begin{proposition}[Domino range]\label{prop:ekedahl-duality}
Let \(X/k\) be a smooth proper variety of dimension \(n\), let \(\tau\) be finitely
supported, and consider the slope spectral sequence of the modification \(M(\tau)\) of
\(M:=R\Gamma(X,W\Omega_X^\bullet)\in D^b_c(R)\),
\[
  E_1^{i,j}(\tau) = \sigma_*^{s_i(\tau)}H^j(X,W\Omega_X^i)
  \;\Longrightarrow\; H^{i+j}\bigl(\mathbf s(M(\tau))\bigr),
  \qquad d_1 = F^{\tau(i)}d ,
\]
the case \(\tau=0\) being the slope spectral sequence itself.
\begin{enumerate}
  \item \(T^{i,j}(X) = T^{n-i-2,\,n-j+2}(X)\) for all \(i,j\). Consequently
  \(T^{i,j}(X) = 0\) unless \(0 \le i \le n-2\) and \(2 \le j \le n\)
  (the \emph{domino range}).
  \item If the differential
  \(d_r\colon E_r^{i,j}(\tau)\to E_r^{i+r,\,j-r+1}(\tau)\) is nonzero for some
  \(r\ge1\), then \(T^{i,j}(X)\neq0\) and \(T^{i+r-1,\,j-r+1}(X)\neq0\): every
  differential flows from one domino to another.
\end{enumerate}
\end{proposition}

\begin{proof}
The duality in (1) is
\cite[Corollary~IV.3.5.1]{ekedahl1}. For the range: \(W\Omega_X^i=0\) for \(i<0\), and
\(H^j(X,W\Omega_X^i)=0\) for \(j\notin[0,n]\) by Grothendieck vanishing on the
\(n\)-dimensional space \(X\). Hence \(T^{i,j}(X)=0\) unless \(0\le i\) and
\(0\le j\le n\). If \(i>n-2\) then \(n-i-2<0\), and if \(j<2\) then \(n-j+2>n\). In
either case the dual domino number vanishes, which gives the range.

For (2), set \(m:=\tau(i)\). At \((i,j)\), the modification replaces \(d\) by
\(F^m d\), with \(F^{-r}d=dV^r\). The identities
\(d_\tau V^q=dV^{\,q-m}\) and \(F^qd_\tau=F^{\,q+m}d\), for \(q\) sufficiently
large, show that the stable kernel and image systems defining the domino are
unchanged up to reindexing. Thus the modified domino is \(\sigma_*^{s_i(\tau)}\bigl(U_X^{i,j}(m\delta_0)\bigr)\), and in particular has dimension \(T^{i,j}(X)\). By survival of the core
(Proposition~\ref{prop:survival-core}), \(d_r\) factors through the degree-\(0\)
component of the source domino and has image in the degree-\(1\) component of the
domino at \((i+r-1,j-r+1)\). Thus a nonzero \(d_r\) forces both modified dominoes
to be nonzero, and hence forces the two displayed domino numbers to be nonzero.
\end{proof}

For a threefold ($n = 3$) the domino range is the $2\times 2$ block $\{i \in \{0,1\},\; j \in \{2,3\}\}$, boxed below on the $E_1$-page, with the differentials permitted by the proposition drawn at the corresponding positions:
\[
\scalebox{0.9}{%
\begin{tikzpicture}[baseline=(current bounding box.center)]
  \foreach \i in {0,1,2,3} {
    \foreach \j in {0,1,2,3} {
      \node (n\i\j) at (\i*1.75, \j*1.12) {$E_1^{\i,\j}$};
    }
  }
  \draw[->] (n02) -- (n12);
  \draw[->] (n03) -- (n13);
  \draw[->] (n12) -- (n22);
  \draw[->] (n13) -- (n23);
  \draw[->, dashed] (n03) -- (n22);
  \node[draw, very thick, inner sep=5pt, fit=(n02)(n12)(n03)(n13)] {};
\end{tikzpicture}%
}
\]
The boxed nodes form the domino range, and $T^{i,j} = 0$ outside the box. By part (2), a nonzero $d_r$ requires its source, and the position one column to the left of its target, to lie in the box. The four horizontal arrows are the $d_1$ differentials with source in the range. Their targets may exit the box by one column. The dashed arrow \(d_2\colon E_2^{0,3}\to E_2^{2,2}\) is the only higher differential not ruled out.

In any dimension, the same constraint prevents a higher differential from touching bidegree \((0,2)\), so \(E_\infty^{0,2}=E_2^{0,2}\) for the slope spectral sequence and each of its Nygaard modifications.

\section{The structure of higher dimensional dominoes}\label{sec:domino-structure}
Throughout this section we work in the abelian category \(\mathrm{Mod}_c(R)\) of
coherent \(R\)-modules over the fixed field \(k\). We study dominoes as objects of this category,
reducing classification questions to extension calculations between the
elementary dominoes \(U_j\). The computational input is the Ext computation of
Proposition~\ref{prop:ext-full}, together with the presentation of extensions
by explicit pushouts in Construction~\ref{cons:Exy}. A domino written as a quotient
of \(\widehat R\) is regarded as a graded \(R\)-module by restriction of scalars.
Conversely, a coherent \(R\)-module \(N\) is complete and separated for
\(V^rN+dV^rN\) by Proposition~\ref{prop:coherence-criterion}, so its \(R\)-action
extends by continuity to a unique \(\widehat R\)-action.
We write
\(\underline{\operatorname{Ext}}^q_R(M,N):=\bigoplus_{n\in\mathbb Z}\operatorname{Ext}^q_{\mathrm{Mod}(R)}(M,N(n))\), a graded abelian group with the
shift \(n\) as its grading, and \(\operatorname{Ext}^q_R(M,N):=\underline{\operatorname{Ext}}^q_R(M,N)^0\) for its degree-zero part.

The section proceeds from dimension two upward. In \S\ref{sec:twodim-dominoes} we classify the \(2\)-dimensional indecomposable dominoes (Theorem~\ref{thm:2d-classification}). Such a domino is an extension of \(U_{j_1}\) by \(U_{j_2}\) with \(j_2-j_1\ge2\). Its class is a nonzero skew polynomial \(f(V)\in k_\sigma[V]_{\le j_2-j_1-2}\), and we denote the corresponding domino by \(U_{j_1,j_2;f}\). Two polynomials give isomorphic dominoes exactly when they differ by a frame change. In \S\ref{sec:domino-groupoids} we prove the analogue in arbitrary dimension: every domino with type sequence \(J\) is presented by a strictly upper triangular matrix \(f\) over \(k_\sigma[V]\), and isomorphisms are described by upper triangular changes of generators (Theorem~\ref{thm:fixed-type-groupoid}).

In \S\ref{sec:explicit-dominoes} we illustrate the presentation theorem with three
families: distinguished dominoes, two-block dominoes and their Kronecker model, and
three-dimensional dominoes with distinct types. For a type sequence
\(J=(j,j+2,\ldots,j+2n-2)\), the quotient
\(U_J=\widehat R/\widehat R(F^n,dV^{j-1})\) is the unique domino of type \(J\) with
maximal \(p\)-exponent \(n\) (Theorem~\ref{thm:distinguished}). These distinguished dominoes
reappear as the domino summands of the cyclic supergeneral family in
\S\ref{sec:supersingular}. The geometric applications in
\S\S\ref{sec:applications} and~\ref{sec:supersingular} use these results through
the two unipotent realizations of a domino \(U\), the formal realization
\(\widehat G(U)\) and the syntomic realization \(G^{\mathrm{perf}}(U)\), developed
in Appendix~\ref{app:unipotent}.

\subsection{$2$-dimensional dominoes}\label{sec:twodim-dominoes}

The first nontrivial extension case is dimension two: for type sequence
\((j_1,j_2)\), the associated graded object is \(U_{j_1}\oplus U_{j_2}\), so
classification amounts to computing extensions between elementary dominoes modulo the
residual automorphisms of the graded pieces. We compute \(\operatorname{Ext}^1\) as a
finite-dimensional \(k\)-vector space of Frobenius-skew polynomials in \(V\). The
nonzero classes are orbits under a Frobenius-skew frame-change action.

To compute the Ext groups we use the Illusie--Raynaud resolution of \(R_n\),
which makes the truncation functor \(R_n\otimes_R^{\mathbf L}(-)\) of
Remark~\ref{rem:truncation} explicit.

\begin{proposition}[Illusie--Raynaud {\cite[Proposition~I.3.2]{IR83}}]
\label{prop:IR-resolution-Rn}
\label{prop:Rn-free-resolution}
Let \(R_n:=R/(V^nR+dV^nR)\), regarded as a right graded \(R\)-module and as a left
\(W_n[d]\)-module. Then \(R_n\) admits the right \(R\)-free resolution
\[
  0\longrightarrow R(-1)\xrightarrow{\ \binom{F^n}{-F^nd}\ }R(-1)\oplus R
  \xrightarrow{\,(dV^n,\;V^n)\,}R\longrightarrow R_n\longrightarrow0,
\]
where both maps are left multiplication, applied to column vectors.
In particular, for any left graded \(R\)-module \(M\), the object
\(R_n\otimes_R^{\mathbf L}M\) is represented in \(D(W_n[d])\) by the complex
\[
  M(-1)\xrightarrow{\ \binom{F^n}{-F^nd}\ }M(-1)\oplus M\xrightarrow{\,(dV^n,\;V^n)\,}M,
\]
placed in cohomological degrees \(-2,-1,0\).
\end{proposition}

\begin{proposition}[Ekedahl {\cite[Corollary~III.1.5.4(iii)]{ekedahl2}}]
\label{prop:ekedahl-comparison}
For \(N\in D^b_c(R)\) there is a natural isomorphism
\(\mathbf R\!\operatorname{Hom}_R(U_0,N)\simeq\bigl(R_1\otimes_R^{\mathbf L}N\bigr)[-2](1)\)
in \(D(k[d])\). In particular,
\(\underline{\operatorname{Ext}}^q_R(U_0,N)\simeq H^{q-2}\bigl(R_1\otimes_R^{\mathbf L}N\bigr)(1)\)
as graded \(k[d]\)-modules.
\end{proposition}

Applying the exact autoequivalence \((-j_1\delta_0)\) of
Definition~\ref{def:autoequivalence} gives
\[
  \underline{\operatorname{Ext}}^i_R(U_{j_1},U_{j_2})
  \simeq
  \underline{\operatorname{Ext}}^i_R(U_0,U_{j_2-j_1}).
\]
It therefore remains to compute the groups with first argument \(U_0\).
By Proposition~\ref{prop:ekedahl-comparison}, these are the cohomology groups
\[
  \underline{\operatorname{Ext}}^i_R(U_0,U_j)
  \simeq
  H^{i-2}\bigl(R_1\otimes_R^{\mathbf L}U_j\bigr)(1).
\]
Tensoring the right \(R\)-free resolution of \(R_1\) in
Proposition~\ref{prop:IR-resolution-Rn} with \(U_j\) gives the three-term
complex used below. Its \(k[d]\)-module structure is induced by the left
\(k[d]\)-action on the bimodule \(R_1\). We now compute its cohomology.
Cf.\ \cite[Corollary~I.3.7]{IR83} for the underlying graded \(k\)-vector
spaces.

\begin{proposition}[Ext groups between elementary dominoes]\label{prop:ext-full}
For every \(j\in\mathbb Z\), the graded \(k[d]\)-modules
\(\underline{\operatorname{Ext}}^i_R(U_0,U_j)\) vanish for \(i\notin\{0,1,2\}\), and for
\(i=0,1,2\) are as follows:
\[
\begin{array}{c|ccc}
 & \underline{\operatorname{Ext}}^0_R(U_0,U_j)
 & \underline{\operatorname{Ext}}^1_R(U_0,U_j)
 & \underline{\operatorname{Ext}}^2_R(U_0,U_j)\\
\hline
j\ge2 & k[d]\oplus k^{\oplus j} & k^{\oplus(j-1)} & k(1)\\
j=1 & k[d]\oplus k & 0 & k(1)\\
j=0 & k[d] & 0 & k[d](1)\\
j=-1 & k(-1) & 0 & k[d](1)\oplus k\\
j\le-2 & k(-1) & k^{\oplus(-j-1)} & k[d](1)\oplus k^{\oplus(-j)}
\end{array}
\]
\end{proposition}

\begin{proof}
Proposition~\ref{prop:IR-resolution-Rn} with \(n=1\) represents
\(C_j:=R_1\otimes_R^{\mathbf L}U_j\) in cohomological degrees \(-2,-1,0\) by
the rows of the commutative diagram
\[
\begin{tikzcd}[column sep=large, ampersand replacement=\&]
U_j(-1) \arrow[r,"{\binom{F}{-Fd}}"] \arrow[d,"-d"'] \&
U_j(-1)\oplus U_j \arrow[r,"{(dV,\,V)}"] \arrow[d,"{\begin{psmallmatrix}0&1\\0&0\end{psmallmatrix}}"] \&
U_j \arrow[d,"d"] \\
U_j(-1) \arrow[r,"{\binom{F}{-Fd}}"'] \&
U_j(-1)\oplus U_j \arrow[r,"{(dV,\,V)}"'] \&
U_j
\end{tikzcd}
\]
All maps act on columns from the left. The vertical maps represent the left
\(k[d]\)-action of the bimodule structure. They have internal degree one and
square to zero.

Forgetting the \(d\)-action, the complex splits, as a complex of graded
\(k\)-modules, into the four two-term pieces
\[
  [(U_j)^1(-1)\xrightarrow{\,F\,}(U_j)^1(-1)],\quad
  [(U_j)^0(-1)\xrightarrow{\,Fd\,}(U_j)^1],\quad
  [(U_j)^0(-1)\xrightarrow{\,dV\,}(U_j)^1],\quad
  [(U_j)^0\xrightarrow{\,V\,}(U_j)^0],
\]
the first two in degrees \((-2,-1)\), the last two in \((-1,0)\). In the
topological bases \(e_n=V^n\) and \(\eta_n=dV^n\) of
Definition~\ref{def:coherent}, the four operators send
\(\eta_n\mapsto\eta_{n-1}\), \(e_n\mapsto\eta_{n-1}\),
\(e_n\mapsto\eta_{n+1}\), and \(e_n\mapsto e_{n+1}\), a basis element mapping
to \(0\) when its target index lies outside the appropriate range
(\(n\ge j\) for \(\eta_n\), and \(n\ge0\) for \(e_n\)). Hence \(F\) is
surjective with kernel \(k\eta_j\); \(V\) is injective with cokernel
\(ke_0\); \(Fd\) has kernel \(k\langle e_0,\dots,e_j\rangle\) and cokernel
\(k\langle\eta_j,\dots,\eta_{-2}\rangle\); and \(dV\) has kernel
\(k\langle e_0,\dots,e_{j-2}\rangle\) and cokernel
\(k\langle\eta_j,\dots,\eta_0\rangle\), each span read as \(0\) when its
index range is empty.

On cohomology, \(H^{-2}=\ker(Fd)\oplus\ker(F)\),
\(H^{-1}=\ker(dV)\oplus\operatorname{coker}(Fd)\), and
\(H^{0}=\operatorname{coker}(V)\oplus\operatorname{coker}(dV)\). The
\(d\)-action pairs \(e_j\) with \(-\eta_j\) in \(H^{-2}\) and \(e_0\) with
\(\eta_0\) in \(H^{0}\), each pair spanning a free \(k[d]\)-module, and
annihilates every other basis class. The only case that is not immediate is a
representative \((0,\eta_m)\) of \(\operatorname{coker}(Fd)\), sent by the
middle vertical map to \((\eta_m,0)\), a boundary since \(F\) is surjective.
Twisting via
\(\underline{\operatorname{Ext}}^i_R(U_0,U_j)=H^{i-2}(C_j)(1)\) gives the
table. The remaining \(\underline{\operatorname{Ext}}^i\) vanish because
\(C_j\) is concentrated in degrees \(-2,-1,0\).
\end{proof}

To realize the preceding Ext classes by explicit pushouts, we now use a
completed left-module resolution of \(U_0\). This has a different role from
the right-module resolution of \(R_1\) used above to compute derived Ext.

\begin{lemma}[Ekedahl {\cite[Lemma~III.1.5]{ekedahl2}}]
\label{lem:U0-resolution}
\label{lem:U0-completed-resolution}
The elementary domino $U_0$ admits the following $3$-term free resolution
by left \(\widehat{R}\)-modules:
\[
  0 \to \widehat{R}(-1)
  \xrightarrow{\cdot(dV,\,-V)}
  \widehat{R} \oplus \widehat{R}(-1)
  \xrightarrow{\ \cdot\binom{F}{Fd}\ }
  \widehat{R}
  \to U_0 \to 0,
\]
where both maps are right multiplication.
\end{lemma}

\begin{construction}[Pushout extensions of \(U_0\)]\label{cons:Exy}
Let \(N\) be a coherent \(R\)-module, regarded as an \(\widehat R\)-module. By
Lemma~\ref{lem:U0-resolution}, the kernel of \(\widehat R\to U_0\) is generated
by \(F\) and \(Fd\) subject to the single relation \(dV\cdot F=V\cdot Fd\), so
a degree-zero map \(\alpha:\ker(\widehat R\to U_0)\to N\) is a pair
\((\alpha(F),\alpha(Fd))=(x,y)\in N^0\oplus N^1\) with \(dVx=Vy\). Define the
extension \(N_\alpha\) by the pushout
\[
\begin{tikzcd}[column sep=large]
0 \arrow[r] & \ker(\widehat R\to U_0) \arrow[r] \arrow[d,"\alpha"'] & \widehat R \arrow[r] \arrow[d] & U_0 \arrow[r] \arrow[d,equal] & 0\\
0 \arrow[r] & N \arrow[r] & N_\alpha \arrow[r] & U_0 \arrow[r] & 0 ;
\end{tikzcd}
\]
concretely, writing \(\widetilde e\) for the image of \(1\in\widehat R\),
\[
  N_\alpha=
  \frac{N\oplus\widehat R\widetilde e}
       {\widehat R(F\widetilde e-x)+\widehat R(Fd\widetilde e-y)}.
\]
Any map \(\beta:\widehat R\to N\) is determined by \(m=\beta(1)\in N^0\), and
its restriction to \(\ker(\widehat R\to U_0)\) is the pair \((Fm,Fdm)\).
\end{construction}

\begin{proposition}[The extension groupoid of \(U_0\)]
\label{prop:RHom-U0}
\label{prop:completed-discrete-ext}
Let \(N\) be a coherent \(R\)-module. Every extension of \(U_0\) by \(N\) is
isomorphic to \(N_\alpha\) for some \(\alpha\), and
\(N_\alpha\simeq N_{\alpha'}\) as extensions if and only if
\(\alpha-\alpha'\) extends to a map \(\beta:\widehat R\to N\). In particular
\(\operatorname{Ext}^1_R(U_0,N)\) consists of the degree-zero pairs \((x,y)\)
with \(dVx=Vy\), modulo \(\{(Fm,Fdm):m\in N^0\}\), and the automorphisms of
an extension inducing the identity on \(N\) and \(U_0\) form
\(\operatorname{Hom}_R(U_0,N)\).
\end{proposition}

\begin{proof}
An extension \(E\) is coherent, hence an \(\widehat R\)-module. A degree-zero
lift \(\widetilde e\in E\) of the generator of \(U_0\) defines
\(\widehat R\to E\), \(1\mapsto\widetilde e\), which restricts on
\(\ker(\widehat R\to U_0)\) to a map \(\alpha\) with \(E\simeq N_\alpha\).
An isomorphism \(N_\alpha\to N_{\alpha'}\) of extensions sends
\(\widetilde e\) to \(\widetilde e'+m\) with \(m\in N^0\), and exists exactly
when \(\alpha-\alpha'\) is the restriction of \(\beta:1\mapsto m\), that is,
\((x-x',\,y-y')=(Fm,\,Fdm)\). For \(\alpha=\alpha'\) the condition reads
\(Fm=Fdm=0\), which says that \(m\) is a map \(U_0\to N\).
\end{proof}

\paragraph{Polynomial normal form.}
For \(j\ge2\), let \(u_j\) be the image of \(1\) in \(U_j^0\). Since \(F\)
vanishes on \(U_j^0\) and \(Fd:U_j^0\to U_j^1\) is surjective, every isomorphism
class of extensions of \(U_0\) by \(U_j\) has a unique representative
\(N_\alpha\) with \(\alpha(Fd)=0\). The relation then reads
\(dV\,\alpha(F)=0\), and
Proposition~\ref{prop:ext-full} identifies
\[
  \ker(dV:U_j^0\to U_j^1)
  =\bigoplus_{q=0}^{j-2}kV^qu_j
  \simeq k_\sigma[V]_{\le j-2},
\]
where \(k_\sigma[V]\) is the skew polynomial ring with \(Va=\sigma^{-1}(a)V\)
and \(k_\sigma[V]_{\le j-2}\) its polynomials of degree at most \(j-2\).
Thus \(\alpha(F)=[f](V)u_j\) for a unique \(f\in k_\sigma[V]_{\le j-2}\), where
\([f]\) denotes the coefficientwise Teichm\"uller lift, and the corresponding
extension is
\begin{equation}
\label{eq:polynomial-Yoneda-extension}
  U_{0,j;f}:=N_\alpha=
  \frac{U_j\oplus\widehat R\widetilde u_0}
       {\widehat R(F\widetilde u_0-[f](V)u_j)+
        \widehat R(Fd\widetilde u_0)}.
\end{equation}
For \(j_2-j_1=j\ge2\), define
\[
  U_{j_1,j_2;f}:=U_{0,j;f}(j_1\delta_0).
\]
Thus \(U_{j_1,j_2;f}\) has type \((j_1,j_2)\). When the type pair is fixed
or irrelevant, we abbreviate it to \(U_f\). The zero polynomial gives the
split extension. If \(f\ne0\), then
\(p\widetilde u_0=V[f](V)u_j\ne0\), while \(p^2U_{0,j;f}=0\), so
\(p\text{-}\!\exp(U_{0,j;f})=2\). Rescaling the quotient lift by
\(b_0\in k^\times\)
and the subobject generator by \(c_0\in k^\times\) sends
\[
  f(V)\longmapsto \sigma(b_0)\,f(V)\,c_0^{-1}.
\]

\begin{definition}[Indecomposable dominoes]\label{def:indecomposable-domino}
A domino \(U\) is \emph{indecomposable} if it is not isomorphic to a direct sum of two nonzero dominoes.
\end{definition}

\begin{remark*}
The category of coherent graded \(R\)-modules is abelian, hence idempotent
complete. Therefore a domino \(U\) is indecomposable if and only if
\(\operatorname{End}_R(U)\) has no idempotents other than \(0\) and \(1\). Moreover,
if \(U=U'\oplus U''\), functoriality of the canonical type filtration gives
\(\operatorname{gr}^{i}_{\mathrm{type}}U\simeq
\operatorname{gr}^{i}_{\mathrm{type}}U'\oplus
\operatorname{gr}^{i}_{\mathrm{type}}U''\) for every type \(i\). Thus a
direct-sum decomposition of \(U\) decomposes every type-graded block. When all
blocks have multiplicity one, it partitions them.
\end{remark*}

\begin{theorem}\label{thm:2d-classification}
Indecomposable \(2\)-dimensional dominoes over \(k\), up to isomorphism, are in
bijection with
\[
  \bigl\{\,(j_1,\,j_2,\,[f(V)])
  \;\big|\;
  j_1,j_2\in\mathbb Z,\;
  j_2-j_1\ge 2,\;
  [f(V)]\in\bigl(k_\sigma[V]_{\le j_2-j_1-2}\setminus\{0\}\bigr)\big/(k^\times\times k^\times)
  \,\bigr\},
\]
where \([f(V)]\) denotes the orbit under the frame-change action
\[
  f(V)\cdot(b_0,c_0)
  :=
  \sigma(b_0)\,f(V)\,c_0^{-1},
\]
computed in \(k_\sigma[V]\).
The representative \(f\) defines the normalized extension
\(U_{0,j_2-j_1;f}\). Transporting it back by the inverse Nygaard modification
gives \(U_{j_1,j_2;f}\).
Every domino in the list has \(p\text{-}\!\exp(U_{j_1,j_2;f})=2\).
\end{theorem}

\begin{proof}
If \(j_1=j_2\), any defining extension of \(U_{j_1}\) by itself splits by
Proposition~\ref{prop:ext-full}, so the domino is decomposable. We may
therefore assume \(j_1<j_2\). The canonical type filtration then gives an
extension \(0\to U_{j_2}\to U\to U_{j_1}\to0\), and every isomorphism
preserves it up to automorphisms of the two constituents. Put \(j=j_2-j_1\)
and apply the Nygaard modification carrying \(U_{j_1}\) to \(U_0\).

Proposition~\ref{prop:ext-full} and the polynomial normal form above show that
the extension group is zero for \(j\le1\), while for \(j\ge2\) its classes are
the polynomials \(f\in k_\sigma[V]_{\le j-2}\). The zero class is split. If
\(f\ne0\), then \(p\text{-}\!\exp(U_{j_1,j_2;f})=2\), so
\(U_{j_1,j_2;f}\) is indecomposable, since
a direct sum of two \(1\)-dimensional dominoes is killed by \(p\).

Finally, Proposition~\ref{prop:ext-full} gives
\(\operatorname{Aut}_R(U_j)=k^\times\), and the frame-change calculation above
identifies isomorphism classes with the stated orbits.
\end{proof}

\begin{example}[A nontrivial extension of $U_0$ by $U_2$]\label{ex:U0U2ext}
There is, up to isomorphism, a unique nontrivial extension of \(U_0\) by
\(U_2\): here the type gap is \(j=2\), so \(k_\sigma[V]_{\le0}=k\), and the frame change is
transitive on \(k^\times\), so \(f=1\) represents the only nonzero orbit. The
corresponding extension \eqref{eq:polynomial-Yoneda-extension} is
\(0\to U_2\to U_{0,2;1}\to U_0\to0\), where
\(U_{0,2;1}=\widehat{R}/\widehat{R}(F^2,Fd)\). Using the defining relations
and the standard normal form in \(\widehat R\) gives
\(U_{0,2;1}^0=W_{2,\sigma}[[V]]\oplus kF\) and
\(U_{0,2;1}^1=\prod_{i\ge1}W_2\,dV^i\oplus k\,d\), where \(W_{2,\sigma}[[V]]\) is the
skew power series ring over \(W_2\).
\end{example}

\begin{remark}[No ordinary projectivization]
\label{rem:no-projectivization}
Fix a pair \((j_1,j_2)\) and put \(j:=j_2-j_1\).
The orbit space of Theorem~\ref{thm:2d-classification} is in general \emph{not}
the projective space of the underlying \(k\)-vector space. On coefficients the frame-change relation is
\(a_i\mapsto\sigma(b_0)\,\sigma^{-i}(c_0^{-1})\,a_i\), so the \(i\)-th
coefficient has Frobenius weight \(\sigma^{-i}\). When
\(\sigma=\mathrm{id}\), for instance over \(\mathbf F_p\), all weights coincide
and the quotient is the ordinary \(\mathbb P^{j-2}(k)\).

Over an algebraically closed \(k\) the first discrepancy occurs at \(j=3\),
where \(f(V)=a_0+a_1V\) and the frame change sends \((a_0,a_1)\) to
\(\bigl(\sigma(b_0)a_0c_0^{-1},\,\sigma(b_0)a_1\sigma^{-1}(c_0^{-1})\bigr)\).
The vanishing of each coefficient is an isomorphism invariant, and each of the
three possibilities is a single orbit: a single nonzero coefficient is rescaled freely, and on
\(\{a_0a_1\ne0\}\) transitivity follows from the surjectivity of
\(u\mapsto\sigma^{-1}(u)/u\) on \(k^\times\). Thus
\(U_{j_1,j_1+3;1}\), \(U_{j_1,j_1+3;V}\), and
\(U_{j_1,j_1+3;1+V}\) are the three indecomposable dominoes of type
\((j_1,j_1+3)\): the
\(\mathbb P^1(k)\)-family of the untwisted problem collapses to three points.
\end{remark}

\begin{remark}[Formal and syntomic realizations]
\label{rem:two-unipotent-2d}
For this remark we take \(k=\bar k\) and use the two realization functors
constructed in Appendix~\ref{app:unipotent}. Every domino \(U\) has two unipotent realizations: a smooth
formal group \(\widehat G(U)\) attached to \(U^0\) by Cartier theory, and a
syntomic realization \(G^{\mathrm{perf}}(U)\), represented by the perfect
unipotent group of \(F\)-fixed points in \(U^1\). For a \(2\)-dimensional
domino \(U_{j_1,j_2;f}\), abbreviated here to \(U_f\), the two realizations
depend on \(f\) through the completion and localization of the parameter ring,
\[
  k_\sigma[[V]]
  \xleftarrow{\ \mathrm{completion}\ }
  k_\sigma[V]
  \xrightarrow{\ \mathrm{localization}\ }
  k_\sigma[V^{\pm1}],
\]
and they detect different quotients of the extension data
(Theorem~\ref{thm:2d-unipotent-realizations}). The formal side records only
\(\operatorname{ord}_V(f)\), and \(\widehat G(U_f)\simeq\widehat W_2\) if and
only if \(f(0)\neq0\). The syntomic side determines \(f\) up to a two-sided
unit orbit in \(k_\sigma[V^{\pm1}]\), and
\(G^{\mathrm{perf}}(U_f)\simeq W_2^{\mathrm{perf}}\) if and only if \(f\) is a
nonzero monomial. Neither realization alone recovers \(U_f\). Their common
invariant is the isogeny partition of Theorem~\ref{thm:isogeny-partition},
\((1,1)\) for the split class and \((2)\) for every nonsplit class of this
type.
\end{remark}

\subsection{Dominoes of fixed type}\label{sec:domino-groupoids}

Fix a type sequence \(J=(j_1\le\cdots\le j_n)\). Let \(\Domgpd_J\) be the
groupoid whose objects are the dominoes \(U\) with \(J(U)=J\), and whose
morphisms are the isomorphisms in \(\mathrm{Mod}_c(R)\). The classification
problem has two parts.
First, the type filtration fixes the graded pieces, so one must record how
these pieces are successively extended. Second, one must remove the choices of
generators of the graded pieces and their lifts. We encode the first datum by
one upper triangular matrix \(f=(f_{\alpha\beta})\), and the second by upper
triangular changes of generators. Theorem~\ref{thm:fixed-type-groupoid} makes
this description precise.

We first identify the graded pieces. Equal type factors introduce no
additional extension data:

\begin{lemma}[Pure type dominoes split]\label{lem:pure-type-split}
If all type factors of a domino \(U\) are equal to \(j\) and \(\dim U=m\), then
\(U\simeq U_j^{\oplus m}\) and \(\operatorname{Aut}_R(U_j^{\oplus m})\simeq GL_m(k)\).
\end{lemma}

\begin{proof}
Since \(j\) is the only type factor of \(U\), its canonical type filtration
has a single nonzero graded piece, namely
\(U\simeq\operatorname{gr}_{\mathrm{type}}^j(U)\). By
Proposition~\ref{prop:type-filtration}, this graded piece is a finite iterated
extension of \(m\) copies of \(U_j\).

We now show by induction on \(m\) that every such iterated extension splits.
For \(m>1\), its last step is an exact sequence
\(0\to U'\to U\to U_j\to0\), where \(U'\) is an iterated extension of
\(m-1\) copies of \(U_j\). By
induction, \(U'\simeq U_j^{\oplus(m-1)}\). The Nygaard modification carrying
\(U_j\) to \(U_0\), together with Proposition~\ref{prop:ext-full}, gives
\(\operatorname{Ext}^1_R(U_j,U')\simeq
\operatorname{Ext}^1_R(U_j,U_j)^{\oplus(m-1)}=0\), so the sequence splits.

Finally, the degree-zero endomorphisms in
Proposition~\ref{prop:ext-full} give \(\operatorname{End}_R(U_j)=k\). Hence
\(\operatorname{End}_R(U_j^{\oplus m})\simeq\operatorname{Mat}_m(k)\), whose
unit group is \(GL_m(k)\).
\end{proof}

To treat the remaining extensions between distinct type values, we add the
graded blocks one at a time and require coordinates for an extension whose
quotient may have any type. Proposition~\ref{prop:completed-discrete-ext}
provides them for \(U_0\). The following lemma is its Nygaard translate to
\(U_a\).

\begin{lemma}[Extension coordinates in arbitrary type]
\label{lem:arbitrary-type-extension-coordinates}
Let \(a\in\mathbb Z\) and let \(N\) be a coherent \(R\)-module. Extensions of
\(U_a\) by \(N\) are represented by pairs
\((x,y)\in N^0\oplus N^1\) satisfying \(dV^{a+1}x=Vy\), modulo
\((Fm,dV^{a-1}m)\) for \(m\in N^0\).
The pair \((x,y)\) gives the pushout presentation obtained by adjoining a
degree-zero lift \(\widetilde e_a\) with \(F\widetilde e_a=x\) and
\(dV^{a-1}\widetilde e_a=y\).
The same description applies componentwise to extensions of
\(U_a^{\oplus m}\).
\end{lemma}

\begin{proof}
The kernel of \(\widehat R\to U_a\) is \(\widehat RF+\widehat R\,dV^{a-1}\) by
definition, and transporting the resolution of
Lemma~\ref{lem:U0-completed-resolution} along the exact autoequivalence
\((a\delta_0)\), which carries \(U_0\) to \(U_a\), shows that its single
relation is \(dV^{a+1}\cdot F=V\cdot dV^{a-1}\). A degree-zero map from this
kernel to \(N\) is therefore the displayed pair, and, as in
Construction~\ref{cons:Exy}, changing the lift by \(m\in N^0\) adds
\((Fm,\,dV^{a-1}m)\).
\end{proof}

To iterate these extensions, let \(i_1<\cdots<i_s\) be the distinct types in \(J\), with
multiplicities \(m_\alpha\), and put
\(B_\alpha:=U_{i_\alpha}^{\oplus m_\alpha}\). By
Lemma~\ref{lem:pure-type-split}, the associated graded object is
\(\bigoplus_\alpha B_\alpha\). The skew polynomial ring \(k_\sigma[V]\)
from the two-dimensional calculation remains the natural coefficient ring:
powers of \(V\) record the possible maps and extensions between elementary
dominoes.

For each block, fix an ordered column of degree-zero generators
\(e_\alpha=(e_{\alpha,1},\ldots,e_{\alpha,m_\alpha})^{\mathsf t}\).
Proposition~\ref{prop:ext-full} gives
\[
  \operatorname{Hom}_R(B_\alpha,B_\beta)\simeq
  \operatorname{Mat}_{m_\alpha\times m_\beta}
  \bigl(k_\sigma[V]_{\le i_\beta-i_\alpha}\bigr)
\]
for all \(\alpha,\beta\), which vanishes when \(\alpha>\beta\). For
\(\alpha<\beta\) it further gives
\[
  \operatorname{Ext}^1_R(B_\alpha,B_\beta)\simeq
  \operatorname{Mat}_{m_\alpha\times m_\beta}
  \bigl(k_\sigma[V]_{\le i_\beta-i_\alpha-2}\bigr),
\]
and the degree-zero groups \(\operatorname{Ext}^q_R(B_\alpha,B_\beta)\) then
vanish for \(q\ge2\). Accordingly, only the Hom and \(\operatorname{Ext}^1\)
coordinates enter the presentation below.

A homomorphism \(B_\alpha\to B_\beta\) has a unique coordinate matrix
\(g_{\alpha\beta}\) characterized by
\(e_\alpha\mapsto g_{\alpha\beta}(V)e_\beta\). Under the second
identification, a matrix \(f_{\alpha\beta}\) represents the extension
obtained by adjoining lifts \(\widetilde e_\alpha\) subject to
\(F\widetilde e_\alpha=[f_{\alpha\beta}](V)e_\beta\) and
\(dV^{i_\alpha-1}\widetilde e_\alpha=0\). Here brackets denote
coefficientwise Teichm\"uller lift. With this column convention, composable
coordinate matrices occur in the same order as the maps. For
\(g=\sum_r g_rV^r\), write \(\sigma(g):=\sum_r\sigma(g_r)V^r\).

To assemble the blocks, we need only know how these matrices behave under
composition and under pushout or pullback of an extension. In the chosen
coordinates the rules take the following simple form.

\begin{remark}[Coordinate rules]\label{rem:skew-composition}
Let \(\alpha<\beta<\gamma\). In the coordinates above, composition, pushout,
and pullback are represented by
\(g_{\alpha\beta}g_{\beta\gamma}\),
\(f_{\alpha\beta}g_{\beta\gamma}\), and
\(\sigma(g_{\alpha\beta})f_{\beta\gamma}\), respectively. Changing the
generators by \(c_\alpha\in GL_{m_\alpha}(k)\) and
\(c_\beta\in GL_{m_\beta}(k)\) sends \(g_{\alpha\beta}\) to
\(c_\alpha g_{\alpha\beta}c_\beta^{-1}\) and \(f_{\alpha\beta}\) to
\(\sigma(c_\alpha)f_{\alpha\beta}c_\beta^{-1}\).

For the pullback formula, lift the source generators in an extension represented
by \(f_{\beta\gamma}\) to
\([g_{\alpha\beta}](V)\widetilde e_\beta\). Applying \(F\) gives
\([\sigma(g_{\alpha\beta})](V)[f_{\beta\gamma}](V)e_\gamma\), which is the
claimed coordinate. Any discrepancy from coefficientwise Teichm\"uller lift
is divisible by \(p\) and vanishes in \(B_\gamma\). The second coordinate is
normalized to zero using the surjectivity of
\(Fd:U_{i_\gamma-i_\alpha}^0\to U_{i_\gamma-i_\alpha}^1\). The other formulas
and all degree bounds follow directly from the definitions.
\end{remark}

Guided by the coordinate rules, we assemble the successive extensions into one
presentation by choosing generators \(\widetilde e_\alpha\) lifting the chosen
generators of the graded blocks and prescribing each \(F\widetilde e_\alpha\)
in terms of generators of larger type. The coefficients form a strictly upper
triangular block matrix \(f=(f_{\alpha\beta})\). Set
\[
  \mathfrak n_1:=
  \prod_{\alpha<\beta}
  \operatorname{Mat}_{m_\alpha\times m_\beta}
  \bigl(k_\sigma[V]_{\le i_\beta-i_\alpha-2}\bigr).
\]
An element \(f\in\mathfrak n_1\) is called a \emph{presentation matrix}. Its
\((\alpha,\beta)\)-block is necessarily zero when
\(i_\beta-i_\alpha\le1\).

Changing the generators of the graded blocks or their lifts replaces the
\(\widetilde e_\alpha\) by upper triangular combinations. The Hom coordinates
identify
\(\mathfrak n_0:=\operatorname{End}_R(\bigoplus_\alpha B_\alpha)\) with
the algebra of upper triangular block matrices \(g=(g_{\alpha\beta})\) whose
\((\alpha,\beta)\)-block lies in
\(\operatorname{Mat}_{m_\alpha\times m_\beta}
(k_\sigma[V]_{\le i_\beta-i_\alpha})\). A
\emph{change-of-generators matrix} is a unit
\(g\in\mathfrak n_0^\times\), equivalently one with
\(g_{\alpha\alpha}\in GL_{m_\alpha}(k)\). Its
diagonal blocks change the generators of the graded blocks, and its
off-diagonal blocks change their lifts. These are exactly the two choices that
should not affect the isomorphism class.

Using the convention \(dV^q:=F^{-q}d\) for \(q<0\), in particular
\(dV^{-1}=Fd\), we associate to \(f\in\mathfrak n_1\) the two-term graded
\(R\)-module
\begin{equation}\label{eq:MX-presentation}
  M_{f}:=
  \Bigl(\bigoplus_{\alpha=1}^s\widehat R^{\oplus m_\alpha}\widetilde e_\alpha\Bigr)
  \Big/
  \sum_{\alpha=1}^s\Bigl(\widehat R\bigl(F\widetilde e_\alpha
  -\textstyle\sum_{\beta>\alpha}[f_{\alpha\beta}](V)\,\widetilde e_\beta\bigr)
  +\widehat R\,dV^{i_\alpha-1}\widetilde e_\alpha\Bigr),
\end{equation}
where \(\widetilde e_\alpha\) denotes a column of \(m_\alpha\) degree-zero
generators and \([f_{\alpha\beta}]\) is the coefficientwise Teichm\"uller lift,
as in \eqref{eq:polynomial-Yoneda-extension}. For \(f=0\) this is
\(B_1\oplus\cdots\oplus B_s\).

For \(1\le r\le s\), let \(T_r\subseteq M_f\) be the submodule generated by
\(\widetilde e_r,\ldots,\widetilde e_s\), and set \(T_{s+1}=0\). The letter
\(T\) refers to the
triangular form of the presentation. After \(T_{r+1}\) has been constructed,
adjoining \(\widetilde e_r\) is a single extension step. This recursion gives
both the type filtration and the normal form needed below.

\begin{lemma}[Triangular normal form]\label{lem:kernel-normal-form}
Assume \(i_1\ge0\) and let \(f\in\mathfrak n_1\). Then, for every \(r\),
\(0\to T_{r+1}\to T_r\to B_r\to0\) is exact. Thus \(M_f\) is a domino of
type \(J\), with canonical type filtration given by the \(T_r\).

Every \(z\in M_f^0\) is uniquely
\(\sum_\alpha P_\alpha(V)\widetilde e_\alpha\), with
\(P_\alpha\in\operatorname{Mat}_{1\times m_\alpha}(k[[V]])\), and every
element of \(M_f^1\) is uniquely a convergent \(k\)-linear combination of the
entries of \(dV^n\widetilde e_\alpha\), for \(n\ge i_\alpha\). For
\(t\ge-1\), one has \(dV^tz=0\) exactly when
\(\deg P_\alpha\le i_\alpha-1-t\) for every \(\alpha\).
\end{lemma}

\begin{proof}
Use descending induction on \(r\). The case \(T_s=B_s\) is elementary.
Given \(T_{r+1}\), put
\(x_r:=\sum_{\beta>r}[f_{r\beta}](V)\widetilde e_\beta\). The degree bound
and the inductive kernel criterion give \(dV^{i_r+1}x_r=0\), so \((x_r,0)\)
defines, by Lemma~\ref{lem:arbitrary-type-extension-coordinates}, an extension
of \(B_r\) by \(T_{r+1}\) whose pushout is \(T_r\).
Iterating these pushouts reproduces exactly the presentation
\eqref{eq:MX-presentation}, so no additional relations occur.
Lifting and projecting the elementary normal forms proves existence,
uniqueness, and the kernel criterion. The latter also characterizes \(T_r\)
as the maximal submodule with \(\operatorname{coker}(dV^{i_r})=0\), hence as
the corresponding step of the canonical type filtration by
Proposition~\ref{prop:type-filtration}.
\end{proof}

The assumption \(i_1\ge0\) in the normal-form argument only avoids negative
exponents. A common Nygaard modification shifts all types and commutes with
the constructions above, while leaving the differences
\(i_\beta-i_\alpha\) unchanged. We may therefore normalize
\(i_1=0\) and transport the result back. We can now classify the
presentations by their changes of generators.

\begin{theorem}[Fixed-type presentations]\label{thm:fixed-type-groupoid}
Let \(J\) be a type sequence with blocks \(B_1,\ldots,B_s\) as above.
The assignment \(f\mapsto M_f\) defines an equivalence of groupoids from the
following groupoid to \(\Domgpd_J\). Its objects are the matrices
\(f\in\mathfrak n_1\), and a morphism
\(f\to f'\) is a matrix \(g\in\mathfrak n_0^\times\) for which the
substitution
\(\widetilde e_\alpha\mapsto
\sum_{\beta\ge\alpha}[g_{\alpha\beta}](V)\widetilde e'_\beta\) respects the
defining \(F\)-relations. Equivalently,
\[
  \bigl([\sigma(g)][f']-[f][g]\bigr)\widetilde e'=0
  \qquad\text{in }M_{f'}.
\]
Composition is ordinary composition of these substitutions.
\end{theorem}

\begin{proof}
By the preceding Nygaard-modification argument, we may assume \(i_1=0\).
Lemma~\ref{lem:kernel-normal-form} shows that \(f\mapsto M_f\) is
well-defined on objects.

Let \(M\) be a domino of type \(J\), choose identifications of its type-graded
blocks with the \(B_r\), and write
\(M=M_1\supset\cdots\supset M_s\supset M_{s+1}=0\) for its canonical type
filtration. Suppose inductively that \(M_{r+1}\simeq T_{r+1}\). By
Lemma~\ref{lem:arbitrary-type-extension-coordinates}, the extension of \(B_r\)
by \(T_{r+1}\) is represented by a pair \((x,y)\). The degree-one normal form
allows us to change the lifts so that \(y=0\). The relation in
Lemma~\ref{lem:arbitrary-type-extension-coordinates} then gives
\(dV^{i_r+1}x=0\), and the kernel criterion writes \(x\) uniquely as the
\(r\)-th row of a matrix in \(\mathfrak n_1\).
Thus \(M_r\simeq T_r\), and descending induction gives \(M\simeq M_f\).

Now let \(\varphi:M_f\to M_{f'}\) be an isomorphism. It preserves the
canonical type filtration, and the kernel criterion writes its values on the
generators uniquely as an upper triangular matrix \(g\). The induced maps on
the graded blocks are invertible, so \(g\in\mathfrak n_0^\times\). Conversely,
the degree bounds on such a \(g\) ensure the \(dV\)-relations, while the
displayed equality in the theorem is precisely the remaining \(F\)-relation.
The resulting map preserves the finite filtrations defined by the \(T_r\) and
induces \(g_{\alpha\alpha}\) on the graded blocks. Since these diagonal blocks are
invertible, the map is an isomorphism. Hence these substitutions are exactly
the isomorphisms \(M_f\xrightarrow{\sim}M_{f'}\).

The normal form also assigns a unique matrix to a composite substitution, so
composition agrees with composition of module isomorphisms. This proves the
equivalence.
\end{proof}

\begin{remark}
For \(J=(j_1,j_2)\), a presentation matrix is the single polynomial
\(f=f_{12}\), and \(M_f=U_{j_1,j_2;f}\). In block form, the three matrices
entering the morphism condition are
\[
  (f_{\alpha\beta})=\begin{pmatrix}0&f\\0&0\end{pmatrix},\qquad
  (f'_{\alpha\beta})=\begin{pmatrix}0&f'\\0&0\end{pmatrix},\qquad
  g=\begin{pmatrix}c_1&h\\0&c_2\end{pmatrix},
\]
where \(c_1,c_2\in k^\times\) and
\(h\in k_\sigma[V]_{\le j_2-j_1}\). The morphism condition in the theorem
then reduces to
\[
  \bigl([\sigma(c_1)][f']-[f][c_2]\bigr)\widetilde e'_2=0
  \qquad\text{in }M_{f'}.
\]
The terminal block \(B_2=U_{j_2}\) is killed by \(p\), and the elements
\(V^q e'_2\), \(q\ge0\), are topologically independent. Thus the displayed
condition is equivalent to \(\sigma(c_1)f'=fc_2\) in \(k_\sigma[V]\), or
\[
  f'=\sigma(c_1)^{-1}fc_2.
\]
This is the frame-change formula of
Theorem~\ref{thm:2d-classification}, with \(b_0=c_1^{-1}\) and
\(c_0=c_2^{-1}\). The off-diagonal entry \(h\) does not enter the formula:
it changes only the lift of the first graded block, and records
\(\operatorname{Hom}_R(U_{j_1},U_{j_2})\). In higher dimension, analogous
off-diagonal terms can contribute through later blocks, which is why the
general condition is stated as an equality in \(M_{f'}\).
\end{remark}

\subsection{Examples of fixed-type dominoes}
\label{sec:explicit-dominoes}

We now illustrate Theorem~\ref{thm:fixed-type-groupoid} in three families. The
first isolates a rigid class of maximal \(p\)-exponent, the second explains the
complexity of the general fixed-type problem, and the third gives a complete
worked classification in the smallest dimension where indecomposability and
\(p\)-exponent can diverge.

\subsubsection*{\normalfont\bfseries Distinguished dominoes}

For the type sequence \(J=(j,j+2,\ldots,j+2n-2)\), maximal
\(p\)-exponent singles out one isomorphism class. Distinguished dominoes of this
form occur as summands in the cyclic supergeneral family of
Theorem~\ref{thm:supergeneral}.

\begin{theorem}[Distinguished dominoes]
\label{thm:distinguished}
For \(j\in\mathbb Z\), \(n\ge1\), and \(J=(j,j+2,\ldots,j+2n-2)\), the module
\(U_J:=\widehat R/\widehat R(F^n,dV^{j-1})\) is, up to isomorphism, the unique
domino of type \(J\) having \(p\)-exponent \(n\). It is indecomposable and
\(\dim U_J=n\). We call it the \emph{distinguished domino} of type \(J\).
\end{theorem}

\begin{proof}
Nygaard modification reduces to \(j=0\). Put \(B_r=U_{2(r-1)}\),
\(U^{(n)}=\widehat R/\widehat R(F^n,Fd)\), and
\(T_r=\widehat R F^{r-1}u\) for \(1\le r\le n\), with \(T_{n+1}=0\), where
\(u\) is the image of \(1\). The identity
\(dV^{2r-3}F^{r-1}=p^{r-1}dV^{r-2}=V^{r-1}Fd\) and the normal form in
\(\widehat R\) give \(T_r/T_{r+1}\simeq B_r\), with adjacent class \(1\), and
\(p^nu=0\ne p^{n-1}u\). Thus \(U^{(n)}\) has the required type and exponent.
Write \(q_r:T_r\twoheadrightarrow B_r\) for the quotient.

For uniqueness, argue by induction on \(n\), the case \(n=1\) being clear. Choose
generators on the graded blocks of a maximal-exponent domino \(M\), and write
\(0=M_{n+1}\subset M_n\subset\cdots\subset M_1=M\), with
\(M_r/M_{r+1}\simeq B_r\), and let
\(a_r\in\operatorname{Ext}^1_R(B_r,B_{r+1})\simeq k\) be the adjacent classes.
Semilinearity shows that the successive component of \(F^{n-1}\) is
\(\sigma^{n-2}(a_1)\sigma^{n-3}(a_2)\cdots a_{n-1}\). Since the \(p\)-exponent
is defined on \(M^0\), maximal exponent gives \(p^{n-1}M^0\ne0\), hence
\(F^{n-1}M^0\ne0\) by \(p^{n-1}=V^{n-1}F^{n-1}\). As \(F\) raises the type
filtration by at least one step, this forces the displayed product, and therefore
every \(a_r\), to be nonzero. Changing generators by
scalars \(c_r\in k^\times\) sends
\(a_r\) to \(\sigma(c_r)a_rc_{r+1}^{-1}\). Taking \(c_1=1\) and
\(c_{r+1}=\sigma(c_r)a_r\) normalizes all adjacent classes to \(1\).

The same argument gives \(p^{n-2}M_2^0\ne0\), while its \(n-1\) elementary
quotients give \(p^{n-1}M_2=0\). Thus \(p\text{-}\!\exp(M_2)=n-1\), and
induction identifies \(M_2\) with the standard subdomino \(T_2\).
It remains only to determine the
extension of \(B_1\) by \(T_2\) from its pushout to \(B_2\).

For every \(2\le r\le n\),
\(q_{r,*}:\operatorname{Ext}^1_R(B_1,T_r)\to\operatorname{Ext}^1_R(B_1,B_r)\)
is an isomorphism and \(\operatorname{Ext}^q_R(B_1,T_r)=0\) for \(q\ge2\). This
follows by descending induction from Proposition~\ref{prop:ext-full}: in the long
exact sequence of \(0\to T_{r+1}\to T_r\to B_r\to0\), the connecting map becomes
\(h\mapsto\sigma(h)\) on \(k_\sigma[V]_{\le2r-2}\) by
Remark~\ref{rem:skew-composition}, hence is bijective because \(k\) is perfect.
The remaining higher-Ext terms vanish by Proposition~\ref{prop:ext-full} and
induction, so exactness gives both assertions.

For \(r=2\), the extension defining \(M\) maps to a nonzero class in
\(\operatorname{Ext}^1_R(B_1,B_2)\simeq k\). Rescaling the generator of \(B_1\)
identifies this class with the class \(1\) of \(U^{(n)}\). The isomorphism
\(q_{2,*}\) then gives \(M\simeq U^{(n)}\), completing the induction.

Shifting back gives \(U_J\). Refining the type filtration of any dimension-\(m\)
domino into elementary quotients shows that it is killed by \(p^m\).
In a nontrivial decomposition of \(U_J\), every summand would therefore have
\(p\)-exponent less than \(n\), contrary to \(p\text{-}\!\exp(U_J)=n\).
\end{proof}

The unipotent realizations of the distinguished dominoes are Witt groups:
\(\widehat G(U_J)\simeq\widehat W_n\) and
\(G^{\mathrm{perf}}(U_J)\simeq W_n^{\mathrm{perf}}\)
(Corollary~\ref{cor:distinguished-props}). This is the form used for the
supergeneral family in \S\ref{sec:supersingular}.

\subsubsection*{\normalfont\bfseries Kronecker dominoes}

Dominoes with two distinct types already exhibit the difficulty of the general
classification problem. Let \(j_1<j_2\), \(B_1:=U_{j_1}^{\oplus m}\),
\(B_2:=U_{j_2}^{\oplus n}\), and \(j:=j_2-j_1\). If \(j\le1\), every extension splits.
If \(j\ge2\), a choice of generators represents an extension by a matrix
\[
  f(V)=A_0+A_1V+\cdots+A_{j-2}V^{j-2}
  \in \operatorname{Mat}_{m\times n}\bigl(k_\sigma[V]_{\le j-2}\bigr),
\]
and changing the generators of \(B_1\) and \(B_2\) by \(c_1\in GL_m(k)\) and
\(c_2\in GL_n(k)\), respectively, sends
\[
  f\longmapsto \sigma(c_1)\,f\,c_2^{-1},
  \qquad
  A_i\longmapsto \sigma(c_1)\,A_i\,\sigma^{-i}(c_2^{-1}).
\]

If one formally sets \(\sigma=\mathrm{id}\), this becomes the ordinary
\((j-1)\)-Kronecker problem with dimension vector \((m,n)\). The one-arrow
case is finite, and the two-arrow case is the tame matrix-pencil problem. With
at least three arrows, the representations of dimension vector \((r,r)\) of
the form \((I,A,B,0,\ldots,0)\) already encode pairs of \(r\times r\) matrices
up to simultaneous similarity, so the problem is wild. Thus even the
untwisted model is wild once \(j\ge4\).
The Frobenius twist can change this behavior: the transformation of the
coefficient \(A_i\) involves \(\sigma^{-i}\). For dimension vector \((1,1)\) and
\(j=3\), the untwisted \(\mathbb P^1\)-family over \(\bar k\) collapses to the
three classes of Remark~\ref{rem:no-projectivization}. We do not pursue the skew
classification for \(j\ge4\) here.

\subsubsection*{\normalfont\bfseries Three-dimensional dominoes}

Our final example classifies the three-dimensional dominoes with three distinct
types. Here all multiplicities are one, and Theorem~\ref{thm:fixed-type-groupoid}
reduces the classification and indecomposability problem to the formulas below.

For \(a<b<c\), write a presentation of type \((a,b,c)\) as
\[
  f=(x_1,x_2,y)\in
  k_\sigma[V]_{\le b-a-2}\times k_\sigma[V]_{\le c-b-2}
  \times k_\sigma[V]_{\le c-a-2},
  \qquad
  x_1=f_{12},\quad x_2=f_{23},\quad y=f_{13},
\]
and write an upper triangular change of generators \(g\in\mathfrak n_0^\times\)
in terms of its diagonal entries \((c_1,c_2,c_3)\in(k^\times)^3\) and
above-diagonal entries \(h_1\in k_\sigma[V]_{\le b-a}\),
\(h_2\in k_\sigma[V]_{\le c-b}\), and \(h_3\in k_\sigma[V]_{\le c-a}\).
Substitution gives
\[
  x_1\mapsto\sigma(c_1)^{-1}x_1c_2,\qquad
  x_2\mapsto\sigma(c_2)^{-1}x_2c_3,\qquad
  y\mapsto\sigma(c_1)^{-1}\bigl(y\,c_3+x_1h_2-\sigma(h_1)\,\sigma(c_2)^{-1}x_2c_3\bigr),
\]
while \(h_3\) does not affect the triple. Put
\[
  Q(x_1,x_2):=
  \frac{k_\sigma[V]_{\le c-a-2}}
  {x_1k_\sigma[V]_{\le c-b}
  +\sigma(k_\sigma[V]_{\le b-a})x_2}.
\]

\begin{proposition}[Three-dimensional dominoes with distinct types]\label{prop:rank-three}
The isomorphism classes of dominoes of type \((a,b,c)\) are the orbits of
\((x_1,x_2,\bar y)\), where \(\bar y\in Q(x_1,x_2)\), under
\[
 (x_1,x_2,\bar y)\longmapsto
 \bigl(\sigma(c_1)^{-1}x_1c_2,\,
 \sigma(c_2)^{-1}x_2c_3,\,
 \sigma(c_1)^{-1}\bar y c_3\bigr).
\]
The corresponding domino \(M_f\) is indecomposable if and only if either both
\(x_1,x_2\) are nonzero, or exactly one is nonzero and \(\bar y\ne0\).
\end{proposition}

\begin{proof}
The presentation is Theorem~\ref{thm:fixed-type-groupoid}. Substitution in its
three \(F\)-relations gives the formulas above. Since there are three blocks of
multiplicity one, every \(p\)-divisible discrepancy from coefficientwise
Teichm\"uller addition lands in the final block and vanishes. Hence these formulas
are exactly the entries of \(\sigma(g)^{-1}fg\). The entries \(h_1,h_2\) change
\(y\) by the denominator defining \(Q(x_1,x_2)\), while the diagonal entries
act on \(\bar y\) by \(\sigma(c_1)^{-1}\bar y c_3\). This proves the orbit
classification.

By the criterion following Definition~\ref{def:indecomposable-domino}, a
decomposition of \(M_f\) partitions its three graded pieces. The partition
\(\{a\}\sqcup\{b,c\}\) forces
\(x_1=\bar y=0\), the partition \(\{a,b\}\sqcup\{c\}\) forces
\(x_2=\bar y=0\), and the partition \(\{b\}\sqcup\{a,c\}\) forces
\(x_1=x_2=0\). Conversely, in each of these cases
an upper triangular change of generators clears a row or column of \(f\),
exhibiting the corresponding splitting. The remaining cases are precisely
the indecomposable ones stated in the proposition.
\end{proof}

\begin{remark}
Two specializations clarify the role of the type gaps. In general,
Theorem~\ref{thm:distinguished} characterizes the distinguished domino as the
unique \emph{maximal-\(p\)-exponent} object of its type, but does not assert that it is
the unique indecomposable. For the type \((0,2,4)\),
Proposition~\ref{prop:rank-three}
gives the stronger conclusion that \(U_{0,2,4}\) is the unique indecomposable.
For type \((0,3,5)\), however,
\(Q(x_1,0)=k_\sigma[V]_{\le3}/x_1k_\sigma[V]_{\le2}\) is one-dimensional when
\(x_1\ne0\). Choosing \(\bar y\ne0\) gives an indecomposable with \(x_2=0\).
Here \(F^2M^0=0\) but \(pM^0\ne0\), so its \(p\)-exponent is \(2<3\).
Thus indecomposability need not imply maximal \(p\)-exponent.
\end{remark}

\section{Mazur--Ogus varieties and Ekedahl modules}
\label{sec:applications}
\label{sec:ekedahl-reconstruction}

For a Mazur--Ogus variety \(X/k\), the domino numbers are determined by Hodge and
Newton data, but the dominoes themselves encode finer type and extension
information. Our goal is to reconstruct from the \(F\)-crystal
\(H^2_{\mathrm{crys}}(X/W)\) the two-term $R$-module
\[
  H^2(X,W\Omega_X^\bullet)^{[0,0]}
  :=
  \bigl[H^2(X,W\mathcal O_X)\xrightarrow d F^\infty B_X^{1,2}\bigr],
\]
where the superscript \([0,0]\) denotes the \emph{diagonal truncation} of
Definition~\ref{def:diagonal-t-structure}. Its quotient by the Hodge--Witt core \(C_X^{0,2}\) is the domino
\(U_X^{0,2}\).

In total degree \(2\), Ekedahl's domino range
(Proposition~\ref{prop:ekedahl-duality}) and Mazur--Ogus d\'evissage
(Theorem~\ref{thm:diagonal-MO-devissage}) force
\(H^2(X,W\Omega_X^\bullet)^{[0,0]}\) to be a two-term coherent $R$-module for
which \(dV\) is surjective. We call such objects \emph{Ekedahl modules}. A general
diagonal truncation need not have this form.

Following Nygaard's reconstruction technique \cite[Theorem~4.3]{Nygaard81}, we encode this
module by a \emph{Nygaard pair}. Write \(K_0:=\ker d\),
\(K_{-1}:=\ker(dV)\), and \(\nu:=V|_{K_{-1}}\). The reconstruction has three
steps:
\[
  H^2_{\mathrm{crys}}(X/W)
  \longmapsto
  (K_0,K_{-1},\iota,\nu)
  \longmapsto
  H^2(X,W\Omega_X^\bullet)^{[0,0]}
  \longmapsto
  (C_X^{0,2},U_X^{0,2}).
\]

The first step extends Nygaard's surface calculation to every smooth proper
Mazur--Ogus variety (Theorem~\ref{thm:geometric-ekedahl-pair-partial-v}). The
second is the algebraic equivalence between no-finite-torsion (nft) Ekedahl modules and Nygaard pairs
(Theorem~\ref{thm:two-kernel-presentation}). The Mazur--Ogus condition guarantees
that the slice above is nft. The last step is the core--domino d\'evissage, with
finite-free core \(C_X^{0,2}\) and positive domino quotient \(U_X^{0,2}\).
Subsection~\ref{Mazur-Ogus} supplies the diagonal and geometric input, while
Subsection~\ref{sec:ekedahl-modules} develops the algebraic reconstruction.

\subsection{Diagonal \texorpdfstring{\(t\)-structure}{t-structure} and Mazur--Ogus objects}\label{Mazur-Ogus}

This subsection develops the diagonal formalism in which the Mazur--Ogus
condition naturally lives, and then identifies its degree-two domino quotient.

\begin{definition}[Diagonal \(t\)-structure and diagonal cohomology]
\label{def:diagonal-t-structure}
\label{def:diagonal-heart}
Let \(N\) be a coherent \(R\)-module. For each \(i\in\mathbb Z\), set
\[
\left\{
\begin{aligned}
  \widetilde\tau_{\le i}N
  &:=
  \begin{gathered}[t]
  \bigl[
    \cdots
    \longrightarrow
    N^{i-1}
    \xrightarrow{d}
    N^i
    \xrightarrow{d}
    F^\infty B^{i+1}(N)
    \longrightarrow
    0
  \bigr],
  \end{gathered}
\\
  \widetilde\tau_{\ge i+1}N
  &:=
  \begin{gathered}[t]
  \bigl[
    0
    \longrightarrow
    N^{i+1}/F^\infty B^{i+1}(N)
    \xrightarrow{d}
    N^{i+2}
    \xrightarrow{d}
    \cdots
  \bigr].
  \end{gathered}
\end{aligned}
\right.
\]
These constructions fit into a natural short exact sequence
\(0\to\widetilde\tau_{\le i}N\to N\to\widetilde\tau_{\ge i+1}N\to0\) of coherent
\(R\)-modules. Ekedahl's \emph{diagonal \(t\)-structure} on \(D^b_c(R)\)
\cite[Theorem~I.1.1 and Definition~I.1.2]{ekedahl3} is defined as
\[
\left\{
\begin{aligned}
  \widetilde D^{\le0}_c(R)
  &=
  \{M:\widetilde\tau_{\le -q}H^q(M)\simeq H^q(M)\ \text{for all }q\},
\\
  \widetilde D^{\ge0}_c(R)
  &=
  \{M:\widetilde\tau_{\ge -q}H^q(M)\simeq H^q(M)\ \text{for all }q\}.
\end{aligned}
\right.
\]
Its heart
\(\Delta:=\widetilde D^{\le0}_c(R)\cap \widetilde D^{\ge0}_c(R)\subset D^b_c(R)\)
is the \emph{diagonal heart}. Objects of \(\Delta\) are called
\emph{diagonal complexes}. We use the same symbols for the associated truncation
functors on \(D^b_c(R)\). On the Postnikov heart they agree with the explicit
constructions above. Following the notation of \cite{LY26}, set
\[
  M^{[i,i]}:=\widetilde\tau_{\le i}\widetilde\tau_{\ge i}M,
  \qquad
  \widetilde H^i(M):=M^{[i,i]}[i]\in\Delta .
\]
\end{definition}

We will use the commuting property, on \(D^b_c(R)\), between this diagonal
\(t\)-structure and the Postnikov \(t\)-structure on \(D(R)\).

\begin{remark}[Commuting \(t\)-structures]
\label{rem:ordinary-diagonal-commute}
Write \(\tau_{\le n},\tau_{\ge n}\) for the truncation functors of the
Postnikov \(t\)-structure on \(D^b_c(R)\), and
\(\widetilde\tau_{\le m},\widetilde\tau_{\ge m}\) for those of the diagonal
\(t\)-structure. The two \(t\)-structures commute: each composite of a Postnikov
truncation with a diagonal truncation has a canonical natural isomorphism with the
composite in the opposite order, e.g.
\(\tau_{\le n}\widetilde\tau_{\le m}\simeq\widetilde\tau_{\le m}\tau_{\le n}\), and
likewise for the three remaining combinations of \(\le,\ge\). This is due to Ekedahl
\cite[Definition~0.1.3, Theorem~0.1.4, and Section~I.1]{ekedahl3}.
\end{remark}

\begin{definition}[Hodge--Witt objects]
\label{def:HW-object}
\label{def:diagonal-HW-object}
Let \(M\in D^b_c(R)\). We say that \(M\) is \emph{Hodge--Witt} if, for each coherent
\(R\)-module \(H^q(M)\), the differential
\(d:H^q(M)^i\to H^q(M)^{i+1}\) is zero for all \(i,q\).
For a smooth proper variety \(X/k\), this recovers the usual terminology:
\(X\) is a Hodge--Witt variety exactly when
\(R\Gamma(X,W\Omega_X^\bullet)\) is Hodge--Witt in this sense.
For objects of the diagonal heart, this defines a full subcategory
\(\Delta_{\mathrm{HW}}\subset\Delta\).
\end{definition}

Ekedahl proves that \(\Delta_{\mathrm{HW}}\) is closed under subobjects,
quotients, and extensions, and that every \(Y\in\Delta\) has a maximal
Hodge--Witt subobject \(HW(Y)\subset Y\)
\cite[Section~0.5 and Proposition~I.2.1]{ekedahl3}.

\begin{lemma}[Hodge--Witt core]\label{lem:two-term-HW-core}
For \(N\in\operatorname{Mod}_c(R)\cap\Delta\) supported in internal degrees
\(0,1\), one has \(HW(N)=C^0(N)=V^{-\infty}Z^0(N)\subset N^0\).
\end{lemma}

\begin{proof}
Since \(F^\infty B^0(N)=0\), the core \(C^0(N)\) is
\(V^{-\infty}Z^0(N)\). It is
concentrated in degree \(0\), hence Hodge--Witt, while
Proposition~\ref{prop:core-domino-devissage} identifies the
quotient by it as a domino, which has no nonzero Hodge--Witt subobject. Hence the
core is maximal.
\end{proof}

\begin{definition}[Diagonal dominoes]\label{def:diagonal-domino}
A \emph{diagonal domino} is an object \(D\in\Delta\) satisfying:
\begin{enumerate}
  \item \(H^r(\mathbf s(D))\) is a torsion \(W\)-module for every \(r\), where
  \(\mathbf s\) is the
  simple functor of Remark~\ref{rem:simple-functor}.
  \item \(D\) is nft: it contains no nonzero subobject in \(\Delta\) of finite
  length over \(W\).
\end{enumerate}
We denote by \(\mathrm{Dom}_\Delta\subset\Delta\) the full subcategory of diagonal
dominoes, with exact sequences inherited from \(\Delta\)
\cite[Definition~III.2.1]{ekedahl3}.
Dominoes are precisely the diagonal dominoes that lie in the Postnikov heart
\(\operatorname{Mod}_c(R)\) \cite[Proposition~III.3.1]{ekedahl3}.
\end{definition}

Ekedahl showed that every diagonal domino \(D\) has a canonical finite decreasing
filtration by half-integer type
\cite[Theorem~III.2.3 and Definition~III.2.4]{ekedahl3}:
\[
  0\subset\cdots\subset
  \mathrm{Fil}_{\Delta}^{1/2}D\subset
  \mathrm{Fil}_{\Delta}^{0}D\subset
  \mathrm{Fil}_{\Delta}^{-1/2}D\subset\cdots\subset D.
\]
For a domino \(U\), we write its type filtration as
\(\mathrm{Fil}_{\mathrm{type}}^j(U)\), as in
Proposition~\ref{prop:type-filtration}. The type factors of \(U\) then have integer
types. Under the identification with diagonal dominoes in the Postnikov heart,
the two filtrations agree in integer degree
\cite[Proposition~III.3.2]{ekedahl3}: for each \(j\in\mathbb Z\),
\[
  \mathrm{Fil}_{\Delta}^{j}U
  =
  \mathrm{Fil}_{\Delta}^{j-1/2}U
  =
  \mathrm{Fil}_{\mathrm{type}}^j(U).
\]

\begin{definition}[Positive diagonal dominoes]\label{def:positive-diagonal-domino}
We denote by
\(\operatorname{Dom}^{+}_{\Delta}\subset\operatorname{Dom}_{\Delta}\) the full
subcategory of diagonal dominoes \(D\) such that
\(\mathrm{Fil}_{\Delta}^{1/2}D=D\), and put
\[
  \operatorname{Dom}^{+}
  :=
  \operatorname{Dom}^{+}_{\Delta}\cap\operatorname{Mod}_c(R).
\]
Equivalently, \(\operatorname{Dom}^{+}\) is the full subcategory of dominoes
whose type sequence has only positive entries.
\end{definition}

\begin{lemma}
\label{lem:positive-iff-dV-surjective}
\label{lem:positive-detector}
Let \(U\) be a domino. Then \(U\in\operatorname{Dom}^{+}\) if and only if
\(dV:U^0\to U^1\) is surjective.
\end{lemma}

\begin{proof}
For the elementary domino \(U_j\), \(\operatorname{coker}(dV)=0\) if and only if
\(j\geq1\).

In a short exact sequence of dominoes \(0\to U'\to U\to U''\to0\), all maps
commute with \(dV\). The snake lemma therefore shows that
surjectivity of \(dV\) is preserved under extensions and passes to quotients.

If \(U\in\operatorname{Dom}^{+}\), then every associated graded piece of the
type filtration is a sum of \(U_j\)'s with \(j\geq1\), by
Lemma~\ref{lem:pure-type-split}. Extension stability then gives surjectivity on
\(U\).

Conversely, set \(U_{\geq1}:=\mathrm{Fil}_{\mathrm{type}}^1(U)\) and
\(U_{\leq0}:=U/U_{\geq1}\). Then \(U_{\geq1}\in\operatorname{Dom}^{+}\), while all
type entries of \(U_{\leq0}\) are \(\leq0\). If \(dV\) is surjective on \(U\), then
so is \(dV\) on the quotient \(U_{\leq0}\). If this quotient were nonzero, one
of its pure type quotients would be a nonzero sum of \(U_j\)'s with \(j\leq0\),
on which \(dV\) is not surjective. This contradicts quotient stability. Hence
\(U_{\leq0}=0\), so
\(U=U_{\geq1}\in\operatorname{Dom}^{+}\).
\end{proof}

Having fixed the diagonal heart, Hodge--Witt subobjects, and
\(\operatorname{Dom}^{+}_{\Delta}\), we can state the abstract Mazur--Ogus formalism
used below. It isolates the hypotheses under which Ekedahl's two structural results
apply: diagonal decomposition (Theorem~\ref{thm:diagonal-MO-splitting}) and
Mazur--Ogus d\'evissage (Theorem~\ref{thm:diagonal-MO-devissage}).

\begin{definition}[Mazur--Ogus objects, {\cite[Definition~IV.1.1]{ekedahl3}}]
\label{def:diagonal-MO-object}
An object \(M\in D^b_c(R)\) is a \emph{Mazur--Ogus object} if the following
conditions hold:
\begin{itemize}
  \item the cohomology groups \(H^n(\mathbf s(M))\) are torsion-free \(W\)-modules.
  \item the analogue of the Hodge--de Rham spectral sequence
  \[
    E_1^{i,j}(M)
    =
    H^j\!\left(R_1\otimes_R^{\mathbf L}M\right)^i
    \Longrightarrow
    H^{i+j}\!\left(k\otimes_W^{\mathbf L}\mathbf s(M)\right),
  \]
  degenerates at \(E_1\).
\end{itemize}
We write \(\Delta_{\mathrm{MO}}\subset\Delta\) for the full subcategory of
Mazur--Ogus objects in the diagonal heart.

\label{MO}
A smooth proper variety \(X/k\) is \emph{Mazur--Ogus} (also called \emph{straight}
in \cite[Definition~1.4]{GSY25}) if \(R\Gamma(X,W\Omega_X^\bullet)\in D^b_c(R)\) is a
Mazur--Ogus object. Equivalently, its crystalline cohomology
\(H^*_{\mathrm{crys}}(X/W)\) is torsion-free and the Hodge--de Rham spectral
sequence for \(X\) degenerates at \(E_1\).
\end{definition}

\begin{theorem}[Diagonal decomposition for Mazur--Ogus objects
{\cite[Theorem~IV.1.2(i) and (iv)]{ekedahl3}}]
\label{thm:diagonal-MO-splitting}
Let \(M\in D^b_c(R)\) be a Mazur--Ogus object. Then there is an isomorphism
\[
  M\simeq\bigoplus_n\widetilde H^n(M)[-n]
\]
in \(D^b_c(R)\). Moreover \(\widetilde H^n(M)\in\Delta_{\mathrm{MO}}\) for every
\(n\).
\end{theorem}

\begin{theorem}[Diagonal Mazur--Ogus d\'evissage
{\cite[Theorem~III.4.6]{ekedahl3}}]
\label{thm:Eke-MO-quotient-side}
\label{thm:diagonal-MO-devissage}
Let \(M\in\Delta_{\mathrm{MO}}\). Then \(M\) has a maximal finite-free
Hodge--Witt subobject \(HW(M)\in\Delta_{\mathrm{HW}}\), fitting into a short
exact sequence
\[
  0\longrightarrow HW(M)\longrightarrow M\longrightarrow D_\Delta(M)\longrightarrow 0
\]
with quotient \(D_\Delta(M)\in\operatorname{Dom}^{+}_{\Delta}\).
\end{theorem}

\begin{example}
Smooth projective curves of any genus are Mazur--Ogus. Abelian varieties and K3
surfaces are straight in the sense of \cite[Definition~1.4 and the following
paragraph]{GSY25}, hence Mazur--Ogus. We use only the abelian-variety case below.
In particular, all supersingular abelian varieties considered here are
Mazur--Ogus.
\end{example}

\begin{remark}[Geometric Mazur--Ogus notation]\label{rem:geometric-MO-object}
For a Mazur--Ogus variety \(X/k\), we write
\[
  \widetilde H^n(X):=
  \widetilde H^n\!\left(R\Gamma(X,W\Omega_X^\bullet)\right)\in\Delta,
  \qquad
  D_{\Delta,X}^n:=
  D_\Delta(\widetilde H^n(X)).
\]
By Remark~\ref{rem:simple-functor}, the simple realization of
\(R\Gamma(X,W\Omega_X^\bullet)\) is crystalline cohomology, and by
Remark~\ref{rem:truncation}, its level-one realization is the algebraic de Rham
complex.
\end{remark}

\begin{proposition}[Positivity of \(U_X^{0,2}\)]
\label{prop:diagonal-ordinary-degree-two-domino}
Let \(X/k\) be a smooth proper Mazur--Ogus variety. Then there is a canonical
isomorphism \(D^2_{\Delta,X}\xrightarrow{\sim}U_X^{0,2}\) in the diagonal heart.
Moreover,
\(H^0(HW(\widetilde H^2(X)))\simeq C_X^{0,2}[0]\).
In particular \(D^2_{\Delta,X}\) lies in the Postnikov heart, and every entry
of the type sequence \(J(U_X^{0,2})\) is positive.
\end{proposition}

\begin{proof}
Put \(Y=\widetilde H^2(X)\).
Since the Postnikov and diagonal \(t\)-structures commute
(Remark~\ref{rem:ordinary-diagonal-commute}), the Postnikov filtration of \(Y\)
in \(\Delta\) has graded pieces \(M_0,M_1[1],M_2[2]\), where
\(M_j=H^{2-j}(X,W\Omega_X^\bullet)^{[j,j]}\).

For \(q=0,1\), the domino-range theorem (Proposition~\ref{prop:ekedahl-duality})
gives \(T^{i,q}(X)=0\) for all \(i\). Thus the domino parts of \(M_1\) and \(M_2\)
vanish. By Proposition~\ref{prop:core-domino-devissage}, \(M_1[1]\) and
\(M_2[2]\), and therefore \(\tau_{\le -1}Y\), are Hodge--Witt.

Let \(\pi:Y\to M_0\) be the quotient map in \(\Delta\). We claim that
\(Y/HW(Y)\simeq M_0/HW(M_0)\). Since \(\tau_{\le -1}Y\) is Hodge--Witt, it is
contained in \(HW(Y)\). Put \(Q:=\pi^{-1}(HW(M_0))\). Then \(Q\) is an extension of
Hodge--Witt objects, hence is Hodge--Witt. By maximality \(Q\subset HW(Y)\).
Conversely, \(\pi(HW(Y))\) is
Hodge--Witt because \(\Delta_{\mathrm{HW}}\) is closed under quotients and subobjects
\cite[Proposition~I.2.1]{ekedahl3}, so \(\pi(HW(Y))\subset HW(M_0)\), i.e.
\(HW(Y)\subset Q\). Hence \(HW(Y)=Q\), and the claim follows.

Lemma~\ref{lem:two-term-HW-core} and
Proposition~\ref{prop:core-domino-devissage}, applied to
\(M_0=\bigl[H^2(X,W\mathcal O_X)\xrightarrow dF^\infty B_X^{1,2}\bigr]\) gives
\(HW(M_0)=V^{-\infty}Z_X^{0,2}[0]\) and
\(M_0/HW(M_0)=U_X^{0,2}\). Hence
\(D^2_{\Delta,X}=Y/HW(Y)\simeq M_0/HW(M_0)=U_X^{0,2}\).
Since \(\ker(\pi)=\tau_{\le-1}Y\), taking Postnikov cohomology also gives
\(H^0(HW(Y))\simeq HW(M_0)=C_X^{0,2}[0]\).
Finally, Theorem~\ref{thm:diagonal-MO-splitting} gives \(Y\in\Delta_{\mathrm{MO}}\),
so Theorem~\ref{thm:diagonal-MO-devissage} gives
\(U_X^{0,2}\in\operatorname{Dom}^{+}_{\Delta}\cap\operatorname{Mod}_c(R)
=\operatorname{Dom}^{+}\).
\end{proof}

The reconstruction below starts from the \(F\)-crystal
\(H^n_{\mathrm{crys}}(X/W)\), so we need the Hodge--Witt filtration without
referring back to the diagonal complex. Ekedahl's formula supplies exactly this
intrinsic \(F\)-crystal description.

\begin{proposition}[Hodge--Witt filtration formula
{\cite[Theorem~III.4.5]{ekedahl3}}]\label{prop:HW-filtration}
Let \(X/k\) be a Mazur--Ogus variety, and let
\(M=H^n_{\mathrm{crys}}(X/W)\) with \(F\)-crystal Frobenius \(\varphi\).
Let \(\mathrm{Fil}^i_{\mathrm{HW}}M\) be the abutment filtration on \(M\) of the
slope spectral sequence for \(W\Omega_X^\bullet\). Then
\[
  \mathrm{Fil}^i_{\mathrm{HW}}M
  =
  \bigcap_{r\ge 0}\varphi^{-r}\!\bigl(p^{ir}M\bigr).
\]
Since \(M\) is torsion-free, one checks by iteration that this is equivalently
the maximal sub-\(F\)-crystal \(L\subset M\) with
\(\varphi(L)\subset p^iL\).
\end{proposition}

\subsection{Ekedahl modules and positive dominoes}\label{sec:ekedahl-modules}

In \cite{ekedahl3}, Ekedahl defines three \(t\)-structures on \(D_c^b(R)\): the
Postnikov, diagonal, and \(F\)-gauge \(t\)-structures, with hearts
\(\operatorname{Mod}_c(R)\), \(\Delta\), and \(\mathcal G\), respectively. The
intersection \(\operatorname{Mod}_c(R)\cap\Delta\) is central in Ekedahl's work, and
natural examples include \(H^j(X,W\Omega_X^\bullet)^{[i,i]}\). We observe that
\(H^2(X,W\Omega_X^\bullet)^{[0,0]}\) lies in the triple-heart intersection
\(\operatorname{Mod}_c(R)\cap\Delta\cap\mathcal G\). Ekedahl does not single out this
exact category as a separate object of study. In homage, we call its objects
\emph{Ekedahl modules}. In the main text we use only the concrete
description below. Appendix~\ref{app:ekedahl-three-hearts} recalls \(\mathcal G\), the
\(F\)-gauge interpretation, and the three-heart criterion.

\begin{definition}[Ekedahl modules]\label{def:ekedahl-module}
Define \(\mathcal E\) to be the full subcategory of \(\operatorname{Mod}_c(R)\)
spanned by coherent two-term $R$-modules \(M=[M^0\xrightarrow d M^1]\), supported
in internal degrees \(0,1\), such that \(dV:M^0\twoheadrightarrow M^1\) is surjective.
The exact structure is inherited from \(D_c^b(R)\).
\end{definition}

\begin{theorem}[D\'evissage for Ekedahl modules]
\label{thm:ekedahl-module-devissage}
Let \(M=[M^0\xrightarrow dM^1]\) be a coherent two-term $R$-module, and put
\[
  C(M):=V^{-\infty}Z^0(M)=\bigcap_{r\ge0}\ker(dV^r:M^0\to M^1),
  \qquad
  \Dom(M):=\bigl[M^0/C(M)\xrightarrow{\bar d}M^1\bigr].
\]
Then \(C(M)\) is a coherent Dieudonn\'e module. Moreover, the following are
equivalent.
\begin{enumerate}
\item \(M\in\mathcal E\).
\item \(\Dom(M)\in\operatorname{Dom}^{+}\).
\end{enumerate}
\end{theorem}

\begin{proof}
Since \(M\) is supported in internal degrees \(0,1\), one has
\(F^\infty B^0(M)=0\), hence \(C(M)=C^0(M)\). The finiteness
statement in \S\ref{sec:hearts-dominoes} makes this a coherent Dieudonn\'e module.

Assume \(M\in\mathcal E\). Since \(dV=d\circ V\), the surjectivity of \(dV\) implies
the surjectivity of \(d\), hence \(F^\infty B^1(M)=M^1\). Therefore \(\Dom(M)\) is
the domino \(\Dom^0(M)\) of Proposition~\ref{prop:core-domino-devissage}. Since
\(C(M)[0]\) has zero degree-\(1\) term, the surjectivity of \(dV\) on \(M\) passes to
\(\Dom(M)\). Lemma~\ref{lem:positive-iff-dV-surjective} gives
\(\Dom(M)\in\operatorname{Dom}^{+}\).

Conversely, assume (2). The sequence \(0\to C(M)[0]\to M\to \Dom(M)\to0\) is exact by
the definition of \(\Dom(M)\). Since \(dV\) is surjective on \(\Dom(M)\) by
Lemma~\ref{lem:positive-iff-dV-surjective}, and \(C(M)[0]\) has no degree-\(1\) term,
\(dV\) is also surjective on \(M\). Thus \(M\in\mathcal E\).
\end{proof}

For reconstruction, we restrict to nft Ekedahl modules, which admit a
description by Nygaard pairs (Definition~\ref{def:nygaard-pair}).

\begin{definition}[No finite torsion]
\label{def:nft-ekedahl}
We say that an object \(M\in\Delta\) has \emph{no finite torsion}, abbreviated
\emph{nft}, if it contains no nonzero finite-torsion subobject in \(\Delta\).
This is Ekedahl's condition ``without finite torsion.'' By the three-heart
criterion of Appendix~\ref{app:ekedahl-three-hearts}, \(\mathcal E\subset\Delta\).
Let \(\mathcal E^{\mathrm{nft}}\subset\mathcal E\) be the full subcategory of nft
Ekedahl modules, with the exact structure inherited from \(\mathcal E\).
Ekedahl proves that an object \(M\in\Delta\cap\operatorname{Mod}_c(R)\) is nft
if and only if \(C(M)=V^{-\infty}Z^0(M)\) is a finite free Dieudonn\'e module
\cite[Theorem~I.1.12(ii)]{ekedahl3}. In particular, an Ekedahl module \(M\) is
nft exactly when its core \(C(M)\) is finite free.
\end{definition}

Let \(\mathcal E^{\mathrm{tor}}\subset\mathcal E\) be the full subcategory of
\emph{torsion Ekedahl modules}, those whose core \(C(M)\) is \(W\)-torsion.
Every \(M\in\mathcal E\) fits into
\(0\to C(M)[0]\to M\to\Dom(M)\to0\), where
\(\Dom(M)\in\operatorname{Dom}^{+}\). The conditions
\(M\in\mathcal E^{\mathrm{nft}}\) and \(M\in\mathcal E^{\mathrm{tor}}\) say that
\(C(M)\) is finite free and torsion, respectively. Thus the two conditions
hold simultaneously exactly when \(C(M)=0\), giving the following diamond of
full inclusions:
\[
\begin{tikzcd}[row sep=.3em, column sep=3.5em]
& \mathcal E^{\mathrm{nft}} \arrow[dr, hook] & \\
\operatorname{Dom}^{+} \arrow[ur, hook] \arrow[dr, hook]
&& \mathcal E \\
& \mathcal E^{\mathrm{tor}} \arrow[ur, hook] &
\end{tikzcd}
\]

The reconstruction of an nft Ekedahl module will use only
\(K_0=\ker d\), \(K_{-1}=\ker(dV)\), their inclusion, and the restriction of
\(V\) to \(K_{-1}\). We first axiomatize the finite data carried by these two
kernels.

\begin{definition}[Nygaard pairs]\label{def:nygaard-pair}
A \emph{Nygaard pair} is the data \((K_0,K_{-1},\iota,\nu)\) of finitely
generated \(W\)-modules \(K_0\) and \(K_{-1}\), a \(W\)-linear injection
\(\iota:K_{-1}\hookrightarrow K_0\), through which we regard \(K_{-1}\) as a
submodule of \(K_0\), and a \(\sigma^{-1}\)-linear injection
\(\nu:K_{-1}\hookrightarrow K_0\), subject to two conditions:
\begin{enumerate}
  \item \(pK_0\subseteq K_{-1}\cap\nu(K_{-1})\).
  \item the decreasing sequence
  \[
    K_{-(r+1)}:=\{x\in K_{-1}:\nu(x)\in K_{-r}\}\qquad(r\ge1)
  \]
  stabilizes at a finite free \(W\)-module \(C\) on which \(\nu\) is
  topologically nilpotent, and \(K_0/C\) has finite \(W\)-length.
\end{enumerate}
A morphism of Nygaard pairs is a \(W\)-linear map \(K_0\to K'_0\) carrying
\(K_{-1}\) into \(K'_{-1}\) and commuting with \(\nu\). Write
\(\mathsf{NygPair}\) for the resulting category. A sequence in
\(\mathsf{NygPair}\) is
\emph{exact} if it is exact on \(K_0\) and on \(K_{-1}\). The equivalence below
shows that these sequences define an exact structure.
\end{definition}

On an Ekedahl module, these two kernels extend in both directions through the
modified differentials \(F^nd\) of \S\ref{sec:devissage}.

\begin{definition}[Kernel ladder]\label{def:kernel-ladder}
For \(M\in\mathcal E\), use the convention \(F^{-r}d=dV^r\). The natural
inclusions between the kernels of \(F^nd:M^0\to M^1\), together with the
restrictions of \(V\), form the \emph{kernel ladder}
\[
\begin{tikzcd}[column sep=2.1em]
\cdots
  \arrow[r,shift left=.55ex,hook]
  \arrow[r,shift right=.55ex,swap,"V"]
& \ker(dV^2)
  \arrow[r,shift left=.55ex,hook]
  \arrow[r,shift right=.55ex,swap,"V"]
& \ker(dV)
  \arrow[r,shift left=.55ex,hook]
  \arrow[r,shift right=.55ex,swap,"V"]
& \ker d
  \arrow[r,shift left=.55ex,hook]
  \arrow[r,shift right=.55ex,swap,"V"]
& \ker(Fd)
  \arrow[r,shift left=.55ex,hook]
  \arrow[r,shift right=.55ex,swap,"V"]
& \cdots .
\end{tikzcd}
\]
The upper arrows are \(W\)-linear and the lower arrows are
\(\sigma^{-1}\)-linear. These kernels form an increasing chain of
\(W\)-submodules of \(M^0\), and
\(\bigcap_{r\ge0}\ker(dV^r)=C(M)\). Define
\[
  \mathbf K(M):=\bigl(\ker d,\ \ker(dV),\ \iota,\ V|_{\ker(dV)}\bigr),
\]
where \(\iota\) is the inclusion \(\ker(dV)\subseteq\ker d\) given by \(d=FdV\).
\end{definition}

The reconstruction from the central pair is \(V\)-adic: \(\ker d\) must
generate \(M^0\) modulo \(V\), while the kernels \(\ker(dV^r)\) must record its
intersections with the \(V\)-adic filtration. The next lemma shows that these
conditions also characterize the degree-zero module reconstructed from a
Nygaard pair.

\begin{lemma}[Degree-zero realization of Nygaard pairs]\label{lem:kernel-ladder}
Let \(M\in\mathcal E^{\mathrm{nft}}\). Then \(V\) is injective,
\(\mathbf K(M)\) is a Nygaard pair, and
\[
  M^0=\ker d+VM^0,
  \qquad
  \ker d\cap V^nM^0=V^n\ker(dV^n)\quad(n\ge0).
\]
Conversely, every Nygaard pair determines functorially a \(V\)-adically
complete \(W\)-module \(M^0\supseteq K_0\), with injective
\(\sigma^{-1}\)-linear \(V\) extending \(\nu\), finite-length quotients
\(M^0/V^nM^0\), with \(M^0=K_0+VM^0\) and
\(K_0\cap V^nM^0=V^nK_{-n}\). For
\(\mathbf K(M)\), where \(K_{-n}=\ker(dV^n)\), this construction recovers
\(M^0\) naturally.
\end{lemma}

\begin{proof}
\emph{Step 1: extracting the Nygaard pair.}
The core \(C(M)\) is finite free because \(M\) is nft. If \(Vx=0\), the Raynaud
relations give \(px=FVx=0\). They also give \(dx=FdVx=0\), while
\(dV^rx=0\) is immediate for \(r\ge1\). Thus
\(x\in\bigcap_{r\ge0}\ker(dV^r)=C(M)\), and \(px=0\) forces \(x=0\) because
\(C(M)\) is finite free. Hence \(V\) is injective on \(M^0\). Given
\(x\in M^0\), surjectivity of \(dV\) provides
\(y\in M^0\) with \(dVy=dx\). Hence \(x-Vy\in\ker d\), proving
\(M^0=\ker d+VM^0\). Moreover, \(V^ny\in\ker d\) exactly when
\(y\in\ker(dV^n)\), so
\(\ker d\cap V^nM^0=V^n\ker(dV^n)\).

For \(x\in\ker d\), the relations \(Vd=pdV\) and \(VF=p\) give
\(px,Fx\in\ker(dV)\) and \(px=VFx\). Hence
\(p\ker d\subseteq\ker(dV)\cap V(\ker(dV))\). The recursively defined terms of
\(\mathbf K(M)\) are
\(K_{-r}=\ker(dV^r)\), so their intersection is \(C(M)\). Since
\(\ker d/C(M)=\ker(\bar d:\Dom(M)^0\to\Dom(M)^1)\) has finite length, the
sequence stabilizes at the finite free Dieudonn\'e module \(C(M)\), on which
\(\nu=V|_{C(M)}\) is topologically nilpotent.
Functoriality is immediate.

\emph{Step 2: the finite quotients \(Q_n\).} Let
\((K_0,K_{-1},\iota,\nu)\) be a Nygaard pair. Using the Frobenius twist
\(\sigma_*\) of Definition~\ref{def:autoequivalence}, define
\[
  \delta_n:\bigoplus_{r=0}^{n-1}\sigma_*^{r+1}K_{-1}
  \longrightarrow\bigoplus_{r=0}^{n-1}\sigma_*^rK_0,
  \qquad
  x\in\sigma_*^{r+1}K_{-1}\longmapsto[\nu x]_r-[\iota x]_{r+1},
\]
where \([\iota x]_{r+1}\) is omitted for \(r=n-1\), and put
\(Q_n:=\operatorname{coker}(\delta_n)\). Here \([a]_r\) denotes the class of
\(a\in\sigma_*^rK_0\). Triangular elimination shows that \(\delta_n\) is
injective. It also shows that \([a]_0=0\) precisely when
\(a=\nu x_0\) and \(x_{r-1}=\nu x_r\) for \(1\le r\le n-1\) in a chain
\(x_0,\ldots,x_{n-1}\in K_{-1}\). By the ladder condition, this is equivalent
to \(a\in\nu^nK_{-n}\), so \(\ker(K_0\to Q_n)=\nu^nK_{-n}\).
Moreover, \(Q_n\) has finite length: if \(pa=\nu(x)\), then
\(p[a]_r=[x]_{r+1}\), so \(p^{n-r}[a]_r=0\).

\emph{Step 3: the \(V\)-adic realization.}
Dropping the last component gives surjections \(Q_{n+1}\twoheadrightarrow Q_n\).
On \(M^0:=\varprojlim_nQ_n\), the shifts \([a]_r\mapsto[a]_{r+1}\) induce a
\(\sigma^{-1}\)-linear \(V\). The same elimination makes \(V\) injective,
identifies \(M^0/V^nM^0\) with \(Q_n\), and yields
\(M^0=K_0+VM^0\) and \(K_0\cap V^nM^0=V^nK_{-n}\).
Here \(K_0\hookrightarrow M^0\) because \(K_{-n}=C\) eventually and \(\nu\) is
topologically nilpotent on \(C\), while \([x]_1=[\nu x]_0\) gives
\(V|_{K_{-1}}=\nu\).

\emph{Step 4: recovering \(M^0\).}
The construction is functorial at each finite level. Applied to
\(\mathbf K(M)\), the identifications \(K_0=\ker d\) and
\(K_{-n}=\ker(dV^n)\) show that the maps \([a]_r\mapsto V^ra\) identify \(Q_n\)
with \(M^0/V^nM^0\). Completeness
(Proposition~\ref{prop:coherence-criterion}) then recovers \(M^0\).
\end{proof}

\begin{theorem}
\label{thm:two-kernel-presentation}
The functor \(\mathbf K:\mathcal E^{\mathrm{nft}}\to\mathsf{NygPair}\) is an
exact equivalence of categories.
\end{theorem}

\begin{proof}
\emph{Step 1: the \(R\)-module attached to a Nygaard pair.} Let
\((K_0,K_{-1},\iota,\nu)\) be a Nygaard pair, and let \(M^0\) be its \(V\)-adic
realization from Lemma~\ref{lem:kernel-ladder}. Put \(M^1:=M^0/K_0\), with
quotient map \(d\). Since \(M^0/VM^0\simeq K_0/\nu(K_{-1})\) is killed by \(p\)
and \(V\) is injective, set \(F:=V^{-1}p\) on \(M^0\). On \(M^1\), set
\(F(dVx):=dx\) and \(V(dx):=p\,dVx\). The equalities
\(M^0=K_0+VM^0\) and \(K_0\cap VM^0=VK_{-1}\), together with the
Nygaard-pair inclusions, make these definitions well defined and give
\(F(K_0)\subseteq K_{-1}\). They yield \(FV=VF=p\), \(FdV=d\), and \(Vd=pdV\).
Writing \(x=a+Vy\), with \(a\in K_0\), gives \(dFx=p\,dy=pFdx\). Thus
\(M:=[M^0\xrightarrow dM^1]\) is an \(R\)-module with \(dV\) surjective.

\emph{Step 2: coherence and the quasi-inverse comparison.}
For \(r\ge0\), the equality \(K_0\cap V^rM^0=V^rK_{-r}\) gives
\[
  \ker(dV^r:M^0\to M^1)=\{x:V^rx\in K_0\cap V^rM^0\}=K_{-r};
\]
in particular \(\ker d=K_0\), \(\ker(dV)=K_{-1}\),
\(V|_{K_{-1}}=\nu\), and the core is \(C=\bigcap_rK_{-r}\). The equalities
\(\ker(dV^r)=K_{-r}\) make \(C[0]\subseteq M\) a finite free Dieudonn\'e module.
In particular,
\(F(C)\subseteq C\) follows from \(d(Fc)=pFdc=0\) and
\(dV^{r+1}(Fc)=dV^r(pc)=0\). The quotient
\(U:=M/C[0]\) is profinite, since \(U^0/V^nU^0\) and
\(U^1/dV^nU^0\simeq M^0/(K_0+V^nM^0)\) are quotients of the finite-length
module \(M^0/V^nM^0\). It has trivial core and surjective differential. By
Remark~\ref{rem:domino-intrinsic-characterization}, \(U\) is a domino, so the extension
\(0\to C[0]\to M\to U\to0\) exhibits \(M\) as coherent
(Definition~\ref{def:coherent}). Hence \(M\in\mathcal E\), and \(M\) is nft
because its core \(C\) is finite free. The functoriality in
Lemma~\ref{lem:kernel-ladder} defines a functor \(\mathbf E\), and
\(\mathbf K\mathbf E\simeq\mathrm{id}\) because
\(\ker d=K_0\), \(\ker(dV)=K_{-1}\), and \(V|_{K_{-1}}=\nu\).
For \(M\in\mathcal E^{\mathrm{nft}}\), the recovery isomorphism of
Lemma~\ref{lem:kernel-ladder} identifies \((\mathbf E\mathbf K(M))^0\) with
\(M^0\), compatibly with \(\ker d\) and \(V\). Since \(M^1=M^0/\ker d\), it
also identifies the degree-one terms, giving
\(\mathbf E\mathbf K\simeq\mathrm{id}\).

\emph{Step 3: exactness of \(\mathbf K\) and \(\mathbf E\).} For an admissible sequence in
\(\mathcal E^{\mathrm{nft}}\), the snake lemma applied to the surjections \(d\)
and \(dV\) makes both \(\ker d\) and \(\ker(dV)\) exact. Hence \(\mathbf K\) is
exact.
Conversely, the snake lemma applied to the injective presentations in the
proof of Lemma~\ref{lem:kernel-ladder} sends an exact sequence of pairs to short
exact sequences of the \(Q_n\). Surjective transition maps preserve exactness
under inverse limits, and quotienting by \(K_0\) gives exactness in degree one.
Thus \(\mathbf E\) is exact.
\end{proof}

The equivalence above reduces the geometric reconstruction problem to
recovering \(\ker d\), \(\ker(dV)\), their inclusion, and
\(V|_{\ker(dV)}\). The next theorem extracts these data from the degree-two
\(F\)-crystal of a Mazur--Ogus variety.

\begin{theorem}
\label{thm:geometric-ekedahl-pair-partial-v}
Let \(X/k\) be a smooth proper Mazur--Ogus variety. Then
\(H^2(X,W\Omega_X^\bullet)^{[0,0]}\) is an nft Ekedahl module. The \(F\)-crystal
\(H^2_{\mathrm{crys}}(X/W)\) functorially determines the Nygaard pair, and
hence the full coherent \(R\)-module structure, of this diagonal slice.
\end{theorem}

\begin{proof}
Put \(M_0=H^0(\widetilde H^2(X))\).
Proposition~\ref{prop:diagonal-ordinary-degree-two-domino} identifies
\(M_0=H^2(X,W\Omega_X^\bullet)^{[0,0]}\),
\(\Dom(M_0)=U_X^{0,2}\), and
\(C(M_0)=V^{-\infty}Z_X^{0,2}=C_X^{0,2}\). Since
\(\Dom(M_0)\in\operatorname{Dom}^{+}\),
Theorem~\ref{thm:ekedahl-module-devissage} gives \(M_0\in\mathcal E\).
The same proposition identifies \(C_X^{0,2}[0]\) with
\(H^0(HW(\widetilde H^2(X)))\), which is finite free by
Theorem~\ref{thm:diagonal-MO-devissage}. Thus
\(M_0\in\mathcal E^{\mathrm{nft}}\) by Definition~\ref{def:nft-ekedahl}.

\emph{Step 1: recovering \(K_0=\ker d\).}
Put \(M:=H^2_{\mathrm{crys}}(X/W)\), with \(\sigma\)-linear \(F\)-crystal Frobenius
\(\Phi\), and set
\(L:=\mathrm{Fil}^1_{\mathrm{HW}}M=\bigcap_{r\ge0}\Phi^{-r}(p^rM)\). By
Proposition~\ref{prop:HW-filtration}, \(L\) is functorially determined by the
\(F\)-crystal. No higher differential enters \((0,2)\) for degree reasons, and
Proposition~\ref{prop:ekedahl-duality} rules out those leaving it. By the
definition of the abutment filtration,
\(M/L\simeq E_\infty^{0,2}=E_2^{0,2}=\ker d\). Hence \(K_0=M/L\).

\emph{Step 2: constructing \(K_{-1}\) from crystalline Frobenius.}
Let \(V_K:=p\Phi^{-1}:M[1/p]\to\sigma_*^{-1}M[1/p]\) be the rational
Verschiebung. Put
\(V_K^{-1}M:=\{m\in M:V_K(m)\in\sigma_*^{-1}M\}\) and
\(K_{-1}:=V_K^{-1}M/(L\cap V_K^{-1}M)\). If \(x\in L\cap V_K^{-1}M\) and
\(y=V_K(x)\), then \(y\in\sigma_*^{-1}M\) and, for \(r\geq1\),
\(\Phi^r(y)=p\Phi^{r-1}(x)\in p^rM\), so
\(y\in\sigma_*^{-1}L\). Thus \([m]\mapsto[m]\) and
\([m]\mapsto[V_K(m)]\) define maps
\(\iota:K_{-1}\to K_0\) and
\(\nu:K_{-1}\to\sigma_*^{-1}K_0\). The map \(\iota\) is injective
by definition of the quotient. If
\(V_K(m)\in\sigma_*^{-1}L\), then \(pm=\Phi(V_K(m))\in pL\), hence \(m\in L\).
Thus \(\nu\) is also injective, and all these data are functorial in the
\(F\)-crystal.

\emph{Step 3: identifying \(K_{-1}\) with \(\ker(dV)\).}
We suppress Frobenius twists in this step. Let
\(W\Omega_X^\bullet(-1)\) be the complex
\[
  0\longrightarrow \sigma_*W\mathcal O_X
    \xrightarrow{\ dV\ } W\Omega_X^1
    \xrightarrow{\ d\ } W\Omega_X^2
    \xrightarrow{\ d\ }\cdots ,
\]
the Nygaard modification of \(W\Omega_X^\bullet\) by \((-\delta_0)\)
of Definition~\ref{def:autoequivalence}, up to a global Frobenius twist. The
relations \(FdV=d\) and \(dF=pFd\) make
\(\mathcal V:=(V,1,1,\ldots)\) and
\(\varphi:=(1,F,pF,p^2F,\ldots)\) chain maps to
\(W\Omega_X^\bullet\). Since \(V\) is injective on \(W\mathcal O_X\) with
cokernel \(\mathcal O_X\), the map \(\mathcal V\) identifies
\(W\Omega_X^\bullet(-1)\) with the first term
\([VW\mathcal O_X\to W\Omega_X^1\to W\Omega_X^2\to\cdots]\) of the Nygaard
filtration \cite[Definition~8.1]{BMS19} (cf.\ \cite[Section~1.3]{GSY25}) and
fits into
\[
  0\longrightarrow W\Omega_X^\bullet(-1)\xrightarrow{\ \mathcal V\ }
  W\Omega_X^\bullet\longrightarrow\mathcal O_X[0]\longrightarrow0 .
\]
Set \(N:=\mathbb H^2(X,W\Omega_X^\bullet(-1))\) and, by abuse of notation,
also write \(\mathcal V,\varphi:N\to M\) for the induced maps. The chain-level
identity \((F,pF,p^2F,\ldots)\circ\mathcal V=p\varphi\) gives
\(\Phi\circ\mathcal V=p\varphi\), hence
\(\mathcal V=V_K\circ\varphi\) after
inverting \(p\).

Put \(L':=\mathrm{Fil}^1_{\mathrm{HW}}N\). The edge argument of Step~1, applied to the
modification, gives isomorphisms
\(\epsilon_0:M/L\xrightarrow{\sim}\ker d\) and
\(\epsilon_{-1}:N/L'\xrightarrow{\sim}\ker(dV)\). Functoriality of the two
chain maps gives
\(\epsilon_0(\varphi(n)+L)=\epsilon_{-1}(n+L')\) in
\(\ker(dV)\subseteq\ker d\). On the edges,
\(\mathcal V\) induces \(V\) and \(\varphi\) the inclusion.

We claim that \(\varphi(N)=V_K^{-1}M\). If \(n\in N\), then
\(V_K(\varphi(n))=\mathcal V(n)\in M\), so
\(\varphi(N)\subseteq V_K^{-1}M\).
Conversely, let \(m\in V_K^{-1}M\) and \(u:=V_K(m)\in M\), so
\(\Phi(u)=pm\in pM\), i.e.\ \(u\) lies in the first term
\(\mathrm{Fil}^1_{\mathcal N}M:=\Phi^{-1}(pM)\) of the Nygaard filtration of
the \(F\)-crystal \(M\). The map \(M\to H^2(X,\mathcal O_X)\) in the long exact
sequence factors through the zeroth graded piece
\(H^2_{\mathrm{dR}}(X/k)/\mathrm{Fil}^1_{\mathrm{H}}\) of the Hodge filtration.
By the Berthelot--Ogus form of Mazur's theorem
\cite[Theorem~8.26(1)]{BO78}, the reduction map sends
\(\mathrm{Fil}^1_{\mathcal N}M\) onto the Hodge filtration
\(\mathrm{Fil}^1_{\mathrm{H}}\), up to the Frobenius twist suppressed above. Thus the reduction of \(u\) lies in
\(\mathrm{Fil}^1_{\mathrm{H}}\), so its image in \(H^2(X,\mathcal O_X)\)
vanishes. Hence \(u=\mathcal V(n)\) for some
\(n\in N\), and
\(V_K(\varphi(n))=u=V_K(m)\), so \(m=\varphi(n)\), \(V_K\) being
injective on \(M[1/p]\).

The equality
\(\epsilon_0(\varphi(n)+L)=\epsilon_{-1}(n+L')\) and the injectivity of
\(\epsilon_0,\epsilon_{-1}\) give \(\varphi(n)\in L\) if and only if
\(n\in L'\). Therefore \(\varphi\) induces
\[
  \ker(dV)\;\simeq\;N/L'\;\xrightarrow{\ \sim\ }\;
  V_K^{-1}M\big/(L\cap V_K^{-1}M)\;=\;K_{-1} .
\]
Under it, \(\iota\) becomes the inclusion \(\ker(dV)\subseteq\ker d\), and
\(\nu=[V_K]\) becomes \([\mathcal V]=V|_{\ker(dV)}\), by the edge comparison.
This identifies the crystalline candidate with the desired Nygaard pair.
Theorem~\ref{thm:two-kernel-presentation} now reconstructs \(M_0\) from this
pair.
\end{proof}

\section{Supersingular abelian varieties}\label{sec:supersingular}

Throughout this section \(k\) is algebraically closed, as the results below rest
on the classification of supersingular abelian varieties and their Dieudonn\'e modules.
We now specialize the reconstruction to a supersingular abelian \(g\)-fold \(A\).
In this case \(H^2_{\mathrm{crys}}(A/W)[1/p]\) has pure slope \(1\). The slope
spectral sequence identifies \(H^2(A,W\mathcal O_A)[1/p]\) with its
slope-\([0,1)\) part, which is zero. Since the core \(C_A^{0,2}\) is finite
free, it follows that \(C_A^{0,2}=0\). Proposition~\ref{prop:diagonal-ordinary-degree-two-domino}
then gives
\(H^2(A,W\Omega_A^\bullet)^{[0,0]}=U_A^{0,2}\in\operatorname{Dom}^{+}\).

\begin{definition}[Numerical invariants in degree two]\label{def:degree-two-artin-invariant}
Let \(A/k\) be a supersingular abelian \(g\)-fold, \(g\ge2\). We study the
\(p\)-exponent and the degree of its domino \(U_A^{0,2}\). We call the latter
the \emph{degree-two Artin invariant} of \(A\) and write
\(\sigmaArt(A):=\deg(U_A^{0,2})\).
\end{definition}

Subsection~\ref{sec:cyclic-matrices} treats the \emph{supersingular cyclic
\(F\)-crystals}: these are the \(F\)-crystals \(H^1_{\mathrm{crys}}(A/W)\) admitting a
basis that Frobenius permutes cyclically, multiplying successive basis vectors
by a unit or by \(p\) times a unit. The resulting combinatorial algorithm
recovers the type sequence \(J(U_A^{0,2})\) and isogeny partition
\(\lambda(U_A^{0,2})\). Subsection~\ref{subsec:supergeneral-superspecial}
applies the algorithm to the supergeneral and superspecial families. Finally,
Subsection~\ref{subsec:extremal-exponent} gives bounds for both
\(p\text{-}\exp(U_A^{0,2})\) and \(\sigmaArt(A)\), with a principal polarization
assumed for the upper bound on \(\sigmaArt(A)\). In particular, when \(g\ge3\),
the equality \(p\text{-}\exp(U_A^{0,2})=g-1\) characterizes the supergeneral
locus. Before giving the algorithm, we record the numerical input that
determines the dimension of the domino.

\begin{definition}[Slope and Hodge--Witt numbers {\cite[Definition~IV.3.1]{ekedahl3}}]\label{def:slope-hw-numbers}
Let \(X/k\) be a smooth projective variety of dimension \(d\). For each \(i,j\ge0\),
the \emph{slope number} is
\[
  m^{i,j}
  =
  \sum_{\lambda\in[i-1,\,i)}
    (\lambda-i+1)\,\dim_K H^{i+j}_{\mathrm{crys}}(X/W)_{[\lambda]}
  +
  \sum_{\lambda\in[i,\,i+1)}
    (i+1-\lambda)\,\dim_K H^{i+j}_{\mathrm{crys}}(X/W)_{[\lambda]},
\]
where \([\lambda]\) denotes the slope-\(\lambda\) part after inverting \(p\). The
\emph{Hodge--Witt number} is
\[
  h_W^{i,j}=m^{i,j}+T^{i,j}-2T^{i-1,j+1}+T^{i-2,j+2},
\]
where \(T^{i,j}=T^{i,j}(X)\) is the domino number of
Definition~\ref{def:domino-numbers}. By
Proposition~\ref{prop:ekedahl-duality}, these numbers vanish outside the
domino range \(0\le i\le d-2\), \(2\le j\le d\).
\end{definition}

\begin{proposition}[Ekedahl {\cite[Corollary~IV.3.3.1]{ekedahl3}}]\label{prop:domino-numbers}
For a Mazur--Ogus variety \(X/k\), one has \(h^{i,j}=h_W^{i,j}\), and hence
\[
  T^{i,j}=h^{i,j}-m^{i,j}+2T^{i-1,j+1}-T^{i-2,j+2}.
\]
In particular, \(T^{0,j}=h^{0,j}-m^{0,j}\).
\end{proposition}

\begin{corollary}\label{cor:domino-supersingular}
For an abelian \(g\)-fold \(A/k\), the domino number \(T^{0,2}(A)\) depends only
on its Newton polygon. If \(A\) is supersingular, then
\(T^{0,2}(A)=\binom{g}{2}\).
\end{corollary}

\begin{proof}
Proposition~\ref{prop:domino-numbers} gives
\(T^{0,2}(A)=h^{0,2}(A)-m^{0,2}(A)\), with
\(h^{0,2}(A)=\binom{g}{2}\). The number \(m^{0,2}(A)\) depends only on the slopes
of \(H^2_{\mathrm{crys}}(A/W)[1/p]\), hence on the Newton polygon of \(A\). In
the supersingular case these slopes are all \(1\), so \(m^{0,2}(A)=0\).
\end{proof}

\subsection{\texorpdfstring{Supersingular cyclic \(F\)-crystals}{Supersingular cyclic F-crystals}}\label{sec:cyclic-matrices}

For this cyclic class, the Nygaard-pair reconstruction becomes combinatorial. A cyclic basis of
\(H^1_{\mathrm{crys}}(A/W)\) decomposes its exterior square into cyclic
Frobenius blocks. Supersingularity makes
\(H^2_{\mathrm{crys}}(A/W)\) isoclinic of slope \(1\). Hence
\(\mathrm{Fil}^1_{\mathrm{HW}}\) has full rank and
\(K_0=H^2_{\mathrm{crys}}(A/W)/\mathrm{Fil}^1_{\mathrm{HW}}\) has finite
\(W\)-length. On each block, the Frobenius valuations determine a normalized
height function \(\mathbf h=(h_1,\ldots,h_n)\). This function determines the
associated Nygaard pair, and thereby the type sequence and isogeny partition
of the corresponding domino.

\begin{definition}[Supersingular cyclic $F$-crystal]\label{def:cyclic-crystal}
Let \(A/k\) be a supersingular abelian \(g\)-fold. The \(F\)-crystal
\((H^1_{\mathrm{crys}}(A/W),\varphi)\) is a \emph{supersingular cyclic
\(F\)-crystal} if it admits a \(W\)-basis \(e_1,\ldots,e_{2g}\) such that
\[
  \varphi(e_i)=p^{\varepsilon_i}u_i e_{i+1},
  \qquad u_i\in W^\times,\quad \varepsilon_i\in\{0,1\},
\]
where the indices are cyclic, and \(\sum_i\varepsilon_i=g\). The cyclically
ordered binary tuple
\(\boldsymbol\varepsilon=(\varepsilon_1,\ldots,\varepsilon_{2g})\), considered
up to rotation, is its \emph{Frobenius valuation word}. Indeed,
\(\varphi^{2g}\) sends each basis vector to \(p^{\sum_i\varepsilon_i}\) times a
unit multiple of itself, so the last condition is equivalent to Newton slope
\(\frac12\).
\end{definition}

\begin{theorem}[Cyclic matrix algorithm]\label{thm:cyclic-matrix}
Let \(A/k\) be a supersingular abelian \(g\)-fold whose
\(H^1_{\mathrm{crys}}(A/W)\) is supersingular cyclic
(Definition~\ref{def:cyclic-crystal}), with cyclic basis \(e_1,\ldots,e_{2g}\)
and Frobenius valuation word
\(\boldsymbol\varepsilon=(\varepsilon_1,\ldots,\varepsilon_{2g})\). Then
\(\boldsymbol\varepsilon\) determines the positive domino \(U_A^{0,2}\) up to
isomorphism. Its Nygaard pair, type sequence, and isogeny partition are
obtained by the following combinatorial procedure.

\begin{enumerate}
\item[\textbf{Step 1.}] \textbf{(Cyclic Frobenius blocks.)}
Put \(\Phi=\bigwedge^2\varphi\). The cyclic basis gives
\[
  H^2_{\mathrm{crys}}(A/W)
  \cong\bigwedge^2H^1_{\mathrm{crys}}(A/W)
  =\bigoplus_{s=1}^g B^{(s)}.
\]
For \(1\le s\le g-1\), \(B^{(s)}\) is spanned by
\(f^{(s)}_j=e_j\wedge e_{j+s}\), \(j\in\mathbb Z/2g\mathbb Z\), and has block
valuation sequence \(\mathbf v^{(s)}=(v^{(s)}_j)_j\), where
\(v^{(s)}_j=\varepsilon_j+\varepsilon_{j+s}\). For \(s=g\), the orbit has
length \(g\): for \(j\in\mathbb Z/g\mathbb Z\), put
\(f^{(g)}_j=e_j\wedge e_{j+g}\) and
\(v^{(g)}_j=\varepsilon_j+\varepsilon_{j+g}\). After \(g\) shifts the basis
vector returns as \(e_{j+g}\wedge e_j=-f^{(g)}_j\). This sign is absorbed into
the unit coefficient.

\item[\textbf{Step 2.}] \textbf{(Nygaard pair of a block.)}
Apply the following construction to each cyclic block
\(B=\bigoplus_{i\in\mathbb Z/n\mathbb Z}Wf_i\). Write
\(\Phi(f_i)=u_ip^{v_i}f_{i+1}\) with \(u_i\in W^\times\). Supersingularity
gives \(\sum_i v_i=n\). Hence there is a unique function
\(\mathbf h\colon\mathbb Z/n\mathbb Z\to\mathbb Z_{\ge0}\), called the
\emph{normalized height function}, such that \(h_{i+1}-h_i=v_i-1\) and
\(\min_i h_i=0\).

The recurrence gives
\(\Phi(L_{\mathbf h})\subset pL_{\mathbf h}\) for
\(L_{\mathbf h}:=\bigoplus_ip^{h_i}Wf_i\). Conversely, since \(\Phi\)
cyclically permutes the lines \(Wf_i\), the intersection formula of
Proposition~\ref{prop:HW-filtration} is diagonal, say
\(\bigoplus_ip^{\ell_i}Wf_i\). Its exponents satisfy
\(\ell_{i+1}\le \ell_i+v_i-1\). Their sum is unchanged around the
cycle because \(\sum_i(v_i-1)=0\), so every inequality is an equality.
Thus \(\ell_i=h_i+c\) for all \(i\), and maximality together with
\(\min_i h_i=0\) forces \(c=0\). Hence
\(\mathrm{Fil}^1_{\mathrm{HW}}B=L_{\mathbf h}\), and, by
Theorem~\ref{thm:geometric-ekedahl-pair-partial-v},
\[
  K_0(B):=B/L_{\mathbf h}=\bigoplus_i(W/p^{h_i}W)\,\overline f_i,
  \qquad \overline f_i:=f_i\bmod L_{\mathbf h} .
\]
This is the contribution of \(B\) to \(K_0=\ker(d)\). The corresponding
summand \(K_{-1}(B)\subseteq K_0(B)\) is given
on each direct summand by
\(K_{-1}(B)\cap(W/p^{h_i}W)\overline f_i
=p^{(v_{i-1}-1)_+}(W/p^{h_i}W)\overline f_i\), and \(\nu\) acts there by
\[
  \nu(p^ra\overline f_i)
  =p^{\,r+1-v_{i-1}}\,\sigma^{-1}(u_{i-1}^{-1}a)\,\overline f_{i-1}.
\]
Lemma~\ref{lem:cyclic-unit-normalization} normalizes all \(u_i\) to \(1\).
Thus the normalized height function determines the block's Nygaard pair
\((K_0(B),K_{-1}(B),\iota,\nu)\), and hence its positive domino, up to
isomorphism. The direct sum over the blocks is the Nygaard pair of
\(U_A^{0,2}\).

\item[\textbf{Step 3.}] \textbf{(Type sequence and isogeny partition.)}
For \(q\ge1\), the superlevel set
\(\{\mathbf h^{(s)}\ge q\}:=\{i:h^{(s)}_i\ge q\}\) is a disjoint union of proper
cyclic arcs. An arc \(I\subseteq\{\mathbf h^{(s)}\ge q\}\) has length \(|I|\) and
height \(\operatorname{ht}(I):=\max_{i\in I}h^{(s)}_i-q+1\). Letting \(I\) range
over the arcs of every block and every superlevel set, one has
\[
  J(U_A^{0,2})=\{\!\{\,|I|\,\}\!\},
  \qquad
  \lambda(U_A^{0,2})^\vee_r=\#\{\,I:\operatorname{ht}(I)=r\,\}.
\]
In particular \(\deg(U_A^{0,2})=\sum_{s,i}h^{(s)}_i\) and
\(p\text{-}\exp(U_A^{0,2})=\max_{s,i}h^{(s)}_i\).
\end{enumerate}
\end{theorem}

It remains to justify the two blockwise claims used in the algorithm. Fix a
normalized height function
\(\mathbf h=(h_i)_{i\in\mathbb Z/n\mathbb Z}\), so \(\min_i h_i=0\) and
\(|h_{i+1}-h_i|\le1\), and put \(v_i=h_{i+1}-h_i+1\). Together with units
\(\mathbf u=(u_i)\in(W^\times)^n\), the formulas of Step~2 define a Nygaard
pair. Indeed, a componentwise check gives
\(pK_0\subseteq K_{-1}\cap\nu(K_{-1})\), while iterating \(\nu\) along
the finite interval components of \(\{\mathbf h\ge1\}\) verifies the ladder
condition. We call a Nygaard pair \emph{cyclic} if it is isomorphic to one obtained
in this way. We first remove the units and then read the two combinatorial
invariants from \(\mathbf h\).

\begin{lemma}[Normal form]\label{lem:cyclic-unit-normalization}
Every cyclic Nygaard pair presented by \((\mathbf h,\mathbf u)\) is isomorphic
to the one presented by \((\mathbf h,\mathbf 1)\), where
\(\mathbf1=(1,\ldots,1)\). Consequently the pair, and
hence the positive domino it reconstructs, is determined up to isomorphism by
\(\mathbf h\).
\end{lemma}

\begin{proof}
The summands indexed by \(h_i=0\) vanish, so the support decomposes into the
interval components of \(\{\mathbf h\ge1\}\). Replacing \(e_i\) by
\(\omega_i e_i\), with
\(\omega_i\in W^\times\), multiplies the coefficient of
\(e_i\mapsto e_{i-1}\) by
\(\sigma^{-1}(\omega_i)\omega_{i-1}^{-1}\). On each interval the \(\omega_i\)
can therefore be chosen recursively to make every coefficient \(1\). There is
no cyclic compatibility condition. Applying the same diagonal rescaling to
\(K_{-1}\) and \(K_0\) preserves \(\iota\), and hence gives the required
isomorphism of Nygaard pairs.
\end{proof}

In the next two propositions, \(U_{\mathbf h}\) denotes the positive domino of
the normalized pair with height function \(\mathbf h\). By the lemma it is well
defined up to isomorphism.

\begin{proposition}[Type sequence]
\label{prop:cyclic-torsion-layers-type}
The type sequence of \(U_{\mathbf h}\) is the multiset of arc lengths:
\(J(U_{\mathbf h})=\{\!\{\,|I|:r\ge1,\ I\in\pi_0(\{\mathbf h\ge r\})\,\}\!\}\).
\end{proposition}

\begin{proof}
For \(r\ge1\), take \(K_0[p^r]\) and \(K_{-1}[p^r]\) as the two components of
the \(r\)-th filtration term. The map \(\iota\) carries the second into the
first, while \(\sigma(p)=p\) and the \(\sigma^{-1}\)-linearity of \(\nu\) give
\(p^r\nu(x)=\nu(p^rx)=0\) for \(x\in K_{-1}[p^r]\). Thus the restrictions of
\(\iota\) and \(\nu\) preserve these submodules. Condition~(1) of
Definition~\ref{def:nygaard-pair} is inherited as well: if
\(px=\nu(y)\) with \(x\in K_0[p^r]\), then the injectivity of \(\nu\) gives
\(p^ry=0\). Finally, the kernel ladder terminates because \(\nu\) moves along
the finite interval components of \(\{\mathbf h\ge1\}\). Hence each level is
a Nygaard subpair. The exact equivalence of
Theorem~\ref{thm:two-kernel-presentation} transports it to an admissible
filtration of \(U_{\mathbf h}\). In its \(r\)-th graded piece, the
\(K_0\)-part has basis \(p^{h_i-r}e_i\) for \(i\in\{\mathbf h\ge r\}\), and
\(\nu\) joins precisely the consecutive indices in each component of
\(\{\mathbf h\ge r\}\). Thus this graded piece is
\(\bigoplus_{I\in\pi_0(\{\mathbf h\ge r\})}U_{|I|}\).

We now read off the type sequence from the kernel lengths of
Remark~\ref{rem:kernel-lengths}, which is legitimate for any domino. By
Theorem~\ref{thm:two-kernel-presentation} the kernels of \(U_{\mathbf h}\) are
the terms \(K_{-t}\) of the ladder, so it suffices to compute
\(\lgth_WK_{-t}\). Call \(s\) the \emph{coheight} of a nonzero element
\(p^{h_i-s}a\overline f_i\) of \(K_0\), so that \(1\le s\le h_i\). Since
\(v_{i-1}=h_i-h_{i-1}+1\), the formula for \(\nu\) in Step~2 reads
\[
  \nu\bigl(p^{h_i-s}a\overline f_i\bigr)
  =p^{h_{i-1}-s}\,\sigma^{-1}(a)\,\overline f_{i-1} :
\]
it preserves the coheight and moves the index one step to the left, and the
result is nonzero exactly when \(s\le h_{i-1}\). Iterating, an element of
coheight \(s\) at \(i\) lies in \(K_{-t}\) if and only if
\(h_{i-1},\ldots,h_{i-t}\ge s\). Hence
\[
  \lgth_WK_{-t}
  =\sum_i\min\{h_i,h_{i-1},\ldots,h_{i-t}\}
  =\sum_{r\ge1}\ \sum_{I\in\pi_0(\{\mathbf h\ge r\})}\max(0,|I|-t),
\]
the second equality by counting, for each \(r\), the indices \(i\) whose
component of \(\{\mathbf h\ge r\}\) contains \(i-1,\ldots,i-t\) as well. This
is \eqref{eq:kernel-lengths} for the multiset of arc lengths. Since
\(\max(0,\ell-t)\) has second difference \(\delta_{t\ell}\) in \(t\), comparing
with \eqref{eq:kernel-lengths} for \(U_{\mathbf h}\) itself gives
\(m_t(U_{\mathbf h})=\#\{I:|I|=t\}\), which is the formula.
\end{proof}

\begin{proposition}[Isogeny partition]
\label{prop:cyclic-isogeny-partition}
For \(I\in\pi_0(\{\mathbf h\ge q\})\), put
\(\operatorname{ht}(I)=\max_{i\in I}h_i-q+1\). Then
\(\lambda(U_{\mathbf h})^\vee_r=\#\{\,I:\operatorname{ht}(I)=r\,\}\) for \(r\ge1\).
In particular \(p\text{-}\exp(U_{\mathbf h})=\max_i h_i\).
\end{proposition}

\begin{proof}
Let \(W_\sigma((V)):=W_\sigma[[V]][V^{-1}]=\widehat R^0[V^{-1}]\) and
\(T=U_{\mathbf h}^0[V^{-1}]\). Its \(p\)-adic completion
\(W_\sigma((V))^\wedge_p\) is a complete skew discrete valuation ring with
uniformizer \(p\) and residue division ring
\(W_\sigma((V))^\wedge_p/p\simeq k_\sigma((V))\). Since every domino is
killed by a power of \(p\), the \(W_\sigma((V))\)-module structure on
\(U^0[V^{-1}]\) extends uniquely to its completion. Exactness of localization
and the type filtration give
\(\lgth_{W_\sigma((V))^\wedge_p}(U^0[V^{-1}])=\dim(U)\): both sides are additive, and
each elementary domino \(U_j\) contributes one. Applying this to
\(p^qU_{\mathbf h}\), whose localized degree-zero term is \(p^qT\), gives
\(\lgth_{W_\sigma((V))^\wedge_p}(p^qT)=d_q(U_{\mathbf h})\) for all \(q\ge0\).

Here the Nygaard pair enters through Theorem~\ref{thm:two-kernel-presentation}:
\(U_{\mathbf h}^0=K_0+VU_{\mathbf h}^0\), so \(K_0\) generates \(T\) after
inverting \(V\), and \(\nu\) is the restriction of \(V\) to \(K_{-1}\).
Localizing the torsion filtration used in the preceding proof therefore gives
an associated graded module over \(k_\sigma((V))\). For an arc
\(I\subseteq\{\mathbf h\ge r\}\), the
formula for \(\nu\) in Step~2 identifies the classes
\(p^{h_i-r}e_i\), \(i\in I\), up to units after \(V\) is inverted. Denote the
resulting basis vector by \(x_I\). Multiplication by \(p\) sends these classes
to \(p^{h_i-(r-1)}e_i\). Thus, if \(I'\) is the component of
\(\{\mathbf h\ge r-1\}\) containing \(I\), then
\(x_I\mapsto(\mathrm{unit})x_{I'}\). A level-one arc maps to zero. Hence the image of
its \(q\)-th power is spanned by the arcs admitting a descendant \(q\) levels
deeper, namely those with \(\operatorname{ht}(I)>q\). The standard filtration
formula for a finite-length module over a skew discrete valuation ring is
\[
  \lgth_{W_\sigma((V))^\wedge_p}(p^qT)
  =\dim_{k_\sigma((V))}
    \operatorname{im}\bigl(\operatorname{gr}(p)^q\bigr),
\]
where \(\operatorname{gr}(p)\) is the parent map above. Consequently,
\[
  d_q(U_{\mathbf h})=\lgth_{W_\sigma((V))^\wedge_p}(p^qT)
  =\#\{I:\operatorname{ht}(I)>q\}.
\]
Therefore Theorem~\ref{thm:isogeny-partition} yields
\(\lambda(U_{\mathbf h})^\vee_r=d_{r-1}-d_r
=\#\{I:\operatorname{ht}(I)=r\}\), and
\(p\text{-}\exp(U_{\mathbf h})=\lambda_1
=\max_I\operatorname{ht}(I)=\max_i h_i\).
\end{proof}

\begin{remark}[The weighted forest]\label{rem:forest-picture}
Ordered by containment, the arcs form a rooted forest: a component
\(I\subseteq\{\mathbf h\ge r\}\), \(r\ge2\), is joined to the unique component
\(I'\subseteq\{\mathbf h\ge r-1\}\) containing it. The vertex \(I\) carries
the type \(|I|\). To describe the edge weight, lift each root arc to an
interval in \(\mathbb Z\), use the induced lifts for its descendants, and let
\(m(I)\) be the right endpoint. The edge \(I\to I'\) then represents the
extension of \(U_{|I|}\) by \(U_{|I'|}\) with class
\(V^{m(I')-m(I)-1}\): the \(\nu\)-string for \(I\) begins exactly
\(m(I')-m(I)-1\) positions from the corresponding end of the \(I'\)-string,
and the coordinates of Theorem~\ref{thm:2d-classification} translate this
offset into the indicated power of \(V\). Although
Theorem~\ref{thm:2d-classification} allows an arbitrary polynomial of degree
at most \(|I'|-|I|-2\), the cyclic construction forces this
monomial.
Lemma~\ref{lem:cyclic-unit-normalization} therefore shows that
the weighted forest determines \(U_{\mathbf h}\).

After the edge weights are forgotten, multiplication by \(p\) is the parent
map of Proposition~\ref{prop:cyclic-isogeny-partition}, whose Jordan type is
\(\lambda(U_{\mathbf h})\) and whose longest part is
\(\lambda_1=\max_i h_i\). In
Example~\ref{ex:nested-arcs-isogeny}, two arcs have a common parent, so the
isogeny partition cannot be computed by assigning one part to each root.
\end{remark}

\begin{proof}[Proof of Theorem~\ref{thm:cyclic-matrix}]
Step~1 is the exterior-square decomposition of the cyclic \(F\)-crystal
\(H^1_{\mathrm{crys}}(A/W)\).
Proposition~\ref{prop:HW-filtration} and
Theorem~\ref{thm:geometric-ekedahl-pair-partial-v} give Step~2 blockwise, while
Lemma~\ref{lem:cyclic-unit-normalization} removes the units. These constructions
commute with finite direct sums, so Propositions~\ref{prop:cyclic-torsion-layers-type}
and~\ref{prop:cyclic-isogeny-partition} give Step~3. Finally, each index \(i\)
belongs to exactly \(h_i^{(s)}\) superlevel arcs. Hence
\(\deg(U_A^{0,2})=\sum_{s,i}h_i^{(s)}\), while
\(p\text{-}\exp(U_A^{0,2})=\lambda_1=\max_{s,i}h_i^{(s)}\).
\end{proof}

\begin{example}[Nested arcs and the isogeny partition]
\label{ex:nested-arcs-isogeny}
Consider the cyclic height function \(\mathbf h\) on \(\mathbb Z/20\mathbb Z\)
with a unique zero and, in cyclic order after it, the values
\[
  (1,1,1,1,2,3,4,4,4,3,2,2,3,3,2,1,1,1,1),
\]
chosen so that two superlevel arcs have a common parent. The left diagram
omits the zero. Its columns are
the summands \((W/p^{h_i})e_i\) of \(K_0=\ker d\), and each arrow is the partial-\(V\)
string on one graded piece of the \(p\)-torsion filtration of
Proposition~\ref{prop:cyclic-torsion-layers-type}: the \(r\)-th piece is
spanned by the classes \(p^{h_i-r}e_i\) along the arcs of \(\{\mathbf h\ge r\}\),
with one string and one elementary summand \(U_{|I|}\) for each arc \(I\).

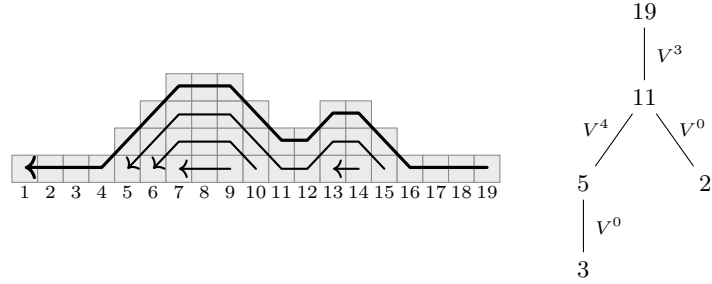
\begin{center}
\begin{tikzpicture}[x=0.34cm,y=0.36cm,line cap=round,line join=round,
    baseline={(current bounding box.center)}]
  \def\block#1#2{\draw[fill=black!8,draw=black!45] (#1,#2) rectangle ++(1,1);}
  \foreach \x/\h in {0/1,1/1,2/1,3/1,4/2,5/3,6/4,7/4,8/4,9/3,10/2,11/2,12/3,13/3,14/2,15/1,16/1,17/1,18/1}
    \foreach \y in {1,...,\h} {
      \pgfmathtruncatemacro{\yy}{\y-1}
      \block{\x}{\yy}
    }

  \foreach \x/\lab in {0/1,1/2,2/3,3/4,4/5,5/6,6/7,7/8,8/9,9/10,10/11,11/12,12/13,13/14,14/15,15/16,16/17,17/18,18/19}
    \node[font=\scriptsize] at (\x+0.5,-0.35) {\lab};

  \draw[->,very thick]
    (18.50,0.55) -- (17.50,0.55) -- (16.50,0.55) -- (15.50,0.55) --
    (14.50,1.55) -- (13.50,2.55) -- (12.50,2.55) -- (11.50,1.55) --
    (10.50,1.55) -- (9.50,2.55) -- (8.50,3.55) -- (7.50,3.55) --
    (6.50,3.55) -- (5.50,2.55) -- (4.50,1.55) --
    (3.50,0.55) -- (2.50,0.55) -- (1.50,0.55) -- (0.50,0.55);
  \draw[->,thick]
    (14.50,0.50) -- (13.50,1.50) -- (12.50,1.50) -- (11.50,0.50) --
    (10.50,0.50) -- (9.50,1.50) -- (8.50,2.50) -- (7.50,2.50) --
    (6.50,2.50) -- (5.50,1.50) -- (4.50,0.50);
  \draw[->,thick]
    (9.50,0.50) -- (8.50,1.50) -- (7.50,1.50) -- (6.50,1.50) --
    (5.50,0.50);
  \draw[->,thick]
    (13.50,0.50) -- (12.50,0.50);
  \draw[->,thick]
    (8.50,0.50) -- (7.50,0.50) -- (6.50,0.50);

\end{tikzpicture}
\qquad
\begin{tikzpicture}[x=0.95cm,y=0.95cm,line cap=round,line join=round,
    baseline={(current bounding box.center)},every node/.style={font=\small}]
  \node (n19) at (0,3.6)     {\(19\)};
  \node (n11) at (0,2.4)     {\(11\)};
  \node (n5)  at (-0.85,1.2) {\(5\)};
  \node (n2)  at (0.85,1.2)  {\(2\)};
  \node (n3)  at (-0.85,0)   {\(3\)};
  \draw (n19) -- (n11) node[midway,right=1pt,font=\scriptsize] {\(V^{3}\)};
  \draw (n11) -- (n5)  node[midway,above left=-2pt,font=\scriptsize] {\(V^{4}\)};
  \draw (n11) -- (n2)  node[midway,above right=-2pt,font=\scriptsize] {\(V^{0}\)};
  \draw (n5)  -- (n3)  node[midway,right=1pt,font=\scriptsize] {\(V^{0}\)};
\end{tikzpicture}
\end{center}

There are five arcs, of lengths \(19,11,5,3,2\) and heights \(4,3,2,1,1\),
nested as in the tree on the right. By Theorem~\ref{thm:cyclic-matrix},
\[
  J(U_{\mathbf h})=(2,3,5,11,19),
  \qquad
  \lambda(U_{\mathbf h})^\vee=(2,1,1,1),
  \qquad
  \lambda(U_{\mathbf h})=(4,1),
\]
with \(\deg(U_{\mathbf h})=\sum_ih_i=40\) and
\(p\text{-}\exp(U_{\mathbf h})=\max_ih_i=4\). The
weights \(V^3,V^4,V^0,V^0\) shown on the tree edges are the pure-\(V\)-power
extension classes of Remark~\ref{rem:forest-picture}. The two arcs at level
\(3\) lie in a single arc at level \(2\). Thus a single root can contribute
more than one part to \(\lambda(U_{\mathbf h})\).
\end{example}

For \(g=2\), the algorithm has the following two specializations.

\begin{example}[The two abelian-surface cases]
\label{ex:superspecial-g2}
\label{ex:supergeneral-g2}
For a superspecial abelian surface \(A\),
\(\boldsymbol\varepsilon=(0,1,0,1)\), the two blocks give
\(\mathbf h^{(1)}=(0,0,0,0)\), \(\mathbf h^{(2)}=(1,0)\). Hence there is a single arc, of length
\(1\), and \(U_A^{0,2}\cong U_1\). For a cyclic supergeneral surface \(A\),
\(\boldsymbol\varepsilon=(0,0,1,1)\), the blocks give
\(\mathbf h^{(1)}=(1,0,0,1)\), \(\mathbf h^{(2)}=(0,0)\). The positive cyclic arc has length \(2\), so
\(U_A^{0,2}\cong U_2\). In both cases \(\lambda=(1)\), \(p\text{-}\exp=1\),
\(\widehat{\mathrm{Br}}_A\cong\widehat{\mathbb G}_a\), and
\(\mathrm{Br}^{\mathrm{perf}}_A\cong\mathbb G_a^{\mathrm{perf}}\). For abelian
surfaces, the elementary type is the usual Artin invariant, giving the well-known
values \(\sigmaArt=1\) and \(\sigmaArt=2\), respectively.
\end{example}

\subsection{Supergeneral and superspecial cases}\label{subsec:supergeneral-superspecial}

This subsection treats two extremal strata of the supersingular locus
\(\mathcal S_g\). A \emph{superspecial} abelian variety is isomorphic over
\(\overline{k}\) to a product of supersingular elliptic curves, whereas the
\emph{supergeneral} locus is the open dense stratum
\(\mathcal S_{g,a=1}\) defined by \(a(A)=1\). By
Lemma~\ref{lem:supersingular-brauer-unipotent}, the formal and perfect Brauer
groups of a supersingular \(A\) are respectively the formal and syntomic realizations
\(\widehat G(U_A^{0,2})\) and \(G^{\mathrm{perf}}(U_A^{0,2})\). We begin
with a cyclic family in \(\mathcal S_{g,a=1}\) to which the algorithm of
Theorem~\ref{thm:cyclic-matrix} applies directly.

\begin{theorem}[A cyclic supergeneral family]\label{thm:supergeneral}
Let \(g\ge1\). There exists a principally polarized supersingular abelian
\(g\)-fold \((A,\lambda)\) over \(k\) whose \(p\)-divisible group has
Dieudonn\'e module \(R^0/R^0(F^g-V^g)\), hence \(a(A)=1\). For any such \(A\),
\[
  U_A^{0,2}
  \cong
  \bigoplus_{i=1}^{g-1} U_{i+1,\,i+3,\,\ldots,\,2g-1-i}.
\]
Consequently,
\(\widehat{\mathrm{Br}}_A\cong\bigoplus_{r=1}^{g-1}\widehat W_r\) and
\(\mathrm{Br}^{\mathrm{perf}}_A\cong\bigoplus_{r=1}^{g-1}W_r^{\mathrm{perf}}\).
\end{theorem}

\begin{proof}
Put \(M_g:=R^0/R^0(F^g-V^g)\). Since \(F^{2g}=p^g\) on \(M_g\),
Dieudonn\'e theory identifies it with a height-\(2g\), dimension-\(g\)
\(p\)-divisible group isoclinic of slope \(1/2\). This group admits a
principal quasi-polarization by \cite[Example~3.3]{NV07} and is realized by a
principally polarized abelian variety \((A,\lambda)\) by
\cite[Proposition~5.3]{Har10}. The resulting \(A\) is supersingular.

Use the ordered basis \(e_1,\ldots,e_{2g}\) of
\(H^1_{\mathrm{crys}}(A/W)\simeq R^0/R^0(F^g-V^g)\) represented by
\[
  1,F,\ldots,F^{g-1},V^g,V^{g-1},\ldots,V,
\]
where \(V^g=F^g\). Then
\(M_g/(FM_g+VM_g)\simeq k\cdot\overline e_1\), since modulo \(p\) the subspace
\(FM_g+VM_g\) is generated by \(\overline e_2,\ldots,\overline e_{2g}\). Hence every
abelian variety with Dieudonn\'e module \(M_g\) has \(a\)-number \(1\). In this
basis the Frobenius is cyclic, with valuation word
\([\underbrace{0,\ldots,0}_{g},\underbrace{1,\ldots,1}_{g}]\).

Apply Step~1 of Theorem~\ref{thm:cyclic-matrix} to
\(H^2_{\mathrm{crys}}(A/W)=\bigwedge^2H^1_{\mathrm{crys}}(A/W)\), obtaining
cyclic blocks \(B^{(1)},\ldots,B^{(g)}\). For \(1\le s\le g-1\), put
\(m=g-s\). The recurrence in Step~2 gives, up to cyclic rotation, the
normalized height function
\(\mathbf h^{(s)}=(m,m-1,\ldots,1,\underbrace{0,\ldots,0}_{s+1},1,\ldots,m-1,
\underbrace{m,\ldots,m}_{s})\), so for \(1\le r\le m\) the entries with \(h_i^{(s)}\ge r\)
form a single cyclic arc of length \(2g-s-2r+1\). Thus the domino \(U_s\) of
this block has type sequence \((s+1,s+3,\ldots,2g-1-s)\) and
\(p\text{-}\exp(U_s)=m\). Theorem~\ref{thm:distinguished} therefore gives
\(U_s\cong U_{s+1,s+3,\ldots,2g-1-s}\). The last block has
\(\mathbf h^{(g)}=0\) and contributes nothing. Summing over the blocks yields
\[
  U_A^{0,2}\cong
  \bigoplus_{s=1}^{g-1}U_{s+1,\,s+3,\,\ldots,\,2g-1-s}.
\]
The formulas for the two Brauer groups follow from the realizations of
distinguished dominoes in Corollary~\ref{cor:distinguished-props}.
\end{proof}

\begin{theorem}[Superspecial abelian varieties]\label{thm:superspecial}
Let \(A/k\) be a superspecial abelian \(g\)-fold, \(g\ge2\). Then
\[
  U_A^{0,2}\cong U_1^{\oplus\binom{g}{2}}.
\]
Consequently,
\(\widehat{\mathrm{Br}}_A\cong\widehat{\mathbb G}_a^{\oplus\binom{g}{2}}\) and
\(\mathrm{Br}^{\mathrm{perf}}_A\cong
\bigl(\mathbb G_a^{\mathrm{perf}}\bigr)^{\oplus\binom{g}{2}}\).
\end{theorem}

\begin{proof}
Since \(k\) is algebraically closed, \(A\) is isomorphic to a product of \(g\)
supersingular elliptic curves. The K\"unneth formula gives
\(H^1_{\mathrm{crys}}(A/W)=\bigoplus_{i=1}^gR^0/R^0(F-V)\), with valuation word
\([0,1]\) on each summand. Its exterior square decomposes into \(g\) blocks of
valuation sequence \([1]\), \(\binom{g}{2}\) blocks of sequence \([1,1]\), and
\(\binom{g}{2}\) blocks of sequence \([0,2]\). The first two sequences have
zero height function. Each \([0,2]\) block has height function \((1,0)\) and
therefore contributes one copy of \(U_1\). Steps~2 and~3 of
Theorem~\ref{thm:cyclic-matrix} apply to any cyclic \(\Phi\)-block of a
slope-one isocrystal, so they give the stated decomposition, and
Corollary~\ref{cor:distinguished-props}
gives its formal and syntomic realizations.
\end{proof}

For a positive domino, both its degree and dimension can be read directly from
the two kernels of its Nygaard pair.

\begin{proposition}[Degree and dimension formula]
\label{prop:ekedahl-pair-weighs-degree}
Let \(U=[U^0\xrightarrow dU^1]\in\operatorname{Dom}^{+}\). Then
\[
  \lgth_W\ker(d)=\deg(U),
  \qquad
  \dim(U)=\lgth_W\ker(d)-\lgth_W\ker(dV).
\]
\end{proposition}

\begin{proof}
All type factors of \(U\) are positive, so
Remark~\ref{rem:kernel-lengths} gives
\(\lgth_W\ker(d)=\sum_jm_j\max(0,j)=\sum_jjm_j=\deg(U)\) and
\[
  \lgth_W\ker(d)-\lgth_W\ker(dV)
  =\sum_jm_j\bigl(\max(0,j)-\max(0,j-1)\bigr)
  =\sum_{j\ge1}m_j
  =\dim(U).
\]
\end{proof}

For a supersingular abelian \(g\)-fold \(A\),
Corollary~\ref{cor:domino-supersingular} gives
\(\dim(U_A^{0,2})=\binom{g}{2}\), and
Proposition~\ref{prop:diagonal-ordinary-degree-two-domino} gives
\(U_A^{0,2}\in\operatorname{Dom}^{+}\). Thus
\(\sigmaArt(A)\ge\binom{g}{2}\).

Positivity imposes the only obstruction to realizing a distinguished domino
in bidegree \((0,2)\). Indeed,
Proposition~\ref{prop:diagonal-ordinary-degree-two-domino} gives
\(U_A^{0,2}\in\Dom^+\) for every abelian variety \(A\), so its type entries are
positive. The following construction realizes every distinguished domino
with this property.

\begin{corollary}[Geometric realization]\label{cor:geometric-realization}
Every distinguished domino \(U_{j,j+2,\ldots,j+2n-2}\), \(j\ge1\),
appears as a direct summand of the domino part of the de Rham--Witt cohomology
of some principally polarized abelian variety over \(k\).
\end{corollary}

\begin{proof}
Suppose first that \(j\ge2\), and set \(g=n+j-1\) and \(i=j-1\).
Choose a principally polarized member \(A_g\) of the cyclic supergeneral family
in Theorem~\ref{thm:supergeneral}. Its decomposition contains the direct
summand \(U_{i+1,i+3,\ldots,2g-1-i}
  =
  U_{j,j+2,\ldots,j+2n-2}\).

For \(j=1\), choose two principally polarized cyclic supergeneral \(n\)-folds
\((A_n,\lambda_{A_n})\) and \((A'_n,\lambda_{A'_n})\), and give
\(A=A_n\times A'_n\) the product principal polarization. The K\"unneth
decomposition of \(H^2_{\mathrm{crys}}(A/W)\) contains
\[
  H^1_{\mathrm{crys}}(A_n/W)
  \otimes_W
  H^1_{\mathrm{crys}}(A'_n/W).
\]
Choose cyclic bases \(e_1,\ldots,e_{2n}\) and \(f_1,\ldots,f_{2n}\), each having
Frobenius valuation word
\[
  [\,\underbrace{0,\ldots,0}_{n},
     \underbrace{1,\ldots,1}_{n}\,]
\]
up to cyclic rotation. The diagonal tensor block spanned by
\(e_r\otimes f_r\), \(r\in\mathbb Z/2n\mathbb Z\), is cyclic with valuation
sequence
\[
  [\,\underbrace{0,\ldots,0}_{n},
     \underbrace{2,\ldots,2}_{n}\,].
\]
The corresponding height function has one arc of length \(2n-2r+1\) at each
level \(1\le r\le n\). Steps~2 and~3 of
Theorem~\ref{thm:cyclic-matrix}, applied to this block, therefore give type
sequence \((1,3,\ldots,2n-1)\) and \(p\)-exponent \(n\). By
Theorem~\ref{thm:distinguished}, the corresponding domino summand is
\(U_{1,3,\ldots,2n-1}\), completing the proof.
\end{proof}

\subsection{Explicit bounds on the \texorpdfstring{\(p\)}{p}-exponent and \texorpdfstring{\(\sigmaArt\)}{sigma_Art}}
\label{subsec:extremal-exponent}

Throughout this subsection \(A/k\) is a supersingular abelian \(g\)-fold,
\(g\ge2\). Put \(M=H^1_{\mathrm{crys}}(A/W)\),
\(N:=\bigwedge_W^2M\simeq H^2_{\mathrm{crys}}(A/W)\),
\(\Phi:=\bigwedge^2F\), and let \(\Theta:=p^{-1}\Phi\) be the normalized
Frobenius. The isocrystal \(N[1/p]\) has slope \(1\). For any
\(\Phi\)-stable lattice \(\Lambda\) in a slope-one isocrystal, let
\(L(\Lambda):=\bigcap_{r\ge0}\Phi^{-r}(p^r\Lambda)=\bigcap_{r\ge0}\Theta^{-r}\Lambda\)
be its Hodge--Witt filtration lattice (Proposition~\ref{prop:HW-filtration}),
and call \(e(\Lambda):=p\text{-}\!\exp(\Lambda/L(\Lambda))\) its \emph{lattice
exponent}. Write
\[
  a:=a(A)=\dim_kM/(FM+VM)
\]
for the \(a\)-number of \(A\), and let \(\bar F,\bar V\) be the maps induced by
\(F,V\) on \(M/pM\simeq H^1_{\mathrm{dR}}(A/k)\). Since
\(\ker\bar F=\operatorname{im}\bar V\) and \(\ker\bar V=\operatorname{im}\bar F\),
both of dimension \(g\), we also have
\[
  \dim_k\bigl(\operatorname{im}\bar F\cap\operatorname{im}\bar V\bigr)
  =2g-\dim_k\bigl(\operatorname{im}\bar F+\operatorname{im}\bar V\bigr)=a,
\]
the form in which the \(a\)-number is used below. If \(\mu_{\min,\Lambda}(r)\) is the smallest elementary-divisor valuation of
the \(W\)-linearization of \(\Phi^r\), then Smith normal form gives
\begin{equation}\label{eq:lattice-exponent}
  e(\Lambda)=\min\{\,m:\Phi^r(\Lambda)\subseteq p^{r-m}\Lambda\ \text{for all }r\ge0\,\}
        =\max_{r\ge0}\{\,r-\mu_{\min,\Lambda}(r)\,\}.
\end{equation}
Put \(L:=L(N)\). The degree-two reconstruction
(Theorem~\ref{thm:geometric-ekedahl-pair-partial-v}) identifies
\(\ker d\subset U_A^{0,2}\) with \(N/L\). Together with
Proposition~\ref{prop:ekedahl-pair-weighs-degree} and the next lemma, this gives
\[
  p\text{-}\!\exp(U_A^{0,2})=e(N),
  \qquad
  \sigmaArt(A)=\deg(U_A^{0,2})=\lgth_W(N/L).
\]
We first bound the exponent. We then bound the degree-two Artin invariant from
below and, assuming a principal polarization, from above.

\begin{lemma}
\label{lem:first-kernel-controls-exponent}
Let \(U=[U^0\xrightarrow dU^1]\in\operatorname{Dom}^{+}\). Then
\(p\text{-}\!\exp(U)=p\text{-}\!\exp(\ker d)\).
\end{lemma}

\begin{proof}
Lemma~\ref{lem:kernel-ladder} gives \(U^0=\ker d+VU^0\), while coherence makes
\(U^0\) \(V\)-adically separated
(Proposition~\ref{prop:coherence-criterion}). If \(p^mU^0=0\), then
\(p^m\ker d=0\). Conversely, suppose \(p^m\ker d=0\). Writing \(x=a+Vy\), with
\(a\in\ker d\), gives \(p^mx=V(p^my)\), and hence
\(p^mU^0\subseteq V(p^mU^0)\). Iterating this inclusion gives
\(p^mU^0\subseteq V^r(p^mU^0)\subseteq V^rU^0\) for every \(r\ge0\), hence
\(p^mU^0=0\).
\end{proof}

The next lemma records two consequences of the definition. First,
\(L(\Lambda)\) is monotone in \(\Lambda\). Second, for \(m\ge0\), the
conditions \(\Phi^r(\Lambda)\subseteq p^{r-m}\Lambda\) for all \(r\ge0\) are
equivalent to the single inclusion
\(\Phi^m(\Lambda)\subseteq L(\Lambda)\). Hence \(e(\Lambda)\) is the least such
\(m\). Since \(\Phi(L(\Lambda))\subseteq pL(\Lambda)\), that inclusion also
implies \(\Phi^{m+s}(\Lambda)\subseteq p^sL(\Lambda)\) for every \(s\ge0\).

\begin{lemma}[Lattice depth]
\label{lem:exponent-depth}
Let \(\Lambda'\subseteq\Lambda\) be \(\Phi\)-stable lattices in the same
slope-one isocrystal, and let \(m\ge0\).
\begin{enumerate}
\item \(L(\Lambda')\subseteq L(\Lambda)\).
\item \(e(\Lambda)\le m\) if and only if \(\Phi^m(\Lambda)\subseteq L(\Lambda)\).
\end{enumerate}
\end{lemma}

\begin{proof}
The definition gives \(\Phi(L(\Lambda))\subseteq pL(\Lambda)\).

For (1), if \(x\in L(\Lambda')\), then \(x\in\Lambda\) and
\(\Phi^r(x)\in p^r\Lambda'\subseteq p^r\Lambda\) for all \(r\ge0\), so
\(x\in L(\Lambda)\).

For (2), if \(\Phi^m(\Lambda)\subseteq L(\Lambda)\), then
\(\Phi^r(\Lambda)\subseteq\Phi^{r-m}(L(\Lambda))\subseteq p^{r-m}\Lambda\) for
\(r\ge m\), while for \(r<m\) the containment
\(\Phi^r(\Lambda)\subseteq p^{r-m}\Lambda\) follows from
\(\Phi^r(\Lambda)\subseteq\Lambda\subseteq p^{r-m}\Lambda\). Hence
\(e(\Lambda)\le m\) by \eqref{eq:lattice-exponent}. Conversely, if
\(e(\Lambda)\le m\), then
\(\Phi^s(\Phi^m(\Lambda))=\Phi^{s+m}(\Lambda)\subseteq p^s\Lambda\) for all
\(s\ge0\), which says \(\Phi^m(\Lambda)\subseteq L(\Lambda)\).
\end{proof}

\subsubsection*{\normalfont\bfseries A sharp lower bound for the \texorpdfstring{\(p\)}{p}-exponent}

\begin{lemma}[Lower bound for the \(p\)-exponent]
\label{lem:rank-estimate-lower-bound}
\(p\text{-}\!\exp(U_A^{0,2})
\ge 1+\left\lfloor\frac{g-2}{a}\right\rfloor
=\left\lceil\frac{g-1}{a}\right\rceil\).
\end{lemma}

\begin{proof}
Let \(\mu_1(r)\le\mu_2(r)\le\cdots\le\mu_{2g}(r)\) be the elementary-divisor
valuations of the \(W\)-linearization of \(F^r\) on \(M\). Those of
\(\bigwedge^2F^r\) are the pairwise sums, so
\(\mu_{\min,N}(r)=\mu_1(r)+\mu_2(r)\). Therefore
\(p\text{-}\!\exp(U_A^{0,2})=e(N)
=\max_{r\ge0}\{\,r-\mu_{\min,N}(r)\,\}\). If \(\bar F^r\) has rank at least two,
then the reduction modulo \(p\) of the \(W\)-linearization of
\(\Phi^r=\bigwedge^2F^r\), namely \(\bigwedge^2\bar F^r\), is nonzero. Thus its
first elementary divisor is a unit, so \(\mu_{\min,N}(r)=0\) and hence
\(e(N)\ge r-\mu_{\min,N}(r)=r\). It therefore remains to control how quickly the rank of
\(\bar F^r\) can decrease. This is where the \(a\)-number enters.

Recall from the start of this subsection that
\(\dim_k\operatorname{im}\bar F=g\) and
\(\dim_k(\operatorname{im}\bar F\cap\ker\bar F)=a\). For \(r\ge2\), the
kernel of \(\bar F:\operatorname{im}\bar F^{r-1}\to\operatorname{im}\bar F^r\) is
\(\operatorname{im}\bar F^{r-1}\cap\ker\bar F\), which is contained in this
\(a\)-dimensional space. Hence
\[
  \operatorname{rank}(\bar F^{r-1})-\operatorname{rank}(\bar F^r)
  =\dim_k(\operatorname{im}\bar F^{r-1}\cap\ker\bar F)\le a.
\]
Starting from \(\operatorname{rank}(\bar F)=g\), induction gives
\(\operatorname{rank}(\bar F^r)\ge g-(r-1)a\).
The largest \(r\) for which the right-hand side is at least two is
\(r_0=1+\lfloor(g-2)/a\rfloor\). To compare the two forms of the bound, write
\(g-2=qa+b\) with \(0\le b<a\). Then
\(1+\lfloor(g-2)/a\rfloor=q+1\), while
\((g-1)/a=q+(b+1)/a\) with \(0<(b+1)/a\le1\). Hence
\(r_0=\lceil(g-1)/a\rceil\). Therefore
\(\operatorname{rank}(\bar F^{r_0})\ge2\), so \(e(N)\ge r_0\), as required.
\end{proof}

\begin{remark}[Sharpness of the lower bound]
For sharpness, distribute the total dimension as evenly as possible among
\(a\) cyclic supergeneral factors. Namely, write \(g=aq+b\), where
\(0\le b<a\), choose \(g_1,\ldots,g_a\in\{q,q+1\}\) with sum \(g\), and set
\(A=\prod_iA_{g_i}\). Here \(A_m\) is a principally polarized member of the
family in Theorem~\ref{thm:supergeneral}. Since each factor has \(a\)-number
one and \(M/(FM+VM)\) is additive, \(a(A)=a\).

Put \(M_m=H^1_{\mathrm{crys}}(A_m/W)\). The K\"unneth decomposition consists
of the diagonal terms \(\bigwedge^2M_{g_i}\) and the cross terms
\(M_{g_i}\otimes_WM_{g_j}\). Since \(L(\Lambda)\) respects finite direct sums,
the exponent of \(A\) is the maximum of their exponents. The diagonal term of
a dimension-\(m\) factor has exponent \(m-1\) by
Theorem~\ref{thm:supergeneral}.

It remains to compute the cross terms. If \(\mu_{\min,m}(r)\) is the smallest
elementary-divisor valuation of \(F^r\) on \(M_m\), then the Frobenius
valuation word
\([\underbrace{0,\ldots,0}_{m},\underbrace{1,\ldots,1}_{m}]\) gives
\(\mu_{\min,m}(r)=cm+\max\{0,t-m\}\) for \(r=2mc+t\), \(0\le t<2m\).
Tensor-product valuations are pairwise sums, so \eqref{eq:lattice-exponent}
gives
\[
  e(M_m\otimes_WM_m)=m,\qquad e(M_{m+1}\otimes_WM_m)=m.
\]
For the second equality, \(0\le r/2-\mu_{\min,j}(r)\le j/2\) for every \(j\), so the
relevant integer is at most \(m+\tfrac12\), hence at most \(m\), with equality
at \(r=m\). Thus a diagonal term of size \(m\) contributes \(m-1\), whereas a
cross term between equal or consecutive sizes contributes the smaller size.

If \(a=1\), the exponent is \(g-1\). If \(a>1\), taking the maximum gives
\(q\) for \(b=0,1\) and \(q+1\) for \(b\ge2\). These values equal
\(\lceil(g-1)/a\rceil\). Thus the lower bound is sharp.
\end{remark}

\subsubsection*{\normalfont\bfseries An upper bound for the \texorpdfstring{\(p\)}{p}-exponent}

Let \(M^{\mathrm{sp}}\subset M[1/p]\) be the smallest superspecial
Dieudonn\'e lattice containing \(M\), whose existence follows from
\cite[Lemma~1.3]{Li89}. Following \cite[Theorem~3.8, proof]{GSY25}, we call
this lattice the \emph{superspecial envelope} of \(M\), and put
\(N^{\mathrm{sp}}:=\bigwedge_W^2M^{\mathrm{sp}}\). Define
\[
  \mathrm{Fil}_{\mathrm{sp}}^iM:=M\cap F^iM^{\mathrm{sp}},\qquad
  \mathrm{Fil}_{\mathrm{sp}}^iN
    :=\bigwedge\nolimits_W^2\mathrm{Fil}_{\mathrm{sp}}^iM,\qquad
  s_i:=\dim_k\bigl(\mathrm{Fil}_{\mathrm{sp}}^iM/
                          \mathrm{Fil}_{\mathrm{sp}}^{i+1}M\bigr),
\]
and define the \emph{superspecial depth} of \(M\) by
\(\delta:=\min\{r\ge0:F^rM^{\mathrm{sp}}\subseteq M\}\). Thus
\(\mathrm{Fil}_{\mathrm{sp}}^0M=M\),
\(\mathrm{Fil}_{\mathrm{sp}}^0N=N\), and \(\delta\) is the nilpotency
index of \(F\), equivalently of \(V\),
on \(M^{\mathrm{sp}}/M\). In particular, \(\delta=0\) exactly when
\(A\) is superspecial. The following consequences of Li's filtration control
\(\delta\).

\begin{lemma}[Li's superspecial-envelope filtration]
\label{lem:li-superspecial-envelope}
We have \(1\le s_0<s_1<\cdots<s_{\delta}=g\) and
\(a\le g-\delta\), and each quotient
\(\mathrm{Fil}_{\mathrm{sp}}^iM/\mathrm{Fil}_{\mathrm{sp}}^{i+1}M\) is
annihilated by \(F\) and \(V\).
\end{lemma}

\begin{proof}
The chain and the bound are proved in \cite[Theorem~3.8, proof]{GSY25}, using
Li's strict-growth lemma \cite[Lemma~1.6]{Li89} and length formula
\cite[p.~337(vi)]{Li89}. The inclusions
\(F(\mathrm{Fil}_{\mathrm{sp}}^iM),
V(\mathrm{Fil}_{\mathrm{sp}}^iM)
\subseteq\mathrm{Fil}_{\mathrm{sp}}^{i+1}M\) give the last assertion.
\end{proof}

The general upper bound below follows by moving from
\(\mathrm{Fil}_{\mathrm{sp}}^0N\) through \(\delta\) Frobenius
iterates to \(\mathrm{Fil}_{\mathrm{sp}}^{\delta}N\), and then using
one further iterate there. To improve this to \(e(N)\le \delta\), it
suffices by Lemma~\ref{lem:exponent-depth}(2) to prove
\(\Phi^{\delta}N\subseteq L(N)\). When \(\delta\ge g-2\),
the strict growth of the numbers \(s_i\) forces one of two
endpoint configurations. If the image of \(M\) in
\(M^{\mathrm{sp}}/FM^{\mathrm{sp}}\) is one-dimensional, the inclusion
\(N\subseteq L(N^{\mathrm{sp}})\), together with
\(\Phi^{\delta}N^{\mathrm{sp}}
=\mathrm{Fil}_{\mathrm{sp}}^{\delta}N\subseteq N\), gives the desired
containment by iteration. If instead
\(\dim_k(\mathrm{Fil}_{\mathrm{sp}}^{\delta-1}M/
\mathrm{Fil}_{\mathrm{sp}}^{\delta}M)=g-1\), it is enough that
\(\mathrm{Fil}_{\mathrm{sp}}^{\delta}N\subseteq
L(\mathrm{Fil}_{\mathrm{sp}}^{\delta-1}N)\), since
\(\Phi^{\delta}N
\subseteq\mathrm{Fil}_{\mathrm{sp}}^{\delta}N\) and
\(L(\mathrm{Fil}_{\mathrm{sp}}^{\delta-1}N)\subseteq L(N)\). The
next lemma proves the two endpoint inclusions.

\begin{lemma}[Exterior-square comparison for superspecial lattices]
\label{lem:superspecial-exterior-square-comparison}
Let \(S\) be a superspecial Dieudonn\'e lattice of rank \(2g\). For every
Dieudonn\'e lattice \(\Lambda\subset S[1/p]\), write
\(N_\Lambda:=\bigwedge_W^2\Lambda\), with \(\Phi=\bigwedge^2F\). Then
\(L(N_S)=FS\wedge S\). In particular, \(e(N_S)=1\) for \(g\ge2\). The two
endpoint criteria are:
\begin{enumerate}
\item if \(\Lambda\subseteq S\) and
\(\dim_k\Lambda/(\Lambda\cap FS)\le1\), then
\(N_\Lambda\subseteq L(N_S)\).
\item if \(S\subseteq\Lambda\), \(V(\Lambda/S)=0\), and
\(\dim_k(\Lambda/S)\ge g-1\), then \(N_S\subseteq L(N_\Lambda)\).
\end{enumerate}
\end{lemma}

\begin{proof}
We use two standard exterior-algebra identities:
\begin{enumerate}
\item[\((\mathrm{E}1)\)] if \(U\subseteq E\), then the quotient map induces
\(\bigwedge^2E/(U\wedge E)\simeq\bigwedge^2(E/U)\).
\item[\((\mathrm{E}2)\)] for a linear map \(f:E\to E'\) of vector spaces,
\(\ker(\bigwedge^2f)=(\ker f)\wedge E\), and the same formula holds for a
semilinear map after twisting its source.
\end{enumerate}

Superspeciality gives \(FS=VS\) and \(F^2S=pS\), hence
\(\Phi^2N_S=p^2N_S\). For \(x\in N_S\), it follows that
\(x\in L(N_S)\) if and only if
\(\Phi x\in pN_S\). Modulo \(p\), the relation
\(\ker\bar F=FS/pS\) gives
\(\ker(\bigwedge^2\bar F)=(FS/pS)\wedge(S/pS)\) by
\((\mathrm{E}2)\). Since \(pS\subseteq FS\), \((\mathrm{E}1)\) identifies its
inverse image in \(N_S\) with \(FS\wedge S\), proving
\(L(N_S)=FS\wedge S\). The same fact gives
\(N_S/L(N_S)\simeq\bigwedge_k^2(S/FS)\). Since
\(\dim_k(S/FS)=g\), this quotient is killed by \(p\) and is nonzero for
\(g\ge2\). Thus \(e(N_S)=1\).

For (1), the image of \(N_\Lambda\) in
\(N_S/L(N_S)\simeq\bigwedge_k^2(S/FS)\) lies in the exterior square of the
image of \(\Lambda\) in \(S/FS\). This image has dimension at most one, so its
exterior square vanishes.

For (2), the periodicity \(F^2S=pS\) reduces the defining conditions for
\(N_S\subseteq L(N_\Lambda)\) to \(N_{FS}\subseteq pN_\Lambda\). Now
\(V(\Lambda/S)=0\) gives \(p\Lambda=F(V\Lambda)\subseteq FS\), and
multiplication by \(p\) identifies \(\Lambda/S\) with \(p\Lambda/pS\). Hence
\[
  \dim_k(FS/p\Lambda)
  =\dim_k(FS/pS)-\dim_k(\Lambda/S)\le g-(g-1)=1 .
\]
Thus \(\bigwedge_k^2(FS/p\Lambda)=0\). Fact
\((\mathrm{E}1)\) then gives
\[
  N_{FS}=(p\Lambda)\wedge FS
  \subseteq(p\Lambda)\wedge\Lambda=pN_\Lambda
\]
as required.
\end{proof}

\begin{theorem}
\label{thm:extremal-exponent}
Assume \(g\ge3\), and put \(E:=p\text{-}\!\exp(U_A^{0,2})\). Then
\[
   E\le \delta+1\le g-a+1,
\]
and if \(\delta\ge g-2\), then \(E\le \delta\).
Consequently, \(a=2\) implies \(E\le g-2\), while \(a=1\) if and only if
\(E=g-1\).
\end{theorem}

\begin{proof}
We first move from \(\mathrm{Fil}_{\mathrm{sp}}^0N=N\) to the superspecial
endpoint \(\mathrm{Fil}_{\mathrm{sp}}^{\delta}N\). This costs
\(\delta\) Frobenius iterates, followed in general by one more at the
endpoint.
Since \(F(\mathrm{Fil}_{\mathrm{sp}}^iM)
\subseteq\mathrm{Fil}_{\mathrm{sp}}^{i+1}M\), we have
\(\Phi(\mathrm{Fil}_{\mathrm{sp}}^iN)
\subseteq\mathrm{Fil}_{\mathrm{sp}}^{i+1}N\), hence
\(\Phi^\delta N
\subseteq\mathrm{Fil}_{\mathrm{sp}}^{\delta}N\). The lattice
\(\mathrm{Fil}_{\mathrm{sp}}^{\delta}M
=F^{\delta}M^{\mathrm{sp}}\) is superspecial because
\(F^{\delta}M^{\mathrm{sp}}
=V^{\delta}M^{\mathrm{sp}}\), so
Lemma~\ref{lem:superspecial-exterior-square-comparison}
and Lemma~\ref{lem:exponent-depth}(2) give
\(\Phi(\mathrm{Fil}_{\mathrm{sp}}^{\delta}N)
\subseteq L(\mathrm{Fil}_{\mathrm{sp}}^{\delta}N)\). Therefore
\[
  \Phi^{\delta+1}N
  \subseteq \Phi(\mathrm{Fil}_{\mathrm{sp}}^{\delta}N)
  \subseteq L(\mathrm{Fil}_{\mathrm{sp}}^{\delta}N)
  \subseteq L(N)
\]
by monotonicity (Lemma~\ref{lem:exponent-depth}(1)), and
Lemma~\ref{lem:exponent-depth}(2) gives \(e(N)\le \delta+1\).
Finally, Lemma~\ref{lem:li-superspecial-envelope} gives
\(\delta+1\le g-a+1\).

Assume now that \(\delta\ge g-2\). We show that the preceding estimate improves
to \(e(N)\le\delta\). The sequence
\(1\le s_0<\cdots<s_{\delta}=g\) then has at least \(g-1\) terms. Hence either
\(s_0=1\), or it is \(2,3,\ldots,g\), in which case
\(s_{\delta-1}=g-1\). These are exactly the two configurations in
Lemma~\ref{lem:superspecial-exterior-square-comparison}(1) and (2).

In the first case, since \(\dim_kM/(M\cap FM^{\mathrm{sp}})=1\),
Lemma~\ref{lem:superspecial-exterior-square-comparison}(1), applied with
\(S=M^{\mathrm{sp}}\) and \(\Lambda=M\), gives
\(N\subseteq L(N^{\mathrm{sp}})\), while
\(\Phi^{\delta}N^{\mathrm{sp}}
=\mathrm{Fil}_{\mathrm{sp}}^{\delta}N\subseteq N\). Thus, for
\(x\in N\) and \(s\ge0\),
\(\Phi^{\delta+s}x
\in p^s\Phi^{\delta}N^{\mathrm{sp}}\subseteq p^sN\). Hence
\(\Phi^{\delta}N\subseteq L(N)\), and
Lemma~\ref{lem:exponent-depth}(2) gives \(e(N)\le \delta\).

In the second case, apply
Lemma~\ref{lem:superspecial-exterior-square-comparison}(2) with
\(S=\mathrm{Fil}_{\mathrm{sp}}^{\delta}M\) and
\(\Lambda=\mathrm{Fil}_{\mathrm{sp}}^{\delta-1}M\). The lattice
\(S\) is superspecial, \(V(\Lambda/S)=0\) by
Lemma~\ref{lem:li-superspecial-envelope}, and
\(\dim_k(\Lambda/S)=g-1\). Thus
\(\mathrm{Fil}_{\mathrm{sp}}^{\delta}N
\subseteq L(\mathrm{Fil}_{\mathrm{sp}}^{\delta-1}N)\). Since
\(\Phi^\delta N
\subseteq\mathrm{Fil}_{\mathrm{sp}}^{\delta}N
\subseteq L(\mathrm{Fil}_{\mathrm{sp}}^{\delta-1}N)\subseteq L(N)\),
Lemma~\ref{lem:exponent-depth}(2) gives \(e(N)\le \delta\).

This proves the two estimates involving \(\delta\). We finish with
their consequences for the \(a\)-number.

If \(a=2\), then \(\delta\le g-2\) by
Lemma~\ref{lem:li-superspecial-envelope}. The general estimate
\(E\le \delta+1\) gives \(E\le g-2\) when \(\delta<g-2\),
while the sharpened estimate \(E\le \delta\) does so when
\(\delta=g-2\).

If \(a=1\), then \(\delta\le g-1\), while
Lemma~\ref{lem:rank-estimate-lower-bound} gives \(E\ge g-1\). If
\(\delta\le g-3\), then \(E\le \delta+1\le g-2\), a
contradiction, so \(\delta\ge g-2\) and
\(E\le \delta\le g-1\). Hence \(E=g-1\).

Conversely, if \(E=g-1\), the estimate \(E\le g-a+1\) forces \(a\le2\),
and \(a=2\) is excluded by the preceding case. Hence \(a=1\).
\end{proof}

\begin{remark}[Sharpness of the upper bound at \(a=2\)]\label{rem:exponent-sharpness}
The upper bound \(g-2\) for \(a=2\) is attained for every \(g\ge3\). Let
\(A=A_{g-1}\times E\) be the product of a cyclic supergeneral \((g-1)\)-fold
\(A_{g-1}\) (Theorem~\ref{thm:supergeneral}) with a supersingular elliptic
curve. Since the Dieudonn\'e module of a product is the direct sum of those
of its factors, \(a(A)=a(A_{g-1})+a(E)=2\).
Moreover, \(U_{A_{g-1}}^{0,2}\) is a K\"unneth direct summand of \(U_A^{0,2}\), so
\(p\text{-}\!\exp(U_A^{0,2})\ge p\text{-}\!\exp(U_{A_{g-1}}^{0,2})=g-2\) by
Theorem~\ref{thm:supergeneral}. Equality holds by
Theorem~\ref{thm:extremal-exponent}.
\end{remark}

\subsubsection*{\normalfont\bfseries A lower bound for \texorpdfstring{\(\sigmaArt\)}{sigma_Art}}

For a positive domino \(U\), Proposition~\ref{prop:ekedahl-pair-weighs-degree}
gives \(\deg(U)-\dim(U)=\lgth_W\ker(dV)\). Applying this identity to
\(U_A^{0,2}\), whose dimension is \(\binom{g}{2}\) by
Corollary~\ref{cor:domino-supersingular}, gives the basic lower bound
\(\sigmaArt(A)\ge\binom{g}{2}\). The next theorem refines it by controlling
\(\ker(dV)\) in terms of the \(a\)-number.

\begin{theorem}
\label{thm:sigma-a-number-lower-bound}
Let \(A/k\) be a supersingular abelian \(g\)-fold, \(g\ge2\), and put
\(a=a(A)\). Then
\[
  \sigmaArt(A)\ge g(g-1)-\binom{a}{2}.
\]
In particular, if \(A\) is not superspecial, then
\(\sigmaArt(A)\ge\binom{g}{2}+g-1\).
\end{theorem}

\begin{proof}
Proposition~\ref{prop:ekedahl-pair-weighs-degree} and
Corollary~\ref{cor:domino-supersingular} give
\(\sigmaArt(A)=\binom{g}{2}+\lgth_W\ker(dV)\).
The \(\mathrm{BT}_1\) \(A[p]\) has contravariant Dieudonn\'e module
\(H^1_{\mathrm{dR}}(A/k)=M/pM\), with Frobenius and Verschiebung
\(\bar F\) and \(\bar V\). Since
\(\lgth_W\ker(dV)\ge
\dim_k\bigl(\ker(dV)/p\ker(dV)\bigr)\), it suffices to bound the latter
dimension. We first construct a surjection
\(\ker(dV)/p\ker(dV)\twoheadrightarrow
\bigwedge^2\bar F\bigl(\ker(\bigwedge^2\bar V)\bigr)\), and then compute
the dimension of its target from the \(\mathrm{BT}_1\).
Put \(V_K:=p\Phi^{-1}\), suppressing Frobenius
twists, and let
\[
  \Lambda:=V_K^{-1}N
  =\{x\in N:(\bigwedge\nolimits^2V)(x)\in pN\}.
\]
The construction of \(K_{-1}\) in
Theorem~\ref{thm:geometric-ekedahl-pair-partial-v} gives
\(\ker(dV)\simeq\Lambda/(L\cap\Lambda)\). Since
\(V_K=p^{-1}\bigwedge^2V\), one has \(pN\subseteq\Lambda\), and reduction
modulo \(p\) gives
\(\Lambda/pN=\ker(\bigwedge^2\bar V)\).
The map obtained by reducing \(\Lambda\) modulo \(pN\) and then applying
\(\bigwedge^2\bar F\) kills \(p\Lambda+(L\cap\Lambda)\): the first summand
dies under reduction, while the second does so because
\(\Phi(L)\subseteq pL\). Since
\(\ker(dV)/p\ker(dV)\simeq
\Lambda/(p\Lambda+(L\cap\Lambda))\), the universal property of the quotient
gives the factorization below:
\[
\begin{tikzcd}[column sep=large]
  \Lambda
    \arrow[r, two heads, "{\bmod\,pN}"]
    \arrow[d, two heads]
  & \ker(\bigwedge^2\bar V)
    \arrow[d, two heads, "{\bigwedge^2\bar F}"] \\
  \ker(dV)/p\ker(dV)
    \arrow[r, dashed, two heads]
  & \bigwedge^2\bar F\bigl(\ker(\bigwedge^2\bar V)\bigr).
\end{tikzcd}
\]

We use two standard exterior-algebra facts. For a linear map \(f:E\to E'\),
one has \(\ker(\bigwedge^2f)=(\ker f)\wedge E\). The same formula holds for
a semilinear map after twisting. If \(E''\subseteq E\), then
\(E''\wedge E\) is the kernel of \(\bigwedge^2E\to\bigwedge^2(E/E'')\), and
hence has dimension
\(\binom{\dim E}{2}-\binom{\dim(E/E'')}{2}\).

Applying the first fact to \(\bar V\) and using the \(\mathrm{BT}_1\) relation
\(\ker\bar V=\operatorname{im}\bar F\) gives
\[
  \bigwedge\nolimits^2\bar F\bigl(\ker(\bigwedge\nolimits^2\bar V)\bigr)
  =\operatorname{im}(\bar F^2)\wedge\operatorname{im}\bar F.
\]
Moreover \(\operatorname{im}\bar F/\operatorname{im}(\bar F^2)\) has dimension
\(a\), since \(\dim_k\operatorname{im}\bar F=g\) and
\(\ker\bar F\cap\operatorname{im}\bar F\) has dimension \(a\).
The second fact therefore shows that the last displayed space has dimension
\(\binom{g}{2}-\binom{a}{2}\).
It follows that
\(\lgth_W\ker(dV)\ge\binom{g}{2}-\binom{a}{2}\), and hence
\(\sigmaArt(A)\ge g(g-1)-\binom{a}{2}\). If \(A\) is not superspecial,
then \(a\le g-1\), which gives the final assertion.
\end{proof}

\begin{corollary}
\label{thm:minimal-degree-superspecial}
Let \(A/k\) be a supersingular abelian \(g\)-fold, \(g\ge2\). Then \(A\) is
superspecial if and only if \(\sigmaArt(A)=\binom{g}{2}\). In this case
\[
  U_A^{0,2}\cong U_1^{\oplus\binom{g}{2}}.
\]
\end{corollary}

\begin{proof}
The forward implication and the displayed decomposition follow from
Theorem~\ref{thm:superspecial}. Conversely, if \(A\) is not superspecial, then
Theorem~\ref{thm:sigma-a-number-lower-bound} gives
\(\sigmaArt(A)\ge\binom{g}{2}+g-1\).
\end{proof}

\begin{remark}[A numerical gap in \texorpdfstring{\(\sigmaArt\)}{sigma_Art}]
\label{rem:sigma-gap-above-superspecial}
No value of \(\sigmaArt(A)\) lies strictly between \(\binom{g}{2}\) and
\(\binom{g}{2}+g-1\). Thus \(\sigmaArt(A)\ne4\) for \(g=3\), while
\(\sigmaArt(A)\notin\{7,8\}\) for \(g=4\).
\end{remark}

\subsubsection*{\normalfont\bfseries An upper bound for \texorpdfstring{\(\sigmaArt\)}{sigma_Art}}

We now assume that \(A\) is principally polarized. The polarization makes
\(\sigmaArt(A)=\lgth_W(N/L)\) half the length of a discriminant group and
allows its exponent to be controlled by
\(e:=p\text{-}\!\exp(U_A^{0,2})=e(N)\). Recall that
\(L=\bigcap_{r\ge0}\Theta^{-r}N\). In what follows, dual lattices are taken
with respect to the pairing constructed in the next lemma. Thus
\(\Lambda^\vee:=\{x\in N[1/p]:\langle x,\Lambda\rangle\subseteq W\}\).

\begin{lemma}\label{lem:polarization-isometry}
The principal polarization induces a perfect symmetric \(W\)-bilinear form
\(\langle\ ,\,\rangle\) on \(N\) for which \(\Theta\) is a
\(\sigma\)-semilinear isometry: \(\langle\Theta x,\Theta y\rangle
=\sigma\langle x,y\rangle\), where \(\sigma\) is the Witt-vector Frobenius.
Consequently, for
every \(W\)-lattice \(\Lambda\subset N[1/p]\) and \(r\in\mathbb Z\),
\[
  (\Theta^{r}\Lambda)^\vee=\Theta^{r}(\Lambda^\vee) .
\]
\end{lemma}

\begin{proof}
The polarization induces a perfect alternating form \(\psi:M\times M\to W\)
with \(\psi(Fx,Fy)=p\,\sigma\psi(x,y)\). See, e.g.,
\cite[Section~1.10]{Oort00}. Set
\(\langle x_1\wedge x_2,\,y_1\wedge y_2\rangle:=\det(\psi(x_i,y_j))_{i,j}\).
The determinant pairing is symmetric and perfect: the induced map
\(N\to N^\vee\) is, under the canonical identification, the exterior square
of the isomorphism \(M\simeq M^\vee\) defined by \(\psi\). Moreover, each entry of the determinant
acquires the factor \(p\) under \(F\), so
\(\langle\Phi x,\Phi y\rangle=p^2\sigma\langle x,y\rangle\). Dividing by
\(p^2\) gives \(\langle\Theta x,\Theta y\rangle=\sigma\langle x,y\rangle\).
Finally,
\(\langle x,\Theta^{r}\Lambda\rangle\subseteq W\) if and only if
\(\langle\Theta^{-r}x,\Lambda\rangle=\sigma^{-r}\langle x,\Theta^{r}\Lambda\rangle
\subseteq W\), which proves the duality formula.
\end{proof}

\begin{theorem}[Upper bound for \(\sigmaArt\)]
\label{thm:sigma-discriminant-bound}
There is a self-dual chain
\(p^{e}N\subseteq L\subseteq N\subseteq L^\vee\subseteq p^{-e}N\). Moreover,
\(L^\vee/L\) is killed by \(p^e\) and has length \(2\,\sigmaArt(A)\).
Consequently,
\[
  \sigmaArt(A)\;\le\;\left\lfloor\frac{e\,g(2g-1)}{2}\right\rfloor
  \;\le\;\frac{(g-1)\,g(2g-1)}{2},
\]
where the second inequality uses \(e\le g-1\) for \(g\ge3\)
(Theorem~\ref{thm:extremal-exponent}). Using instead the \(a\)-dependent bound
for \(e\) from the same theorem gives
\(\sigmaArt(A)\le\lfloor(g-a+1)\,g(2g-1)/2\rfloor\), with \(g-a+1\) improved to
\(g-2\) when \(a=2\).
\end{theorem}

\begin{proof}
The inclusions \(p^{e}N\subseteq L\subseteq N\) follow from the definitions.
The pairing of Lemma~\ref{lem:polarization-isometry} is perfect, so \(N^\vee=N\).
Dualizing gives \(N\subseteq L^\vee\subseteq p^{-e}N\). It also gives
\(L^\vee/N\simeq\operatorname{Hom}_W(N/L,W[1/p]/W)\), and hence
\(\lgth_W(L^\vee/L)=2\lgth_W(N/L)=2\,\sigmaArt(A)\).

The inclusion \(p^eN\subseteq L=\bigcap_{t\ge0}\Theta^{-t}N\) is equivalent to
\(p^e\Theta^tN\subseteq N\) for every \(t\ge0\). Dualizing these inclusions
gives the same statement for \(t<0\). Hence
\(p^e\Theta^tN\subseteq N\) for every \(t\in\mathbb Z\).
The partial intersections defining \(L\) stabilize between \(p^eN\) and \(N\),
so Lemma~\ref{lem:polarization-isometry} gives
\[
  L^\vee=\Bigl(\bigcap_{r\ge0}\Theta^{-r}N\Bigr)^\vee
        =\sum_{r\ge0}\Theta^{-r}N.
\]
For \(r,s\ge0\), the uniform bound with \(t=s-r\) gives
\(p^e\Theta^{-r}N\subseteq\Theta^{-s}N\). Thus every summand of \(p^eL^\vee\)
lies in every \(\Theta^{-s}N\), proving \(p^eL^\vee\subseteq L\).

Finally, \(L^\vee/L\) is generated by \(\operatorname{rank}_W(L^\vee)=g(2g-1)\)
elements and killed by \(p^{e}\), so
\(2\,\sigmaArt(A)=\lgth_W(L^\vee/L)\le e\,g(2g-1)\).
\end{proof}

\begin{remark}[A principal polarization obstruction]
Consider the rank-\(6\) supersingular cyclic \(F\)-crystal with Frobenius
valuation word \((0,0,1,0,1,1)\). For this cyclic representative, the three
exterior-square blocks \(B^{(1)},B^{(2)},B^{(3)}\) have valuation sequences
\[
  \mathbf v^{(1)}=(0,1,1,1,2,1),\qquad
  \mathbf v^{(2)}=(1,0,2,1,1,1),\qquad
  \mathbf v^{(3)}=(0,1,2).
\]
The recurrence \(h_{i+1}-h_i=v_i-1\), followed by the normalization
\(\min_i h_i=0\), gives the corresponding height functions
\[
  (1,0,0,0,0,1),\qquad (1,1,0,1,1,1),\qquad (1,0,0).
\]
The blockwise construction of Theorem~\ref{thm:cyclic-matrix} therefore gives
an associated positive domino \(U\) with \(J(U)=(1,2,5)\), so \(\deg(U)=8\) and
\(p\text{-}\!\exp(U)=1\). If this \(F\)-crystal admitted a principal
quasi-polarization, the determinant pairing of
Lemma~\ref{lem:polarization-isometry} and the proof of
Theorem~\ref{thm:sigma-discriminant-bound} would give
\(\deg(U)\le\lfloor3(2\cdot3-1)/2\rfloor=7\), a contradiction. Thus this
\(F\)-crystal is \emph{not} principally quasi-polarizable.
\end{remark}

\appendix
\addtocontents{toc}{\protect\setcounter{tocdepth}{1}}
\section{Dominoes and unipotent invariants}
\label{app:unipotent}

Throughout this appendix \(k\) is algebraically closed. Cartier theory
gives the formal realization of a domino, while its degree-\(1\) term gives a
weight-one syntomic realization represented by a perfect unipotent group. We
prove that these groups have the same isogeny partition, compute them in the
principal examples of Section~3, and relate them to \(p\)-primary Brauer
groups.

\subsection{Formal and syntomic realizations}
\label{sec:formal-syntomic-realizations}

Cartier theory attaches to a domino \(U\) a smooth unipotent formal group
\(\widehat G(U)\) whose covariant Cartier module is \(U^0\). In the geometric
case \(U=U_X^{0,2}\), this is the domino contribution to the Artin--Mazur
formal Brauer group: when \(\widehat{\mathrm{Br}}_X=\Phi_X^2\) is
representable, its Cartier module is \(H^2(X,W\mathcal O_X)\)
\cite[Corollary~I.4.3]{AM77}.

The degree-one term has a parallel realization in weight-one syntomic
cohomology. Let \(\mathrm{Perf}_k\) be the \'etale site of perfect
\(k\)-algebras. For \(S\in\mathrm{Perf}_k\), put
\(X_S=X\times_k\operatorname{Spec}S\) and define
\(\mathbf H^3_{\mathrm{syn}}(X,\mathbb Z_p(1))(S)
:=H^3_{\mathrm{syn}}(X_S,\mathbb Z_p(1))\). Illusie--Raynaud
show that the pro-sheaves \(W_\bullet\Omega_X^1\) and
\(W_\bullet\Omega^1_{X,\log}\) have degree-two pro-algebraic cohomology
functors on \(\mathrm{Perf}_k\), denoted by
\(\mathbf H^2(X,W_\bullet\Omega_X^1)\) and
\(\mathbf H^2(X,W_\bullet\Omega^1_{X,\log})\). Their \(k\)-points are the
inverse limits of finite-level cohomology. For example,
\(\mathbf H^2(X,W_\bullet\Omega_X^1)(k)
=\varprojlim_n H^2(X,W_n\Omega_X^1)\), and likewise for logarithmic forms.

These descriptions are related by natural isomorphisms
\[
  \mathbf H^3_{\mathrm{syn}}(X,\mathbb Z_p(1))
  \simeq
  \mathbf H^2(X,W_\bullet\Omega^1_{X,\log})
  \simeq
  \mathbf H^2(X,W_\bullet\Omega_X^1)^{F=1}.
\]
The first isomorphism follows from the Bhatt--Morrow--Scholze comparison
\(R\lambda_*\mathbb Z_p(1)\simeq W\Omega^1_{\log}[-1]\), where \(\lambda\)
maps the quasisyntomic site of \(X\) to its pro-\'etale site
\cite[Theorem~1.15(1) and Corollary~8.21]{BMS19}. For the second, Illusie--Raynaud
identify \(W_\bullet\Omega^1_{X,\log}\) with the kernel of \(1-F\) on
\(W_\bullet\Omega_X^1\) and prove that \(1-F\) is surjective on degree-one
cohomology. The associated long exact sequence therefore identifies the
logarithmic degree-two term with the \(F\)-fixed part
\cite[IV, formulas~(3.1.2) and (3.2.3)--(3.2.4), and Theorem~3.3(a)]{IR83}.

Ekedahl gives a derived version of this same Illusie--Raynaud construction on
coherent Raynaud complexes and proves that it has amplitude \([0,0]\). Hence
it preserves short exact sequences of coherent \(R\)-modules
\cite[Lemma~III.1.3(i)]{ekedahl2}. For \(S\in\mathrm{Perf}_k\), let \(R_S\)
be obtained from the presentation of \(R\) in \S\ref{sec:raynaud-ring}
by replacing \(W(k)\) with \(W(S)\) and \(\sigma\) with its Witt Frobenius.
The map \(k\to S\) induces \(R\to R_S\). For a domino \(U\), set
\(G^{\mathrm{perf}}(U)(S):=((R_S\otimes_RU)^1)^{F=1}\). This is represented
by a connected perfect unipotent group of dimension \(\dim(U)\)
\cite[Lemma~IV.3.8(b)]{IR83}. This is the \emph{syntomic realization} of
\(U\). For \(U=U_X^{0,2}\), it is the domino contribution to the functor
above.

\begin{proposition}[Elementary syntomic realizations
  {\cite[Lemma~IV.3.7]{IR83}}]
\label{prop:elementary-perfect-realization}
For each \(j\in\mathbb Z\), the functor on perfect \(k\)-algebras
\(S\mapsto\bigl((R_S\otimes_R U_j)^1\bigr)^{F=1}\) is represented by
\(\mathbb G_a^{\mathrm{perf}}\). Thus the syntomic realization of every
elementary domino is canonically \(\mathbb G_a^{\mathrm{perf}}\).
\end{proposition}

\begin{proof}
Let \(S\) be a perfect \(k\)-algebra with Frobenius \(\sigma_S\). Since \(U_j\)
is killed by \(p\), extension of scalars gives
\((R_S\otimes_R U_j)^1=\prod_{n\ge j}S\,dV^n\), and \(F\) sends
\(b\,dV^{n+1}\) to \(\sigma_S(b)\,dV^n\). Hence
\(x=\sum_{n\ge j}b_n\,dV^n\) is fixed by \(F\) precisely when
\(b_n=\sigma_S(b_{n+1})\) for every \(n\ge j\). Since \(S\) is perfect, the
recursion has the unique solution \(b_{j+m}=\sigma_S^{-m}(a)\), where
\(a:=b_j\). Hence
\[
  a\longmapsto a\,dV^j+\sigma_S^{-1}(a)dV^{j+1}+\sigma_S^{-2}(a)dV^{j+2}+\cdots
\]
is a functorial isomorphism from \(S\) to
\(\bigl((R_S\otimes_R U_j)^1\bigr)^{F=1}\), which proves the claim.
\end{proof}

\subsection{The isogeny partition}

For a domino \(U\), its image \(p^rU\) is coherent, profinite, and concentrated
in degrees \(0\) and \(1\). Since \(p\) commutes with the Raynaud operators,
\(V^{-\infty}Z^0(p^rU)=0\) and
\(F^\infty B^1(p^rU)=p^rF^\infty B^1(U)=p^rU^1\). Hence \(p^rU\) is again a
domino by Remark~\ref{rem:domino-intrinsic-characterization}. We may therefore
set \(d_r(U):=\dim(p^rU)=\dim_k(p^rU^0/Vp^rU^0)
=\dim_k((p^rU^1)[F])\). Moreover, \(d_r(U)=0\) for \(r\gg0\).

\begin{theorem}[Isogeny partition]\label{thm:isogeny-partition}
Let \(U\) be a domino. Write the isogeny decompositions
\[
  \widehat G(U)
  \sim_{\mathrm{isog}}
  \prod_{i=1}^a\widehat W_{\lambda_i},
  \qquad
  G^{\mathrm{perf}}(U)
  \sim_{\mathrm{isog}}
  \prod_{j=1}^bW_{\mu_j}^{\mathrm{perf}},
\]
with partitions \(\lambda=(\lambda_1\ge\cdots\ge\lambda_a)\) and
\(\mu=(\mu_1\ge\cdots\ge\mu_b)\). Then \(\lambda=\mu\). More precisely, if this
partition is denoted by \(\lambda(U)\), then
\begin{equation}\label{eq:isogeny-conjugate-partition}
  \lambda(U)^\vee_r=d_{r-1}(U)-d_r(U),
  \qquad r\ge1.
\end{equation}
If \(U\neq0\), then \(\lambda_1(U)=p\text{-}\!\exp(U)\).
\end{theorem}

\begin{proof}
Cartier theory and its homomorphism normal form identify
\(\operatorname{im}[p]^r_{\widehat G(U)}\) with the formal group attached to
\(p^rU^0\) \cite[Theorems~3.28, 5.1--5.2]{Zin84}. Its dimension is therefore
\(d_r(U)\). On the syntomic side, exactness of the \(F\)-fixed realization and
Proposition~\ref{prop:elementary-perfect-realization} similarly give
\(\operatorname{im}[p]^r_{G^{\mathrm{perf}}(U)}
\simeq G^{\mathrm{perf}}(p^rU)\), again of dimension \(d_r(U)\).

Zink and Serre decompose the formal and perfect groups, respectively, into
Witt groups up to isogeny
\cite[Theorem~5.38]{Zin84} and
\cite[Chapter~VII, Section~2, no.~10, Theorem~1]{Serre88}.
A Witt block of length \(n\) contributes \(\max\{n-r,0\}\) to the dimension
of the image of \([p]^r\). Hence
\(d_r(U)=\sum_i\max\{\lambda_i-r,0\}=\sum_j\max\{\mu_j-r,0\}\) for all \(r\ge0\).
Taking first differences gives
\(d_{r-1}(U)-d_r(U)=\#\{i:\lambda_i\ge r\}=\#\{j:\mu_j\ge r\}\), which is
\eqref{eq:isogeny-conjugate-partition}. In particular, \(\lambda=\mu\). Finally, a nonzero domino
has positive dimension, so \(d_r(U)=0\) iff \(p^rU=0\), and the least such
\(r\) is both \(\lambda_1\) and \(p\text{-}\!\exp(U)\).
\end{proof}

\subsection{Explicit realizations}

We now compute the two realizations for the principal families of Section~3.
For a two-dimensional domino they encode complementary features of the
extension polynomial. For a distinguished domino they are Witt groups of the
same length.

For \(n\ge0\), let \(\mathcal F_n\) be the two-dimensional unipotent formal
group whose covariant Cartier module is
\((k_\sigma[[V]]e_0\oplus k_\sigma[[V]]e_1)/
(pe_1,\;pe_0-V^{n+1}e_1)\), and put
\(\mathcal F_\infty:=\widehat{\mathbb G}_a^{\,2}\).
Under the standard passage from Manin's Dieudonn\'e convention to covariant
Cartier modules, the finite \(\mathcal F_n\) are his modules \(M(n)\), which are
pairwise non-isomorphic \cite[Chapter~III, Section~8, Theorem~3.13]{Manin63}. In this
normalization \(\mathcal F_0\simeq\widehat W_2\). Moreover,
\(\mathcal F_\infty\) is killed by \(p\), whereas no finite \(\mathcal F_n\) is.

On the syntomic side, put \(\mathbb G:=\mathbb G_a^{\mathrm{perf}}\). Serre
identifies \(\operatorname{Ext}^1(\mathbb G,\mathbb G)\) with
\(\operatorname{End}(\mathbb G)\simeq k_\sigma[V^{\pm1}]\), whose units are
\(k^\times V^{\mathbb Z}\). Write \(G_h\) for the extension corresponding to
\(h\in k_\sigma[V^{\pm1}]\). In Serre's connecting-map normalization,
\(W_2^{\mathrm{perf}}\) corresponds to \(h=1\). If \(P_E\) denotes the
operator defined for a framed extension by
\(p\widetilde a=\iota(P_E(a))\), where \(\iota\) identifies the kernel with
\(\mathbb G\), then \(p=VF\) on \(W_2\) gives
\(P_E=hV^{-1}\). Thus the Serre parameter is \(P_EV\). Two such extensions
are isomorphic exactly when their parameters lie in the same two-sided
\(k^\times V^{\mathbb Z}\)-orbit, and
\(G_h\simeq W_2^{\mathrm{perf}}\) exactly when \(h\) is a unit
\cite[Section~8.5, Lemma~3]{Ser60}.

The following theorem computes both realizations from the extension
polynomial. The formal side records its \(V\)-adic order, whereas the syntomic
side records the full polynomial up to multiplication on both sides by units.

\begin{theorem}[Two-dimensional unipotent realizations]
\label{thm:2d-unipotent-realizations}
Let \(U_{j_1,j_2;f}\) be the two-dimensional domino associated with
\(f(V)=\sum_{q=0}^{j-2}c_qV^q\in k_\sigma[V]_{\le j-2}\), where
\(j:=j_2-j_1\ge2\), and abbreviate it to \(U_f\). Set
\(\operatorname{ord}_V(0)=\infty\) and put
\(f_{\mathrm{syn}}:=\sum_{q=0}^{j-2}
\sigma^{-(j-q)}(c_q)V^{j-q-1}
\in k_\sigma[V^{\pm1}]\). Then
\(\widehat G(U_f)\simeq\mathcal F_{\operatorname{ord}_V(f)}\) and
\(G^{\mathrm{perf}}(U_f)\simeq G_{f_{\mathrm{syn}}}\). In particular,
\(\widehat G(U_f)\simeq\widehat W_2\) if and only if
\(\operatorname{ord}_V(f)=0\), while
\(G^{\mathrm{perf}}(U_f)\simeq W_2^{\mathrm{perf}}\) if and only if \(f\) is
a nonzero monomial. Both realization isomorphism classes depend only on \(f\),
not on the type pair \((j_1,j_2)\).
\end{theorem}

\begin{proof}
By a Nygaard modification it suffices to treat the normalized domino
\(U_{0,j;f}\) of type \((0,j)\) defined in
\eqref{eq:polynomial-Yoneda-extension}. The
modification changes the realizations only by invertible Frobenius twists.
Let \(e_j\) generate \(U_j^0\), and let \(\widetilde e_0\) lift the generator
of \(U_0^0\). The degree-zero part of the Yoneda presentation
\eqref{eq:polynomial-Yoneda-extension} gives the covariant Cartier module
\[
  \frac{k_\sigma[[V]]\widetilde e_0\oplus k_\sigma[[V]]e_j}
       {(pe_j,\;p\widetilde e_0-Vf(V)e_j)}.
\]
If \(f=0\), this is the Cartier module of \(\mathcal F_\infty\). Otherwise write
\(f=V^ng\) with \(n=\operatorname{ord}_V(f)\) and
\(g\in k_\sigma[[V]]^\times\). Replacing \(e_j\) by \(ge_j\) gives
the Cartier module of \(\mathcal F_n\), proving the formal assertion.

For the syntomic side, Proposition~\ref{prop:elementary-perfect-realization}
identifies the kernel and quotient with \(\mathbb G\). Over a perfect
\(k\)-algebra \(S\) with Witt Frobenius \(\sigma_S\), use the fixed vectors
\[
  \widetilde\psi_0(a)
  :=
  \sum_{m\ge0}[\sigma_S^{-m}(a)]dV^m\widetilde e_0,
  \qquad
  \psi_j(a)
  :=
  \sum_{r\ge0}[\sigma_S^{-r}(a)]dV^{j+r}e_j.
\]
Here \(\widetilde\psi_0(a)\) lifts \(a\in\mathbb G(S)\). The relations
\(p\widetilde e_0=V[f](V)e_j\) and \(dV^se_j=0\) for \(s<j\) give
\[
  p\widetilde\psi_0(a)
  =
  \psi_j\!\left(
    \sum_{q=0}^{j-2}
      \sigma_S^{-(j-q)}(c_q)\,
      \sigma_S^{-(j-q-1)}(a)
  \right)
  =
  \psi_j\bigl(f_{\mathrm{syn}}(a)\bigr).
\]
Thus the \(p\)-map of the framed extension is \(f_{\mathrm{syn}}\), so its
Serre parameter is \(f_{\mathrm{syn}}V\). Since this lies in the same two-sided
unit orbit as \(f_{\mathrm{syn}}\), we obtain
\(G^{\mathrm{perf}}(U_f)\simeq G_{f_{\mathrm{syn}}}\). The element
\(f_{\mathrm{syn}}\) is a unit exactly when \(f\) is a nonzero monomial. Under
a frame change it is multiplied on the two sides by units of
\(k_\sigma[V^{\pm1}]\), so the result depends only on the unframed domino.
Increasing \(j\) by one replaces
\(f_{\mathrm{syn}}\) by \(Vf_{\mathrm{syn}}\), while the formal realization
is unchanged. This proves the final assertion.
\end{proof}

\begin{example}
\label{ex:formal-syntomic-complementarity-03}
For type \((0,3)\), Theorem~\ref{thm:2d-unipotent-realizations} shows that
the extensions represented by \(1\) and \(1-V\) have the same formal realization
\(\widehat W_2\) but distinct syntomic realizations \(W_2^{\mathrm{perf}}\)
and \(G_{V^2-V}\). By contrast, those represented by \(1\) and \(V\) have the same
syntomic realization \(W_2^{\mathrm{perf}}\), but their formal realizations
are \(\mathcal F_0\simeq\widehat W_2\) and \(\mathcal F_1\). Thus neither
realization alone determines the domino.
\end{example}

The distinguished family is more rigid: both realizations are Witt groups of
the same length.

\begin{corollary}
\label{cor:distinguished-props}
For \(J=(j,j+2,\ldots,j+2n-2)\), the two realizations of the distinguished
domino \(U_J\) are \(\widehat G(U_J)\simeq\widehat W_n\) and
\(G^{\mathrm{perf}}(U_J)\simeq W_n^{\mathrm{perf}}\).
\end{corollary}

\begin{proof}
By a Nygaard modification it suffices to treat \(J=(0,2,\ldots,2n-2)\).
Put \(U^{(n)}=U_{0,2,\ldots,2n-2}=\widehat R/\widehat R(F^n,Fd)\), and
let \(x\) be the image of \(1\). The degree-zero normal form in the proof of
Theorem~\ref{thm:distinguished} identifies \((U^{(n)})^0\) with the Cartier
module of \(\widehat W_n\).

For the syntomic realization, let \(S\) be a perfect \(k\)-algebra, set
\(\eta_m=dV^m x\) and \(r_m=\min\{n,m+1\}\). The degree-one normal form gives
\((R_S\otimes_RU^{(n)})^1\simeq\prod_{m\ge0}W_{r_m}(S)\eta_m\), with
\(F\eta_0=0\) and \(F(a\eta_{m+1})=\sigma_S(a)\eta_m\). Thus an \(F\)-fixed
element is determined by its \(\eta_{n-1}\)-coefficient \(w\in W_n(S)\). The
inverse sends \(w\) to
\(\sum_{m\ge0}(\sigma_S^{\,n-1-m}(w)\bmod p^{r_m})\eta_m\). Negative powers
of \(\sigma_S\) are defined because \(S\) is perfect. This gives
\(G^{\mathrm{perf}}(U^{(n)})\simeq W_n^{\mathrm{perf}}\). Finally, Nygaard
modification changes the identification only by an invertible Frobenius
twist, so the result holds for every \(J\).
\end{proof}

\subsection{Brauer groups}

The common input for the prime-to-\(p\) and \(p\)-primary parts of the Brauer
group is the Kummer sequence. For every prime \(\ell\) and \(n\ge1\), the
sequence \(1\to\mu_{\ell^n}\to\mathbb G_m\xrightarrow{\ell^n}\mathbb
G_m\to1\) is exact in the fppf topology and gives
\[
  0\longrightarrow \operatorname{Pic}(X)/\ell^n
  \longrightarrow H^2_{\mathrm{fppf}}(X,\mu_{\ell^n})
  \longrightarrow \operatorname{Br}(X)[\ell^n]
  \longrightarrow0.
\]
If \(\ell\ne p\), this is already an \(\acute{e}\)tale sequence. Passing to
\(\ell\)-adic cohomology gives
\[
  0\longrightarrow(\mathbb Q_\ell/\mathbb Z_\ell)^{b_2-\rho}
  \longrightarrow\operatorname{Br}(X)[\ell^\infty]
  \longrightarrow
  H^3_{\mathrm{\acute et}}(X,\mathbb Z_\ell(1))_{\mathrm{tors}}
  \longrightarrow0;
\]
Here \(b_2:=\dim_{\mathbb Q_\ell}H^2_{\mathrm{\acute et}}
(X,\mathbb Q_\ell)\) and \(\rho:=\operatorname{rank}\operatorname{NS}(X)\).
See \cite[\S1, equations~(1)--(2) and the following remark]{SZ12}.

For \(\ell=p\), one instead works with the fppf Kummer system and forms the weight-one
syntomic complex
\(R\Gamma(X,\mathbb Z_p(1)):=R\varprojlim_n
R\Gamma_{\mathrm{fppf}}(X,\mu_{p^n})\). In the appendix to
Skorobogatov--Petrov's work on the \(p\)-primary Brauer group, Petrov combines
the resulting exact sequences with structure results of Illusie and
Illusie--Raynaud to obtain
\[
  \mathrm{Br}(X)[p^\infty]
  \cong
  (\mathbb Q_p/\mathbb Z_p)^{r_2-\rho}
  \oplus H^3_{\mathrm{syn}}(X,\mathbb Z_p(1))[p^\infty],
\]
where \(r_2\) is the slope-\(1\) rank. The
second summand has finite \(p\)-exponent
\cite[Appendix~A, Theorem~A.1 and equations~(7)--(10)]{SP25}. The \(F\)-fixed construction of
Subsection~\ref{sec:formal-syntomic-realizations} represents this second
summand on the perfect site by a perfect quasi-algebraic group, denoted
\(\mathrm{Br}^{\mathrm{perf}}_X\). Thus
\(\mathrm{Br}^{\mathrm{perf}}_X(k)
=H^3_{\mathrm{syn}}(X,\mathbb Z_p(1))[p^\infty]\), and its
connected unipotent part is \(G^{\mathrm{perf}}(U_X^{0,2})\). The next lemma
identifies the formal and perfect Brauer groups in the supersingular abelian
case.

\begin{lemma}[Brauer groups of a supersingular abelian variety]
\label{lem:supersingular-brauer-unipotent}
Let \(A/k\) be a supersingular abelian \(g\)-fold. Then \(\mathrm{Br}(A)\) is
\(p\)-primary, and the formal and perfect Brauer groups are connected and
unipotent. They are naturally isomorphic to the two realizations of the
bidegree-\((0,2)\) domino:
\[
  \widehat{\mathrm{Br}}_A\simeq\widehat G(U_A^{0,2}),
  \qquad
  \mathrm{Br}^{\mathrm{perf}}_A\simeq G^{\mathrm{perf}}(U_A^{0,2}).
\]
If \(g\ge2\), the group
\(\mathrm{Br}(A)=\mathrm{Br}^{\mathrm{perf}}_A(k)\) has \(p\)-exponent
\(p\text{-}\!\exp(U_A^{0,2})\).
\end{lemma}

\begin{proof}
For supersingular \(A\), one has
\(r_2=b_2=\rho(A)=g(2g-1)\), by the pure slope \(1\) of \(H^2_{\mathrm{crys}}\)
and the Rosati description of \(\operatorname{NS}(A)_{\mathbb Q}\). Thus the
divisible terms in both Kummer descriptions above vanish. For \(\ell\ne p\),
the group \(H^3_{\mathrm{\acute et}}(A,\mathbb Z_\ell(1))\) is torsion-free,
since \(H^3_{\mathrm{\acute et}}(A,\mathbb Z_\ell)
\simeq\bigwedge^3H^1_{\mathrm{\acute et}}(A,\mathbb Z_\ell)\). Hence
\(\mathrm{Br}(A)[\ell^\infty]=0\).
By
\cite[Theorem~IV.3.3(b)]{IR83}, the component group of
\(\mathrm{Br}^{\mathrm{perf}}_A\) is the \(p\)-primary torsion in the
\(F\)-fixed degree-one core. Since \(H^3_{\mathrm{crys}}(A/W)\) is
torsion-free, survival of the core and the abutment slope bounds make this
group zero \cite[Theorem~II.3.4]{IR83}. See also
\cite[Theorem~1.5]{Yang26}. Hence \(\mathrm{Br}^{\mathrm{perf}}_A\) is connected
unipotent and equals \(G^{\mathrm{perf}}(U_A^{0,2})\).

On the formal side, \(H^2(A,W\mathcal O_A)[1/p]=0\), so its finite free core
vanishes. Thus \((U_A^{0,2})^0=H^2(A,W\mathcal O_A)\), and Cartier theory
gives \(\widehat{\mathrm{Br}}_A\simeq\widehat G(U_A^{0,2})\). Finally, the
syntomic identification in the proof of Theorem~\ref{thm:isogeny-partition}
gives
\(p^r\mathrm{Br}(A)\simeq G^{\mathrm{perf}}(p^rU_A^{0,2})(k)\), which
vanishes exactly when \(p^rU_A^{0,2}=0\).
\end{proof}

There is a complementary group-theoretic construction. For
\(\pi:X\to\operatorname{Spec}k\), Bragg--Olsson prove that each fppf sheaf
\(R^2\pi_*\mu_{p^n}\) is representable \cite[Corollary~1.6]{BO21}. Let
\(\mathcal U_{X,n}\) be its reduced identity component. The transition maps
are injective and these groups stabilize for large \(n\). Following
\cite[Section~3.2]{GSY25}, write \(\mathcal U_X\) for the stable group. Its formal
completion at the identity is the unipotent part of the Artin--Mazur formal
Brauer group, while its perfection is \emph{isogenous} to the connected unipotent
part of the \(p\)-primary Brauer group
\cite[Lemma~3.5 and the following paragraph]{GSY25}. For a
Mazur--Ogus variety these are precisely the groups realized by
\(\widehat G(U_X^{0,2})\) and \(G^{\mathrm{perf}}(U_X^{0,2})\), respectively.
On the syntomic side, the domino description recovers both the isogeny class
and the group \(G^{\mathrm{perf}}(U_X^{0,2})\) itself.
At the same time, the coherent \(R\)-module \(U_X^{0,2}\) records integral
extension data invisible in the common isogeny partition of the two
realizations.

\addtocontents{toc}{\protect\setcounter{tocdepth}{2}}
\section{Ekedahl's three \texorpdfstring{\(t\)}{t}-structures and their hearts}
\label{app:ekedahl-three-hearts}

This appendix first defines Ekedahl's three \(t\)-structures on \(D_c^b(R)\)
(Postnikov, diagonal, and \(F\)-gauge) and recalls his derived equivalence between
coherent $R$-complexes and coherent \(F\)-gauge complexes. It then proves the
three-heart criterion
\(\mathcal E=\operatorname{Mod}_c(R)\cap\Delta\cap\mathcal G\) and transports
Ekedahl modules across the equivalence, where they become coherent \(F\)-gauges
with Hodge--Tate weights in \([0,2]\), one step beyond the weights \([0,1]\) of
classical Dieudonn\'e theory.

The Postnikov \(t\)-structure on \(D_c^b(R)\) has heart
\(\operatorname{Mod}_c(R)\), the category of coherent $R$-modules
(\S\ref{sec:raynaud-ring}). The diagonal \(t\)-structure, with heart \(\Delta\),
was recalled in Definition~\ref{def:diagonal-t-structure}: Ekedahl proves that
the filtration \(\{\widetilde\tau_{\leq i}\}_{i\in\mathbb Z}\) is \emph{radical}
\cite[Theorem~I.1.1]{ekedahl3}, and his general
\emph{radical-filtration theorem}
\cite[Theorem~0.1.4]{ekedahl3} then produces the diagonal
\(t\)-structure \cite[Definition~I.1.2]{ekedahl3} and shows that it
commutes with the Postnikov one (Remark~\ref{rem:ordinary-diagonal-commute}).
The intersection of these first two hearts is described concretely in
\cite[Section~I.1]{ekedahl3}: it consists of the coherent \(R\)-modules
\(M\) satisfying
\begin{equation}\label{eq:appendix-diagonal-criterion}
        M^i=0\quad(i\notin\{0,1\}),
        \qquad
        F^\infty B^1(M)=M^1 .
\end{equation}

The third \(t\)-structure is defined intrinsically on the Raynaud side: following
\cite[Definition~I.3.1]{ekedahl3}, call \(M\in\Delta\)
\emph{\(\mathbf s\)-\(1\)-torsion} if \(H^*(\mathbf s(M))\) is a torsion \(W\)-module and
\(H^0(\mathbf s(M))=0\), where \(\mathbf s\) is the simple functor of
Remark~\ref{rem:simple-functor}. By \cite[Theorem~I.4.5]{ekedahl3}, every
\(M\in\Delta\) admits a largest
\(\mathbf s\)-\(1\)-torsion quotient. Let \(G_2(M)\subseteq M\) be the kernel of the
projection onto it. Then \(G_2\) is an idempotent radical on \(\Delta\), and the
\(t\)-structure associated to it by the radical-filtration theorem is the
\emph{\(F\)-gauge \(t\)-structure}
\[
  (G^{\leq 0},G^{\geq 0}),
  \qquad
  \mathcal G:=G^{\leq 0}\cap G^{\geq 0},
\]
of \cite[Definition~II.1.1]{ekedahl3}. In particular, for
\(M\in\Delta\),
\[
  M\in\mathcal G
  \quad\Longleftrightarrow\quad
  G_2(M)=M
  \quad\Longleftrightarrow\quad
  M\ \text{has no nonzero \(\mathbf s\)-\(1\)-torsion quotient.}
\]
An amplitude criterion in terms of \(\mathbf s(M)\) and
\(R_1\otimes^{\mathbf L}_RM\) is given in
\cite[Proposition~II.1.2]{ekedahl3}.

Since the \(F\)-gauge \(t\)-structure arises from a radical filtration on the
diagonal heart, it commutes with the diagonal \(t\)-structure. It does
\emph{not} commute with the Postnikov \(t\)-structure in general
\cite[Chapter~VII, Section~3]{ekedahl3}, so commutation of \(t\)-structures
is not a transitive relation. Belonging to the triple intersection
\(\operatorname{Mod}_c(R)\cap\Delta\cap\mathcal G\) therefore imposes genuine
simultaneous conditions from all three \(t\)-structures.

The name of this third \(t\)-structure comes from a second, equivalent
description of \(D_c^b(R)\) in terms of \(F\)-gauges.

An \emph{\(F\)-gauge} is a graded module
\(N=\bigoplus_{r\in\mathbb Z}N^r\) over \(W[u,t]/(ut-p)\), i.e.\ a graded
\(W\)-module
\[
\begin{tikzcd}[column sep=2.7em]
\cdots
  \arrow[r, shift left=.6ex, "t"]
& N^{i+1}
  \arrow[l, shift left=.6ex, "u"]
  \arrow[r, shift left=.6ex, "t"]
& N^{i}
  \arrow[l, shift left=.6ex, "u"]
  \arrow[r, shift left=.6ex, "t"]
& N^{i-1}
  \arrow[l, shift left=.6ex, "u"]
  \arrow[r, shift left=.6ex, "t"]
& \cdots
  \arrow[l, shift left=.6ex, "u"]
\end{tikzcd}
\]
with \(W\)-linear \(u\) raising and \(t\) lowering the degree and
\(ut=tu=p\), together with a \(\sigma\)-linear identification
\(\tau:N^\infty\xrightarrow{\ \sim\ }N^{-\infty}\) between the colimits along
\(u\) and along \(t\). We say that its \emph{Hodge--Tate weights} lie in
\([m,n]\) if \(t\) is an isomorphism in degrees \(<m\) and \(u\) is an isomorphism
in degrees \(>n\).\footnote{Ekedahl calls this condition \emph{level} \([m,n]\)
\cite[Definition~II.2.1]{ekedahl3}. Under the identification with prismatic
\(F\)-gauges recalled in Remark~\ref{rem:prismatic-fgauge} below, it agrees with
the Hodge--Tate weights of \cite[Remark~5.3.14]{BhattFgauges}.} Only the ranges
\([0,1]\) and \([0,2]\) occur here.

\begin{definition}[Ekedahl's functor \(S\)]
For an \(R\)-module \(M\), with \(\sigma_*\) denoting the Frobenius twist, the
adjacent components of \(S(M)\) and their transition maps are
\cite[Definition~II.3.1]{ekedahl3}:
\begin{equation}
\label{eq:ekedahl-S-component}
\begin{tikzcd}[column sep=1.8em, row sep=2.5em]
S(M)^{r+1}={}
  & \cdots
      \arrow[r, "\sigma_*d"]
      \arrow[d, shift left=.6ex, "p"]
  & \sigma_*M^{r-1}
      \arrow[r, "\sigma_*d"]
      \arrow[d, shift left=.6ex, "p"]
  & \sigma_*M^r
      \arrow[r, "dV"]
      \arrow[d, shift left=.6ex, "V"]
  & M^{r+1}
      \arrow[r, "d"]
      \arrow[d, shift left=.6ex, "1"]
  & \cdots
      \arrow[d, shift left=.6ex, "1"]
\\
S(M)^r={}
  & \cdots
      \arrow[r, "\sigma_*d"]
      \arrow[u, shift left=.6ex, "1"]
  & \sigma_*M^{r-1}
      \arrow[r, "dV"]
      \arrow[u, shift left=.6ex, "1"]
  & M^r
      \arrow[r, "d"]
      \arrow[u, shift left=.6ex, "F"]
  & M^{r+1}
      \arrow[r, "d"]
      \arrow[u, shift left=.6ex, "p"]
  & \cdots
      \arrow[u, shift left=.6ex, "p"]
\end{tikzcd}
\end{equation}
The downward arrows are \(t\), the upward arrows are \(u\), and each vertical
pair composes to \(p\). At the two ends,
\(S(M)^{-\infty}=\mathbf s(M)\) and
\(S(M)^{\infty}=\sigma_*\mathbf s(M)\), with \(\tau\) the identity of the
underlying complex.
\end{definition}

A complex of \(F\)-gauges is \emph{coherent} if it is complete with
mod-\(p\) reduction of finite-dimensional cohomology
\cite[Definition~II.5.1]{ekedahl3}. In particular, coherence already
builds in \(p\)-completeness, so no separate completion hypothesis is needed.
Every bounded coherent $R$-complex has bounded internal degrees, and
\eqref{eq:ekedahl-S-component} is independent of the chosen weight range. Hence
the equivalences of \cite[Theorem~II.5.3]{ekedahl3}, one for each interval
\(I\), are
compatible under enlarging \(I\) and assemble into a single equivalence
\[
  S:
  D_c^b(R)
  \xrightarrow{\ \sim\ }
  D_c^b(F\text{-gauge}),
\]
between bounded coherent $R$-complexes and bounded coherent \(F\)-gauge
complexes, with quasi-inverse \(L\widehat t\)
\cite[Proposition~II.4.7(ii)]{ekedahl3}. We will not need the
formula for \(L\widehat t\). Moreover, \(S\) carries the \(F\)-gauge
\(t\)-structure to the componentwise Postnikov \(t\)-structure on \(F\)-gauge
complexes \cite[Proposition~II.5.4]{ekedahl3}. Concretely, for
\(M\in\operatorname{Mod}_c(R)\),
\[
  M\in\mathcal G
  \quad\Longleftrightarrow\quad
  S(M)\ \text{is a coherent \(F\)-gauge placed in degree \(0\).}
\]

\begin{remark}[Interpretation via the stacky approach]
\label{rem:prismatic-fgauge}
Bhatt proves that the syntomification \(k^{\mathrm{Syn}}\) is a noetherian
regular \(p\)-adic formal stack and identifies prismatic \(F\)-gauges over
\(k\) with \(D_{\mathrm{qc}}(k^{\mathrm{Syn}})\)
\cite[Remark~4.1.5 and Definition~4.2.1]{BhattFgauges}. Thus its bounded coherent
subcategory is well-defined, and the explicit \(F\)-gauge description
\cite[Remarks~4.2.4 and 4.2.8]{BhattFgauges}, together with Ekedahl's equivalence,
gives
\[
  D_c^b(R)
  \simeq D_c^b(F\text{-gauge})
  \simeq D_c^b(k^{\mathrm{Syn}}).
\]
\end{remark}

\begin{remark}[The \(F\)-gauge of a two-term \(R\)-module]
For \(M=[\,M^0\xrightarrow dM^1\,]\in\operatorname{Mod}_c(R)\),
equation~\eqref{eq:ekedahl-S-component} specializes to the following diagram:
\begin{equation}\label{eq:two-term-gauge-components}
\begin{tikzcd}[column sep=3.3em, row sep=2.2em]
{} & S(M)^2 & S(M)^1 & S(M)^0 & {}
\\[-1.3em]
\cdots
  \arrow[r, shift left=.7ex, "p"]
& \sigma_*M^0
  \arrow[l, shift left=.7ex, "1"]
  \arrow[r, shift left=.7ex, "p"]
  \arrow[d, "\sigma_*d"]
& \sigma_*M^0
  \arrow[l, shift left=.7ex, "1"]
  \arrow[r, shift left=.7ex, "V"]
  \arrow[d, "dV"]
& M^0
  \arrow[l, shift left=.7ex, "F"]
  \arrow[r, shift left=.7ex, "1"]
  \arrow[d, "d"]
& \cdots
  \arrow[l, shift left=.7ex, "p"]
\\
\cdots
  \arrow[r, shift left=.7ex, "p"]
& \sigma_*M^1
  \arrow[l, shift left=.7ex, "1"]
  \arrow[r, shift left=.7ex, "V"]
& M^1
  \arrow[l, shift left=.7ex, "F"]
  \arrow[r, shift left=.7ex, "1"]
& M^1
  \arrow[l, shift left=.7ex, "p"]
  \arrow[r, shift left=.7ex, "1"]
& \cdots
  \arrow[l, shift left=.7ex, "p"]
\end{tikzcd}
\end{equation}
The two rows lie in cohomological degrees \(0\) and \(1\). The rightward arrows
are \(t\), the leftward arrows are \(u\), and \(S(M)^r=S(M)^2\) for \(r\ge2\),
while \(S(M)^r=S(M)^0\) for \(r\le0\). The limits are
\(S(M)^\infty=\sigma_*\mathbf s(M)\) and
\(S(M)^{-\infty}=\mathbf s(M)\), with
\(\tau\) the identity on the underlying complex.

Hence \(S(M)\) is coherent with Hodge--Tate weights in \([0,2]\)
\cite[Theorem~II.5.3]{ekedahl3} and lies in degree \(0\) if and only if
\(dV\) is surjective.
\end{remark}

Combining this description of the first two hearts with the component formula
now identifies their intersection with the third.

\begin{theorem}[Three-heart criterion]
\label{thm:three-heart-criterion}
Let \(M\in\operatorname{Mod}_c(R)\). Then \(M\in\Delta\cap\mathcal G\) if and
only if \(M^i=0\) for \(i\notin\{0,1\}\) and \(dV:M^0\to M^1\) is surjective.
Consequently, the Ekedahl modules of Definition~\ref{def:ekedahl-module} are
exactly the objects of the three-heart intersection:
\[
  \mathcal E
  =
  \operatorname{Mod}_c(R)\cap\Delta\cap\mathcal G
  \subset D_c^b(R).
\]
In particular, \(\mathcal E\) is closed under extensions and direct summands,
with exact structure inherited from \(D_c^b(R)\).
\end{theorem}

\begin{proof}
By \eqref{eq:appendix-diagonal-criterion}, the first two hearts impose the
two-term condition together with \(F^\infty B^1(M)=M^1\). For a two-term
module, the \(t\)-exactness of \(S\) and
\eqref{eq:two-term-gauge-components} show that membership in \(\mathcal G\) is
equivalent to the surjectivity of \(dV\). This surjectivity also gives
\(d(M^0)=M^1\), hence \(F^\infty B^1(M)=B^1(M)=M^1\). The stated description
of \(\mathcal E\) now follows from Definition~\ref{def:ekedahl-module}. Closure
under extensions and direct summands follows from the same property for each
heart.
\end{proof}

\begin{remark}[The \(F\)-gauge of a Nygaard pair]
\label{rem:ekedahl-pair-gauge}
Let \(M\in\mathcal E^{\mathrm{nft}}\) have Nygaard pair
\((K_0,K_{-1},\iota,\nu)=\mathbf K(M)\), so that \(K_0=\ker(d)\),
\(K_{-1}=\ker(dV)\), and \(\nu=V|_{K_{-1}}\)
(Theorem~\ref{thm:two-kernel-presentation}). Taking degree-\(0\) cohomology in
\eqref{eq:two-term-gauge-components}, and suppressing Frobenius twists, gives
the \(F\)-gauge with Hodge--Tate weights in \([0,2]\)
\[
\begin{tikzcd}[column sep=3.4em]
\cdots
  \arrow[r, shift left=.6ex, "p"]
& K_0
  \arrow[l, shift left=.6ex, "\sim"]
  \arrow[r, shift left=.6ex, "p"]
& K_{-1}
  \arrow[l, shift left=.6ex, "\iota"]
  \arrow[r, shift left=.6ex, "\nu"]
& K_0
  \arrow[l, shift left=.6ex, "\nu^{-1}p"]
  \arrow[r, shift left=.6ex, "\sim"]
& \cdots
  \arrow[l, shift left=.6ex, "p"]
\end{tikzcd}
\]
with \(N^2=K_0,\ N^1=K_{-1},\ N^0=K_0\). Its limiting components are
\(N^\infty=\sigma_*K_0\) and \(N^{-\infty}=K_0\), with
\(\tau:\sigma_*K_0\xrightarrow{\sim}K_0\) the identity on the underlying
module. In \(F\)-gauge degrees outside \([0,2]\),
the outward maps are isomorphisms and the inward maps are multiplication by \(p\). The
arrow \(p:K_0\to K_{-1}\) uses
\(pK_0\subseteq\iota(K_{-1})\), while \(\nu^{-1}p\) is the Raynaud Frobenius:
\(Fx\) is the unique element of \(K_{-1}\) satisfying \(\nu(Fx)=px\). The
arrows \(\nu\) and \(\nu^{-1}p\) carry the suppressed twist. Thus \(\nu\)
encodes both the \(F\)-gauge
gluing and the Raynaud Frobenius.
\end{remark}

An \(F\)-gauge with Hodge--Tate weights in \([0,1]\) is an \(R^0\)-module
\cite[Section~II.2]{ekedahl3}. Imposing coherence recovers the classical
Dieudonn\'e modules of Definition~\ref{def:coherent}, which are finitely
generated over \(W\) with \(V\) topologically nilpotent. The functor \(S\)
sends every Ekedahl module to a coherent \(F\)-gauge with Hodge--Tate weights in
\([0,2]\).
By \cite[Theorem~I.1.12(i)]{ekedahl3} and
\eqref{eq:ekedahl-S-component}, an Ekedahl module \(M\) is nft precisely when
\(S(M)\) contains no nonzero finite-torsion sub-\(F\)-gauge with Hodge--Tate
weights in \([0,1]\) or \([1,2]\). For such \(M\), both \(\iota\) and
\(\nu=V|_{K_{-1}}\) are injective (Lemma~\ref{lem:kernel-ladder}).
Imposing the torsion condition cuts out the positive dominoes:
\[
  \operatorname{Dom}^{+}
  =\mathcal E^{\mathrm{nft}}\cap\mathcal E^{\mathrm{tor}}
  \subset\mathcal E^{\mathrm{nft}}\subset\mathcal E.
\]
Here the torsion condition means that \(C(M)\) is \(W\)-torsion. Since nft makes
\(C(M)\) finite free, it is equivalent to \(C(M)=0\). Equivalently, this is the
torsion condition on the \(F\)-gauge side: taking \(H^0\) in
\eqref{eq:two-term-gauge-components} gives
\[
  H^0(S(M)^r)\cong
  \begin{cases}
    \ker d, & r\leq0,\\
    \ker(dV), & r=1,\\
    \sigma_*\ker d, & r\geq2.
  \end{cases}
\]
Since
\(\ker d/C(M)=\ker(\bar d:\Dom(M)^0\to\Dom(M)^1)\) has finite \(W\)-length
and \(C(M)\) is finite free, \(\ker d\) is finitely generated over \(W\).
The inclusion \(\ker(dV)\subseteq\ker d\), which follows from \(FdV=d\),
then shows that every \(H^0(S(M)^r)\) is finitely generated over \(W\).
For a finitely generated \(W\)-module, being \(W\)-torsion is equivalent to
having finite length. It follows that \(C(M)\) is \(W\)-torsion if and only if
\(H^0(S(M)^r)\) has finite length over \(W\) for every \(r\in\mathbb Z\).
Consequently, positive dominoes are characterized on the \(F\)-gauge side as
the torsion objects among the \(F\)-gauges with Hodge--Tate weights in \([0,2]\)
arising from nft
Ekedahl modules. Through the stacky interpretation of
Remark~\ref{rem:prismatic-fgauge}, this identifies \(\operatorname{Dom}^+\)
with a concrete torsion subcategory of \(D_c^b(k^{\mathrm{Syn}})\). In future
work, we study a broader torsion subcategory of \(D_c^b(k^{\mathrm{Syn}})\)
from the perspective of Harder--Narasimhan filtrations.

\bibliographystyle{alpha}
\bibliography{refs}

\end{document}